\documentclass[11pt]{article}
\newcommand{\filename}{}
\usepackage{amsthm}
\usepackage{amsfonts}
\usepackage{amssymb}
\usepackage{amsmath}
\usepackage{epsfig}
\usepackage{comment}
\usepackage[]{graphicx}
\usepackage[colorlinks=true]{hyperref}
\usepackage{color}

\newcommand{\mg}{\color{magenta}}

\newcommand{\mycomment}[1]{}
\usepackage{dsfont} 
\newtheorem{lemma}{LEMMA}[section]
\newtheorem{theorem}[lemma]{THEOREM}

\newtheorem{definition}[lemma]{DEFINITION}
\newtheorem{corollary}[lemma]{COROLLARY}
\newtheorem{remark}[lemma]{REMARK}

\newcommand{\be}{\begin{equation}}
\newcommand{\ee}{\end{equation}}
\newcommand{\bes}{\begin{equation*}}
\newcommand{\ees}{\end{equation*}}
\newcommand{\bea}{\begin{eqnarray}}
\newcommand{\eea}{\end{eqnarray}}
\newcommand{\beas}{\begin{eqnarray*}}
\newcommand{\eeas}{\end{eqnarray*}}
 \newcommand{\nc}{\newcommand}
 \nc{\ha}{\frac{1}{2}}
 \nc{\tha}{\frac{3}{2}}

 \nc{\ov}{\overline}
 \nc{\pa}{\partial}
 \nc{\pO}{{\partial\T}}

\nc{\C}{{\mathbb{C}}}

\newcommand{\F}{{\cal F}}

 \nc{\N}{{\mathbb{N}}}

\newcommand{\R}{{\mathbb R}}
\newcommand{\T}{{\mathbb T}}

\newcommand{\Z}{\mathbb{Z}}

\newcommand{\sZ}{\sum_{\xi\in\Z^n}}

\newcommand{\bs}{\boldsymbol}

\renewcommand{\div}{{\rm div}}

\makeatletter \@addtoreset{equation}{section}\makeatother
 \makeatletter
\renewcommand{\@oddhead}{\vbox{\hbox to\textwidth{\scriptsize %
 \hfill S.E.Mikhailov\hfill \arabic{page}\vspace{1ex}}
 \hrule
 }}
 \renewcommand{\@evenhead}{\vbox{\hbox to\textwidth{\scriptsize %
 \filename\hfill NS-torus
 }
\hrule
 }}
 \makeatother
 \allowdisplaybreaks
 \numberwithin{equation}{section}

\begin{document}

\title
{\bf
Spatially-Periodic Solutions for Evolution Anisotropic 
Variable-Coefficient Navier-Stokes Equations: \\
II. Serrin-Type Solutions}

\author
{Sergey E. Mikhailov\footnote{
e-mail: {\sf sergey.mikhailov@brunel.ac.uk}, 
}\\
     Brunel University London,
     Department of Mathematics,\\
     Uxbridge, UB8 3PH, UK}

 \maketitle


\begin{abstract}\noindent
We consider evolution (non-stationary) space-periodic solutions to the $n$-dimensional non-linear Navier-Stokes equations of anisotropic fluids with the viscosity coefficient tensor variable in space and time and satisfying the relaxed ellipticity condition. 
Employing the Galerkin algorithm, we prove the existence of  Serrin-type solutions, that is,  weak solutions with the velocity in the periodic space $L_2(0,T;\dot{\mathbf H}^{n/2}_{\#\sigma})$, $n\ge 2$.
 The solution uniqueness and regularity results are also discussed. \\

\noindent{\bf Keywords}. Partial differential equations;  Evolution Navier-Stokes equations; Anisotropic Navier-Stokes; Spatially periodic solutions; Variable coefficients; Relaxed ellipticity condition; Serrin-type solutions.

\noindent {\bf MSC classes}:	35A1, 35B10, 35K45, 35Q30, 76D05
\end{abstract}

\section{Introduction}\label{S1}
Analysis of Stokes and Navier-Stokes equations is an established and active field of research in applied mathematical analysis; see, e.g., \cite{Ladyzhenskaya1969, Lions1969, Constantin-Foias1988,  Lemarie-Rieusset2002, Lemarie-Rieusset2016, Galdi2011, RRS2016, Seregin2015, Sohr2001, Temam1995, Temam2001} and many other publications.
These works were mainly devoted to the flows of isotropic fluids with constant viscosity coefficient, and some of the employed methods were heavily based on these properties.

On the other hand, in many cases the fluid viscosity can vary in time and spatial coordinates, e.g., due to variable ambient  temperature. 
Moreover, from the point of view of rational mechanics of continuum, fluids can be anisotropic and this feature is indeed observed in liquid crystals, electromagnetic fluids, etc., see, e.g.,  \cite{Duffy1978} and references therein.
In \cite{KMW2020, KMW-DCDS2021, KMW-LP2021, KMW-transv2022, Mikhailov2022, Mikhailov2023} the classical Navier-Stokes analysis has been extended to the transmission and boundary-value problems for {\it stationary} Stokes and Navier-Stokes equations of anisotropic fluids, particularly with relaxed ellipticity condition on the viscosity tensor. 

In Part I, \cite{Mikhailov2024}, we considered {\it evolution (non-stationary)} spatially-periodic solutions in $\R^n$, $n\ge 2$, to the  Navier-Stokes equations of an anisotropic fluid with the viscosity coefficient tensor variable in spatial coordinates and time and satisfying the relaxed ellipticity condition. 
We implemented  the Galerkin algorithm but unlike the traditional approach, for example in \cite{Temam1995, Temam2001},  where the Galerkin basis consisted of the eigenfunctions of the corresponding isotropic constant-coefficient Stokes operator, we employed the basis constituted by the eigenfunctions of the periodic Bessel-potential operator having an advantage that it is universal, i.e., independent of the analysed anisotropic variable-coefficient Navier-Stokes operator.
To analyse the solution in higher dimensions, the definition of the weak solution was generalised to some extent.
Then the periodic weak solution existence was considered in the spaces of Banach-valued functions mapping a finite time interval to  periodic Sobolev (Bessel-potential) spaces on $n$-dimensional flat torus, $L_{\infty}(0,T;\dot{\mathbf H}_{\#\sigma}^{0})\cap L_2(0,T;\dot{\mathbf H}_{\#\sigma}^{1})$. 
The periodic setting is interesting on its own, modelling fluid flow in periodic composite structures, and is also a common element of homogenisation theories for inhomogeneous fluids and in the Large Eddy Simulation. 

In this paper, Part II, we prove  the existence, uniqueness and regularity of the  weak solutions  that belongs to the  space $L_2(0,T;\dot{\mathbf H}^{n/2}_{\#\sigma})$ (we call them Serrin-type solutions).
It is well known that the regularity results available at the moment for evolution Navier-Stokes equations are rather different  for dimensions $n=2$ and $n=3$, even for isotropic constant-viscosity fluids. 
The weak solution global regularity under arbitrarily large smooth input data for $n=2$ is proved and can be found, e.g., in \cite{Ladyzhenskaya1969, Lions1969, Constantin-Foias1988, Lemarie-Rieusset2002, Galdi2011,   RRS2016, Seregin2015, Sohr2001, Temam1995, Temam2001}. However  for $n=3$ it is still an open question and constitutes one of the Clay Institute famous Millennium problems. 
Our motivation for considering arbitrary $n\ge 2$ is particularly to place the cases $n=2$ and $n=3$ in a more general set and to see which of them is an exception.

The paper material is presented as follows. In Section \ref{S1.1} we provide essentials on anisotropic Stokes and Navier-Stokes equations.
Section \ref{S2} gives an introduction to the periodic Sobolev (Bessel-potential) functions spaces in spatial coordinates on $n$-dimensional flat torus and to the corresponding Banach-valued functions mapping a finite time interval to these periodic Sobolev spaces.
In Section \ref{S3} we describe the existence results for evolution spatially-periodic anisotropic Navier-Stokes problem available from Part I,  \cite{Mikhailov2024}.
Sections \ref{S3-Serrin-intro}, \ref{S5ERC} and \ref{S5ERV} contain the main results of the paper.
In Section \ref{S3-Serrin-intro} we define the Serrin-type solutions and prove the energy equality for them and also their uniqueness, for the $n$-dimensional periodic setting, $n\ge2$. We also remark on their relations with the strong solutions and show that for $n=2$, any weak solution is a Serrin-type solution.
In Section \ref{S5ERC} we analyse the Serrin-type solution existence and regularity for constant anisotropic viscosity coefficients,
while in Section \ref{S5ERV} we generalise these results to variable anisotropic viscosity coefficients.
In Section \ref{Appendix} we collect some technical results used in the main text of the paper, several of which might be new.

\subsection{Anisotropic Stokes and Navier-Stokes PDE systems}\label{S1.1}

Let $n\ge 2$ be an integer, $\mathbf x\in\mathbb R^n$ denote the space coordinate vector, and $t\in\R$ be time.
Let
$\boldsymbol{\mathfrak L}$ denote the second-order differential operator represented in the component-wise divergence form as
\begin{align}
\label{L-oper}
&(\boldsymbol{\mathfrak L}{\mathbf u})_k:=
\partial _\alpha\big(a_{kj}^{\alpha \beta }E_{j\beta }({\mathbf u})\big),\ \ k=1,\ldots ,n,
\end{align}
where ${\mathbf u}\!=\!(u_1,\ldots ,u_n)^\top$, 
\begin{align}\label{strain-r}
E_{j\beta }({\mathbf u})\!:=\!\frac{1}{2}(\partial_j u_\beta +\partial _\beta u_j)
\end{align}
are the entries of the symmetric part, ${\mathbb E}({\mathbf u})$, of the gradient,  $\nabla {\mathbf u}$,  in space coordinates,
and $a_{kj}^{\alpha \beta }(\mathbf x,t)$ are variable components of the tensor viscosity coefficient, cf. \cite{Duffy1978}, 
${\mathbb A}(\mathbf x,t)=\!\left\{{a_{kj}^{\alpha \beta }}(\mathbf x,t)\right\}_{1\leq i,j,\alpha ,\beta \leq n}$, depending on the space coordinate vector $\mathbf x$ and time $t$.
We also denoted $\partial_j=\dfrac{\partial}{\partial x_j}$, $\partial_t=\dfrac{\partial}{\partial t}$.
Here and further on, the Einstein  convention on summation in repeated indices from $1$ to $n$ is used unless stated otherwise.

The following symmetry conditions are assumed (see \cite[(3.1),(3.3)]{Oleinik1992}),
\begin{align}
\label{Stokes-sym}
a_{kj}^{\alpha \beta }(\mathbf x,t)=a_{\alpha j}^{k\beta }(\mathbf x,t)=a_{k\beta }^{\alpha j}(\mathbf x,t).
\end{align}

In addition, we require that tensor ${\mathbb A}$ satisfies the relaxed ellipticity condition  in terms of all {\it symmetric} matrices in ${\mathbb R}^{n\times n}$ with {\it zero matrix trace}, see \cite{KMW-DCDS2021}, \cite{KMW-LP2021}. Thus, we assume that there exists a constant $C_{\mathbb A} >0$ such that, 
\begin{align}
\label{mu}
&C_{\mathbb A}a_{kj}^{\alpha \beta }(\mathbf x,t)\zeta _{k\alpha }\zeta _{j\beta }\geq |\bs\zeta|^2\,,
\ \ 
\mbox{for a.e. } \mathbf x, t,\\
&\forall\ \bs\zeta =\{\zeta _{k\alpha }\}_{k,\alpha =1,\ldots ,n}\in {\mathbb R}^{n\times n}
\mbox{ such that }\, \bs\zeta=\bs\zeta^\top \mbox{ and }
\sum_{k=1}^n\zeta _{kk}=0,
\nonumber
\end{align}
where $|\bs\zeta |=|\bs\zeta |_F:=(\zeta _{k\alpha }\zeta _{k\alpha })^{1/2}$ is the Frobenius matrix norm and the superscript $^\top $ denotes the transpose of a  matrix.
Note that in the more common, strong ellipticity condition (called S-ellipticity condition in \cite[Definition 4.1]{Mitrea-Mitrea2013}), inequality \eqref{mu} should be satisfied for all matrices (not only symmetric with zero trace), which makes it much more restrictive (cf. also E-class in \cite[Section 3.1]{Oleinik1992}, where condition \eqref{mu} is assumed for all symmetric matrices).

We assume that $a_{ij}^{\alpha \beta}\in L_\infty(\R^n\times [0,T])$, where $[0,T]$ is some finite time interval, and the tensor ${\mathbb A}$  is endowed with the norm
\begin{align}\label{TensNorm}
\|{\mathbb A}\|:=\|{\mathbb A}\|_{L_\infty (\R^n\times[0,T]),F}
:=\left| \left\{\|a_{ij}^{\alpha \beta}\|_{L_\infty (\T\times[0,T])}\right\}_{\alpha,\beta,i,j=1}^n\right|_F<\infty,
\end{align}
where $\left|\left\{b_{ij}^{\alpha \beta}\right\}_{\alpha,\beta,i,j=1}^n\right|_F:=\left(b_{ij}^{\alpha \beta}b_{ij}^{\alpha \beta}\right)^{1/2}$ is the Frobenius norm of a 4-th order tensor.

Symmetry conditions \eqref{Stokes-sym} lead to the following equivalent form of the operator $\boldsymbol{\mathfrak L}$
\begin{equation}
\label{Stokes-0}
(\boldsymbol{\mathfrak L}{\mathbf u})_k=\partial _\alpha\big(a_{kj}^{\alpha \beta }\partial _\beta u_j\big),\ \ k=1,\ldots ,n.
\end{equation}

Let ${\mathbf u}(\mathbf x,t)$ be an unknown vector velocity field, $p(\mathbf x,t)$ be an unknown (scalar) pressure field, and ${\mathbf f}(\mathbf x,t)$ be a given vector field 
$\R^n $, where $t\in\R$ is the time variable. 
%
The nonlinear system
\begin{align*}
\partial_t\mathbf u -\boldsymbol{\mathfrak L}{\mathbf u}+\nabla p+({\mathbf u}\cdot \nabla ){\mathbf u}={\mathbf f}\,, \ \ {\rm{div}} \, {\mathbf u}=0
\end{align*}
is the {\it evolution anisotropic incompressible Navier-Stokes system}, where we use the notation $({\mathbf u}\cdot \nabla ):=u_j\partial_j$.

\subsection{Periodic function spaces}\label{S2}

Let us introduce some function spaces on torus, i.e., periodic function spaces (see, e.g., 
\cite[p.26]{Agmon1965},  \cite{Agranovich2015}, \cite{McLean1991}, \cite[Chapter 3]{RT-book2010}, \cite[Section 1.7.1]{RRS2016},  
\cite[Chapter 2]{Temam1995}
for more details).

Let $n\ge 1$  be an integer and $\T$ be the $n$-dimensional flat torus that can be parametrized as the semi-open cube $\T=\T^n=[0,1)^n\subset\R^n$, cf.  \cite[p. 312]{Zygmund2002}.
In what follows, ${\mathcal D}(\T)=\mathcal C^\infty(\T)$ denotes the (test) space of infinitely smooth real or complex functions on the torus.
As usual, $\N$ denotes the set of natural numbers,  $\N_0$ the set of natural numbers augmented by 0, and $\mathbb{Z}$ the set of integers.

Let   $\bs\xi \in \mathbb{Z}^n$ denote the $n$-dimensional vector with integer components. 
We will further need also the set 
$$\dot\Z^n:=\Z^n\setminus\{\mathbf 0\}.$$
Extending the torus parametrisation to $\R^n$, it is often useful to identify $\T$ with the quotient space $\R^n\setminus \Z^n$. 
Then the space of functions $\mathcal C^\infty(\T)$ on the torus can be identified with the space of $\T$-periodic (1-periodic) functions 
$\mathcal C^\infty_\#=\mathcal C^\infty_\#(\R^n)$ that consists of functions $\phi\in \mathcal C^\infty(\R^n)$ such that
\begin{align}
\label{E3.1}
\phi(\mathbf x+\bs\xi)=\phi(\mathbf x)\quad \forall\,  \bs\xi \in \mathbb{Z}^n.
\end{align}
Similarly, the Lebesgue space on the torus $L_{p}(\T)$, $1\le p\le\infty$,  can be identified with the periodic Lebesgue space $L_{p\#}=L_{p\#}(\R^n)$ that consists of functions $\phi\in L_{p,\rm loc}(\R^n)$, which satisfy the periodicity condition \eqref{E3.1} for a.e. $\mathbf x$.

The space dual  to $\mathcal D(\T)$, i.e.,  the space of linear bounded functionals on $\mathcal D(\T)$,  called the space of torus distributions, is denoted by $\mathcal D'(\T)$ and can be identified with the space of periodic distributions $\mathcal D'_\#$ acting on $\mathcal C^\infty_\#$.

The toroidal/periodic Fourier transform 
mapping a  function $g\in \mathcal C_\#^\infty$ to a set of its Fourier coefficients $\hat g$ is defined as (see, e.g., \cite[Definition 3.1.8]{RT-book2010})  
\begin{align*}
 \hat g(\bs\xi)=[\F_{\T} g](\bs\xi):=\int_{\T}e^{-2\pi i \mathbf x\cdot\bs\xi}g(\mathbf x)d\mathbf x,\quad \bs\xi\in\Z^n,
\end{align*}
and can be generalised to  the Fourier transform acting on a distribution $g\in\mathcal D'_\#$.

For any $\bs\xi\in\Z^n$, let $|\bs\xi|:=(\sum_{j=1}^n \xi_j^2)^{1/2}$ be the Euclidean norm in $\Z^n$ and let us denote
 \begin{align}\label{rho}
\varrho(\bs\xi):=2\pi(1+|\bs\xi|^2)^{1/2}.
\end{align}
Evidently,
\begin{align}\label{eq:mik9}
\frac{1}{2}\varrho(\bs\xi)^2\le |2\pi\bs\xi|^2\le \varrho(\bs\xi)^2\quad\forall\,\bs\xi\in \dot\Z^n.
\end{align}

Similar to \cite[Definition 3.2.2]{RT-book2010}, for $s\in\R$ we define the {\em periodic/toroidal Sobolev (Bessel-potential) spaces} $H_\#^s:=H_\#^s(\R^n):=H^s(\T)$ that consist of the torus distributions $ g\in\mathcal D'(\T)$, for which the norm
\begin{align}\label{eq:mik10}
\| g\|_{H_\#^s}:=\| \varrho^s\widehat g\|_{\bs\ell_2(\Z^n)}:=\left(\sum_{\bs\xi\in\Z^n}\varrho(\bs\xi)^{2s}|\widehat g(\bs\xi)|^2\right)^{1/2}
\end{align}
is finite, i.e., the series in \eqref{eq:mik10} converges.
Here $\| \cdot\|_{\bs\ell_2(\Z^n)}$ is the standard norm in the space of square summable sequences with indices in $\Z^n$.
By \cite[Proposition 3.2.6]{RT-book2010}, $H_\#^s$ is the Hilbert space with 
the inner (scalar) product in $H_\#^{s}$  defined as
\begin{align}\label{E3.3i}
(g,f)_{H_\#^s}
:=\sum_{\bs\xi\in\Z^n}\varrho(\bs\xi)^{2s}\hat g(\bs\xi)\overline{\hat f(\bs\xi)},\quad \forall\, g,f\in H_\#^{s},\ s\in\R,
\end{align}
where the bar denotes complex conjugate.
Evidently, $H_\#^{0}=L_{2\#}$.

The dual product between $g\in H_\#^s$ and $f\in H_\#^{-s}$, $s\in\R$, is defined (cf. \cite[Definition 3.2.8]{RT-book2010}) as
\begin{align}\label{E3.3}
\langle g,f\rangle_{\T}
:=\sum_{\bs\xi\in\Z^n}\hat g(\bs\xi)\hat f(-\bs\xi).
\end{align}
If $s=0$, i.e., $g,f\in L_{2\#}$,  then \eqref{E3.3i} and \eqref{E3.3} reduce to
$$
\langle g,f\rangle_{\T}=\int_{\T}g(\mathbf x)f(\mathbf x)d\mathbf x=(g,\bar f)_{L_{2\#}}.
$$
For real function $g,f\in L_{2\#}$ we, of course, have $\langle g,f\rangle_{\T}=(g,f)_{L_{2\#}}$.

For any $s\in\R$, the space  $H_\#^{-s}$ is Banach adjoint (dual) to  $H_\#^{s}$, i.e., $H_\#^{-s}=(H_\#^{s})^*$.
Similar to, e.g., \cite[p.76]{McLean2000} one can show that 
\begin{align*}
\| g\|_{H_\#^s}=\sup_{f\in H_\#^{-s}, f\ne 0}\frac{|\langle g,f\rangle_{\T}|}{\|f\|_{H_\#^{-s}}}.
\end{align*}

For $g\in H_\#^s$, $s\in\R$, and  $m\in\N_0$, let us consider the partial sums 
$$g_m(\mathbf x)=\sum_{\bs\xi\in\Z^n, |\bs\xi|\le m}\hat g(\bs\xi)e^{2\pi i \mathbf x\cdot\bs\xi}.$$ 
 Evidently, $g_m\in \mathcal C_\#^\infty$, $\hat g_m(\bs\xi)=\hat g(\bs\xi)$ if $|\bs\xi|\le m$ and  $\hat g_m(\bs\xi)=0$ if $|\bs\xi|> m$.
This implies that $\|g-g_m\|_{H_\#^s}\to 0$ as $m\to\infty$ and hence we can write 
\begin{align}\label{eq:mik11}
g(\mathbf x)=\sum_{\bs\xi\in\Z^n}\hat g(\bs\xi)e^{2\pi i \mathbf x\cdot\bs\xi},
\end{align}
where the Fourier series converges in the sense of norm \eqref{eq:mik10}.
Moreover, since $g$ is an arbitrary distribution from $H_\#^s$, this also implies that the space ${\mathcal C}^\infty_\#$ is dense in $H_\#^s$ for any $s\in\R$ (cf. \cite[Exercise 3.2.9]{RT-book2010}).

There holds the compact embedding $H_\#^t\hookrightarrow H_\#^s$ if $t>s$,  embeddings $H_\#^s\subset \mathcal C_\#^m$ if $m\in\N_0$, $s>m+\frac{n}{2}$, and moreover, $\bigcap_{s\in\R}H_\#^s={\mathcal C}^\infty_\#$ (cf. \cite[Exercises 3.2.10, 3.2.10, and Corollary 3.2.11]{RT-book2010}). 

Note that for each $s$, the periodic norm \eqref{eq:mik10} is equivalent to the periodic norm that we used in \cite{Mikhailov2022, Mikhailov2023}, which is obtained from \eqref{eq:mik10} by replacing there $\varrho(\bs\xi)=2\pi(1+|\bs\xi|^2)^{1/2}$ with $\rho(\bs\xi)=(1+|\bs\xi|^2)^{1/2}$.
We employ here the norm \eqref{eq:mik10} to simplify some norm estimates further in the paper.
Note also that the periodic norms on $H_\#^s$ are equivalent to the corresponding standard (non-periodic) Bessel potential norms on $\T$ as an $n-$cubic domain, see, e.g., \cite[Section 13.8.1]{Agranovich2015}.

Let 
\begin{align}
\label{E3.14}
 \left(\Lambda_\#^{r}\,{g}\right)(\mathbf x)
&:=\sum_{\bs\xi\in\Z^n}\varrho(\bs\xi)^{r}\hat g(\bs\xi)e^{2\pi i \mathbf x\cdot\bs\xi}\quad \forall\, g\in  H_\#^{s}.
\end{align}
denote the periodic Bessel-potential operator of the order $r\in\R$. 
For any $s\in\R$, the operator 
\begin{align}\label{3.23-0}
\Lambda_\#^{r}&: {H}_{\#}^s\to {H}_{\#}^{s-r},
\end{align}
is  continuous, see, e.g., \cite[Section 13.8.1]{Agranovich2015}.

By \eqref{eq:mik10}, 
$\| g\|^2_{H_\#^s}=|\widehat g(\mathbf 0)|^2 +| g|^2_{H_\#^s},$ 
where
\begin{align*}
| g|_{H_\#^s}:=\| \varrho^s\widehat g\|_{\bs\ell_2(\dot\Z^n)}:=\left(\sum_{\bs\xi\in\dot\Z^n}\varrho(\bs\xi)^{2s}|\widehat g(\bs\xi)|^2\right)^{1/2}
\end{align*}
is the seminorm in $H_\#^s$.

For any $s\in\R$, let us also introduce  the space
\begin{align*}
\dot H_\#^s:=\{g\in H_\#^s: \langle g,1\rangle_{\T}=0\}.
\end{align*}
The definition implies that if $g\in \dot H_\#^s$, then $\widehat g(\mathbf 0)=0$ and 
\begin{align}\label{eq:mik12}
\| g\|_{\dot H_\#^s}=\| g\|_{H_\#^s}=| g|_{H_\#^s}=\| \varrho^s\widehat g\|_{\bs\ell_2(\dot\Z^n)}\ .
\end{align}
The space $\dot H_\#^s$ is the Hilbert space with inner product inherited from \eqref{E3.3i}, that is,
\begin{align}\label{E3.3idot}
(g_1,g_2)_{\dot H_\#^s}
:=\sum_{\bs\xi\in\dot\Z^n}\varrho(\bs\xi)^{2s}\hat g_1(\bs\xi)\overline{\hat g_2(\bs\xi)},\quad \forall\, g_1,g_2\in \dot H_\#^{s},\ s\in\R.
\end{align}
The dual product between $g_1\in \dot H_\#^s$ and $f_2\in (\dot H_\#^{s})^*$, $s\in\R$, is represented as 
\begin{align*}
\langle g_1,f_2\rangle_{\T}
:=\sum_{\bs\xi\in\dot\Z^n}\hat g_1(\bs\xi)\hat f_2(-\bs\xi).
\end{align*}
If $\hat g(\mathbf 0)=0$ then \eqref{E3.14} implies that $\widehat{\Lambda_\#^{r}\,g}(\mathbf 0)=0$, and 
thus the operator 
\begin{align}\label{E3.15}
\Lambda_\#^{r}: \dot H_\#^s\to \dot H_\#^{s-r}
\end{align}
is continuous as well.
Due to the Riesz representation theorem, 
$(\dot H_\#^{s})^*=\dot H_\#^{-s}$, as shown in \cite[Section 2]{Mikhailov2024}.

Denoting 
$
\dot {\mathcal C}^\infty_\#:=\{g\in {\mathcal C}^\infty_\#: \langle g,1\rangle_{\T}=0\}
$,
then $\bigcap_{s\in\R}\dot H_\#^s=\dot {\mathcal C}^\infty_\#$.
The corresponding spaces of $n$-component vector functions/distributions are denoted as $\mathbf L_{q\#}:=(L_{q\#})^n$, $\mathbf H_\#^s:=(H_\#^s)^n$, etc.

Note that  the norm $\|\nabla (\cdot )\|_{{\mathbf H}_\#^{s-1}}$
is an equivalent norm in $\dot H_\#^s$. 
Indeed, by \eqref{eq:mik11}
\begin{align*}
\nabla g(\mathbf x)=2\pi i\sum_{\bs\xi\in\dot\Z^n}\bs\xi e^{2\pi i \mathbf x\cdot\bs\xi}\hat g(\bs\xi),\quad
\widehat{\nabla g}(\bs\xi)=2\pi i\bs\xi \hat g(\bs\xi)\quad \forall\,g\in \dot H_\#^s,
\end{align*}
and 
then \eqref{eq:mik9} and \eqref{eq:mik12} imply
\begin{align}
\frac12|g|^2_{H_\#^s}&\le \| \nabla g\|^2_{{\mathbf H}_\#^{s-1}}
\le |g|^2_{H_\#^s} \quad \forall\,g\in H_\#^s,
\nonumber\\
\label{eq:mik13}
\frac12\|g\|^2_{H_\#^s}=\frac12\|g\|^2_{\dot H_\#^s}=\frac12|g|^2_{H_\#^s}&\le \| \nabla g\|^2_{{\mathbf H}_\#^{s-1}}
\le |g|^2_{H_\#^s}=\|g\|^2_{\dot H_\#^s}=\| g\|^2_{H_\#^s} \quad \forall\,g\in \dot H_\#^s.
\end{align}
The vector counterpart of \eqref{eq:mik13} takes form
\begin{align}\label{eq:mik14}
\frac12\| \mathbf v\|^2_{{\mathbf H}_\#^s}=\frac12\| \mathbf v\|^2_{\dot{\mathbf H}_\#^s}
\le \| \nabla \mathbf v\|^2_{(H_\#^{s-1})^{n\times n}}
\le \| \mathbf v\|^2_{\dot{\mathbf H}_\#^s}=\| \mathbf v\|^2_{{\mathbf H}_\#^s} \quad \forall\,\mathbf v\in \dot {\mathbf H}_\#^s.
\end{align}
Note that the second inequalities in \eqref{eq:mik13} and \eqref{eq:mik14} are valid also in the more general cases, i.e., for $\,g\in H_\#^s$ and $\mathbf v\in  {\mathbf H}_\#^s$, respectively.

We will further need the first Korn inequality
\begin{align}
\label{eq:mik15}
\|\nabla {\mathbf v}\|^2_{(L_{2\#})^{n\times n}}\leq 2\|\mathbb E ({\mathbf v})\|^2_{(L_{2\#})^{n\times n}}\quad\forall\, \mathbf v\in {\mathbf H}_\#^1 
\end{align}
that can be easily proved by adapting, e.g., the proof in \cite[Theorem 10.1]{McLean2000} to the periodic Sobolev space; cf. also \cite[Theorem 2.8]{Oleinik1992}.

Let us also define the Sobolev spaces of divergence-free functions and distributions,
\begin{align*}
\dot{\mathbf H}_{\#\sigma}^{s}
&:=\left\{{\bf w}\in\dot{\mathbf H}_\#^{s}:{\div}\, {\bf w}=0\right\},\quad s\in\R,
\end{align*}
endowed with the same norm \eqref{eq:mik10}.
Similarly, ${\mathbf C}^\infty_{\#\sigma}$ and $\mathbf L_{q\#\sigma}$ denote the subspaces of divergence-free vector-functions from 
${\mathbf C}^\infty_{\#}$ and $\mathbf L_{q\#}$, respectively, etc.

The space $\dot{\mathbf H}_{\#\sigma}^s$ is the Hilbert space with inner product inherited from \eqref{E3.3i} and \eqref{E3.3idot}.
In addition, see \cite[Section 2]{Mikhailov2024},
$
(\dot{\mathbf H}_{\#\sigma}^{s})^*=\dot{\mathbf H}_{\#\sigma}^{-s}.
$
Not that for any $r,s\in\R$ the operator 
\begin{align}\label{3.23}
\Lambda_\#^{r}: \dot{\mathbf H}_{\#\sigma}^s\to \dot{\mathbf H}_{\#\sigma}^{s-r}
\end{align}
defined as in \eqref{E3.14} is continuous.

Let us also introduce the space 
\begin{align*}
\dot{\mathbf H}_{\# g}^{s}
&:=\left\{{\bf w}=\nabla q,\ q\in\dot{H}_\#^{s+1}\right\},\quad s\in\R,
\end{align*}
endowed with the norm \eqref{eq:mik10}.

The following assertion is produced in \cite[Theorem 1]{Mikhailov2024}.
\begin{theorem} \label{Th2.1}
Let $s\in\R$ and $n\ge 2$. 

(a) The space $\dot{\mathbf H}_{\#}^{s}$ has the Helmholtz-Weyl decomposition,
$
\dot{\mathbf H}_{\#}^{s}=\dot{\mathbf H}_{\# g}^{s}\oplus\dot{\mathbf H}_{\# \sigma}^{s},
$
that is, any $\mathbf F\in \dot{\mathbf H}_{\#}^{s}$ can be uniquely represented as 
$\mathbf F=\mathbf F_g+ \mathbf F_\sigma$, where $\mathbf F_g\in \dot{\mathbf H}_{\# g}^{s}$ and 
$\mathbf F_\sigma\in \dot{\mathbf H}_{\#\sigma}^{s}$.

(b) The spaces $\dot{\mathbf H}_{\# g}^{s}$ and $\dot{\mathbf H}_{\# \sigma}^{s}$ are orthogonal subspaces of $\dot{\mathbf H}_{\#}^{s}$ in the sense of inner product, i.e., $(\mathbf w,\mathbf v)_{H_\#^s}=0$ for any $\mathbf w\in \dot{\mathbf H}_{\# g}^{s}$ and $\mathbf v\in \dot{\mathbf H}_{\# \sigma}^{-s}$.

(c) The spaces $\dot{\mathbf H}_{\# g}^{s}$ and $\dot{\mathbf H}_{\# \sigma}^{-s}$ are orthogonal in the sense of dual product, i.e., $\langle\mathbf w,\mathbf v\rangle=0$ for any $\mathbf w\in \dot{\mathbf H}_{\# g}^{s}$ and 
$\mathbf v\in \dot{\mathbf H}_{\# \sigma}^{-s}$.

(d) There exist the bounded orthogonal projector operators $\mathbb P_g:\dot{\mathbf H}_\#^{s}\to \dot{\mathbf H}_{\#g}^{s}$ and $\mathbb P_\sigma:\dot{\mathbf H}_\#^{s}\to \dot{\mathbf H}_{\#\sigma}^{s}$ (the Leray projector), while $\mathbf F=\mathbb P_g\mathbf F+\mathbb P_\sigma\mathbf F$ for any $\mathbf F\in \dot{\mathbf H}_\#^{s}$.
\end{theorem}

For the evolution problems we will systematically use the spaces  $L_q(0,T;H^s_\#)$, $s\in\R$, $1\le q\le\infty$, $0<T<\infty$, which consists of functions that map $t\in(0,T)$ to a function or distributions from $H^s_\#$.
For $1\le q<\infty$, the space $L_q(0,T;H^s_\#)$ is endowed with the norm
\begin{align*}
 \|h\|_{L_q(0,T;H^s_\#)}=\left(\int_0^T\|h(\cdot,t)\|^q_{H^s_\#} d t\right)^{1/q}
=\left(\int_0^T\left[\sZ\varrho(\bs\xi)^{2s}|\widehat h(\bs\xi,t)|^2\right]^{q/2} d t\right)^{1/q}<\infty,
\end{align*}
and for $q=\infty$ with the norm
\begin{align*}
 \|h\|_{L_\infty(0,T;H^s_\#)}={\rm ess\,sup}_{t\in(0,T)}\|h(\cdot,t)\|_{H^s_\#}
={\rm ess\,sup}_{t\in(0,T)}\left[\sZ\varrho(\bs\xi)^{2s}|\widehat h(\bs\xi,t)|^2\right]^{1/2}<\infty.
\end{align*}

For a function (or distribution) $h(\mathbf x,t)$, we will use the notation 
$$
h'(\mathbf x,t):=\partial_t h(\mathbf x,t):=\dfrac{\partial}{\partial t}h(\mathbf x,t),\quad
h^{(j)}(\mathbf x,t):=\partial^j_t h(\mathbf x,t):=\dfrac{\partial^j}{\partial t^j}h(\mathbf x,t)
$$ 
for the partial derivatives in the time variable $t$. 

Let $X$ and $Y$ be some Hilbert spaces. We will further need the space 
\begin{align*}
W^1(0,T;X, Y):=\{u\in L_2(0,T; X) : u'\in L_2(0,T; Y)\}
\end{align*}
endowed with the norm
\begin{align*}
\|u\|_{W^1(0,T;X, Y)}=(\|u\|_{L_2(0,T; X)}^2 +\|u'\|_{L_2(0,T; Y)}^2)^{1/2}.
\end{align*}
Spaces of such type are considered in \cite[Chapter 1, Section 2.2]{Lions-Magenes1}. 
We will particularly need the spaces $W^1(0,T;H^s_\#, H^{s'}_\#)$  and their vector counterparts.

We will also employ the following spaces for $k\in\N$, cf., e.g., \cite[Chapter 1, Section 1.3, Remark 1.5]{Lions-Magenes1},
\begin{align*}
W^k(0,T;X):=\{u\in L_2(0,T; X) : \partial_t^j {u}\in X,\ j=1,\ldots,k\}
\end{align*}
endowed with the norm
\begin{align*}
\|u\|_{W^k(0,T;X)}=\left(\sum_{j=0}^k\|\partial_t^j u\|_{L_2(0,T; X)}^2 \right)^{1/2}.
\end{align*}

Unless stated otherwise, we will assume in this paper that all vector and scalar variables are real-valued (however with complex-valued Fourier coefficients).

\section{Existence results available for evolution spatially-periodic anisotropic Navier-Stokes problem}\label{S3}

Let us consider the following Navier-Stokes problem
for the real-valued unknowns $({\mathbf u},p )$, 
\begin{align}
\label{NS-problem-div0}
\mathbf u' -\bs{\mathfrak L}{\mathbf u}+\nabla p+({\mathbf u}\cdot \nabla ){\mathbf u}&=\mathbf{f}
\quad \mbox{in } \T\times(0,T), 
\\
\label{NS-problem-div0-div}
{\rm{div}}\, {\mathbf u}&=0\quad \mbox{in } \T\times(0,T),
\\
\label{NS-problem-div0-IC}
\mathbf u(\cdot,0)&=\mathbf u^0\ \mbox{in } \T ,
\end{align}
with given data
${\mathbf f}\in L_2(0,T;\dot{\mathbf H}_\#^{-1})$,  $\mathbf u^0\in \dot{\mathbf H}_{\#\sigma}^{0}$. 
Note that the time-trace $\mathbf u(\cdot,0)$ for $\mathbf u$ solving the weak form of \eqref{NS-problem-div0}--\eqref{NS-problem-div0-div} is well defined, see Definition \ref {D6.1} and Remark \ref{R4.3}.

Let us introduce the bilinear form
\begin{align}
\label{NS-a-vsigma}
a_{\T}({\mathbf u},{\mathbf v})=
a_{\T}(t;{\mathbf u},{\mathbf v})
:=
\left\langle a_{ij}^{\alpha \beta }(\cdot,t)E_{j\beta }({\mathbf u}),E_{i\alpha }({\mathbf v})\right\rangle _{\T}
&\,\
\forall \ {\mathbf u}, {\mathbf v}\in \dot{\mathbf H}_{\#}^1.
\end{align}
By the boundedness condition \eqref{TensNorm} and inequality \eqref{eq:mik14} we have
\begin{align}
\label{NS-a-1-v2-S-0v}
|a_{\T}(t;{\mathbf u},{\mathbf v})|
\le \|\mathbb A\| \|{\mathbb E}({\mathbf u})\|_{L_{2\#}^{n\times n}}\|{\mathbb E}({\bf v})\|_{L_{2\#}^{n\times n}}
&\le \|\mathbb A\| \|\nabla{\mathbf u}\|_{L_{2\#}^{n\times n}}\|\nabla{\bf v}\|_{L_{2\#}^{n\times n}}
\nonumber\\
&\le \|\mathbb A\| \|{\mathbf u}\|_{\dot{\mathbf H}_{\#}^1}\|{\bf v}\|_{\dot{\mathbf H}_{\#}^1}
\quad
\forall \ {\mathbf u}, {\mathbf v}\in \dot{\mathbf H}_{\#}^1.
\end{align}
If the relaxed ellipticity condition \eqref{mu} holds, taking into account the relation $\sum_{i=1}^nE_{ii}({\bf w})=\div {\bf w}=0$ for  ${\bf w}\in 
{\dot{\mathbf H}_{\#\sigma}^1}$, 
equivalence of the norm $\|\nabla (\cdot )\|_{L_{2\#}^{n\times n}}$
to the norm $\|\cdot \|_{\dot{\mathbf H}_{\#\sigma}^1}$ in $\dot{\mathbf H}_{\#\sigma}^1$, see \eqref{eq:mik14}, and the first Korn inequality \eqref{eq:mik15},
we obtain
\begin{align}
\label{NS-a-1-v2-S-0}
a_{\T}(t;{\bf w},{\mathbf w})=\left\langle a_{ij}^{\alpha \beta}(\cdot,t)E_{j\beta }({\bf w}),E_{i\alpha }({\bf w})\right\rangle _{\T}
&\geq C_{\mathbb A}^{-1}\|{\mathbb E}({\bf w})\|_{L_{2\#}^{n\times n}}^2
\nonumber\\
&\geq\frac{1}{2}C_{\mathbb A}^{-1}\|\nabla {\bf w}\|_{L_{2\#}^{n\times n}}^2
\geq \frac14 C_{\mathbb A}^{-1}\|{\bf w}\|_{\dot{\mathbf H}_{\#\sigma}^1}^2
\quad\forall \, {\bf w}\in \dot{\mathbf H}_{\#\sigma}^1.
\end{align}
Then  \eqref{NS-a-1-v2-S-0v} and \eqref{NS-a-1-v2-S-0} give
\begin{align}
\label{NS-a-1-v2-S-}
\frac14 C_{\mathbb A}^{-1}\|{\bf w}\|_{\dot{\mathbf H}_{\#\sigma}^1}^2
\le a_{\T}(t;{\bf w},{\mathbf w})
\le\|\mathbb A\| \|{\bf w}\|_{\dot{\mathbf H}_{\#\sigma}^1}^2\quad
\forall \, {\bf w}\in \dot{\mathbf H}_{\#\sigma}^1.
\end{align}

Let us denote
\begin{align}
\label{Eq-F}
\mathbf F:=\mathbf{f}+\bs{\mathfrak L}{\mathbf u}-({\mathbf u}\cdot \nabla ){\mathbf u}.
\end{align}
Let $\mathbf u\in\dot{\mathbf H}_{\#\sigma}^1$.
Acting on \eqref{NS-problem-div0} by the Leray projector $\mathbb P_\sigma$ and taking into account that 
$\mathbb P_\sigma\mathbf u'=\mathbf u'$ and $\mathbb P_\sigma\nabla p=\mathbf 0$, we obtain
\begin{align}
\label{NS-problem-just}
\mathbf u' =\mathbb P_\sigma\mathbf F
&=\mathbb P_\sigma[\mathbf{f}+\bs{\mathfrak L}{\mathbf u}-({\mathbf u}\cdot \nabla ){\mathbf u}]
\quad \mbox{in } \T\times(0,T).
\end{align}
On the other hand, acting on \eqref{NS-problem-div0} by the projector $\mathbb P_g$ and taking into account that 
$\mathbb P_g\mathbf u'=0$ and $\mathbb P_g\nabla p=\nabla p$, we obtain
\begin{align}
\label{Eq-p}
\nabla p=
\mathbb P_g\mathbf F=\mathbb P_g[\mathbf{f}+\bs{\mathfrak L}{\mathbf u}-({\mathbf u}\cdot \nabla ){\mathbf u}]
\quad \mbox{in } \T\times(0,T).
\end{align}

We use the following definition of weak solution given in \cite[Definition 1]{Mikhailov2024}.
\begin{definition}\label{D6.1} Let $n\ge 2$, $T>0$, ${\mathbf f}\in L_2(0,T;\dot{\mathbf H}_\#^{-1})$ and  
$\mathbf u^0\in \dot{\mathbf H}_{\#\sigma}^{0}$.
A function ${\mathbf u}\in  L_{\infty}(0,T;\dot{\mathbf H}_{\#\sigma}^{0})\cap L_2(0,T;\dot{\mathbf H}_{\#\sigma}^{1})$ is called a weak solution of the evolution space-periodic anisotropic Navier-Stokes initial value problem \eqref{NS-problem-div0}--\eqref{NS-problem-div0-IC} if it solves  the initial-variational problem
\begin{align}
\label{NS-eq:mik51a}
&\left\langle {\mathbf u}'(\cdot,t)+\mathbb P_\sigma[({\mathbf u}(\cdot,t)\cdot \nabla ){\mathbf u}(\cdot,t)],{\mathbf w}\right\rangle _{\T }
+a_{\T}({\mathbf u}(\cdot,t),{\mathbf w})
\nonumber\\
&\hspace{16em}
=\langle\mathbf{f}(\cdot,t), \mathbf w\rangle_{\T},\ 
\text{for a.e. } t\in(0,T),\
\forall\,\mathbf w\in \dot{\mathbf H}_{\#\sigma}^{1},
\\ 
\label{NS-eq:mik51in}
&\langle {\mathbf u}(\cdot,0), \mathbf w\rangle_{\T}=\langle\mathbf u^0, \mathbf w\rangle_{\T},\
\forall\,\mathbf w\in \dot{\mathbf H}_{\#\sigma}^{0}.
\end{align}
The associated pressure $p$ is a distribution on $\T\times(0,T)$ satisfying \eqref{NS-problem-div0} in the sense of distributions, 
i.e.,
\begin{align}
\label{NS-eq:mik51a-d}
\left\langle {\mathbf u}'(\cdot,t)+({\mathbf u}(\cdot,t)\cdot \nabla ){\mathbf u}(\cdot,t),{\mathbf w}\right\rangle _{\T }
+a_{\T}({\mathbf u}(\cdot,t),{\mathbf w})
&+\langle\nabla p(\cdot,t),{\mathbf w}\rangle _{\T }
\nonumber\\
&=\langle\mathbf{f}(\cdot,t), \mathbf w\rangle_{\T},\ 
\text{for a.e. }  t\in(0,T),\ 
\forall\,\mathbf w\in {\mathbf C}^\infty_{\#}.
\end{align}
\end{definition}

The following assertion is proved in \cite[Lemma 1]{Mikhailov2024}.
\begin{lemma}\label{L4.2}
Let $n\ge 2$, $T>0$, $a_{ij}^{\alpha \beta }\in L_\infty(0,T;L_{\infty\#})$, ${\mathbf f}\in L_2(0,T;\dot{\mathbf H}_\#^{-1})$ and  
$\mathbf u^0\in \dot{\mathbf H}_{\#\sigma}^{0}$.
Let ${\mathbf u}\in  L_{\infty}(0,T;\dot{\mathbf H}_{\#\sigma}^{0})\cap L_2(0,T;\dot{\mathbf H}_{\#\sigma}^{1})$ solve equation  \eqref{NS-eq:mik51a}. 

(i) Then 
\begin{align}\label{E4.13}
\mathbf{Du}:={\mathbf u}'+\mathbb P_\sigma[({\mathbf u}\cdot \nabla ){\mathbf u}]\in L_2(0,T;\dot{\mathbf H}_{\#\sigma}^{-1})\quad
\text{and}\quad \mathbf{Du}(\cdot,t)\in \dot{\mathbf H}_{\#\sigma}^{-1} \text{ for a.e. } t\in [0,T],
\end{align}
while 
\begin{align}\label{E4.13A}
&({\mathbf u}\cdot \nabla ){\mathbf u}\in L_2(0,T;\dot{\mathbf H}_{\#}^{-n/2})\quad\text{and}\quad 
({\mathbf u}\cdot \nabla ){\mathbf u}(\cdot,t)\in \dot{\mathbf H}_{\#}^{-n/2} \quad \text{ for a.e. } t\in [0,T],
\\
\label{E4.13B}
&
{\mathbf u}'\in L_2(0,T;\dot{\mathbf H}_{\#\sigma}^{-n/2})
\quad \text{and}\quad 
{\mathbf u}'(\cdot,t)\in \dot{\mathbf H}_{\#\sigma}^{-n/2}\quad \text{ for a.e. } t\in [0,T],
\end{align}
and hence ${\mathbf u}\in W^1(\dot{\mathbf H}_{\#\sigma}^{1},\dot{\mathbf H}_{\#\sigma}^{-n/2})$.

In addition,
\begin{align}
\label{E4.29-3an}
&\partial_t\|{\mathbf u}\|_{\dot{\mathbf H}^{{-(n-2)/4}}_{\#\sigma}}^2
=2\langle \Lambda_\#^{-n/2} \mathbf u', \Lambda_\# {\mathbf u}\rangle _{\T } 
=2\langle \mathbf u',\Lambda_\#^{1-n/2}{\mathbf u}\rangle _{\T } 
=2\langle \Lambda_\#^{1-n/2}\mathbf u',{\mathbf u}\rangle _{\T } 
\end{align}
for a.e. $t\in(0,T)$ and also in the distribution sense on $(0,T)$.

(ii) Moreover, $\mathbf u$ is almost everywhere on $[0,T]$ equal to a function $\widetilde{\mathbf u}\in \mathcal C^0([0,T];\dot{\mathbf H}_{\#\sigma}^{-(n-2)/4})$,
and $\widetilde{\mathbf u}$ is also $\dot{\mathbf H}^0_{\#\sigma}$-weakly continuous in time on $[0,T]$, that is,
$$
\lim_{t\to t_0}\langle\widetilde{\mathbf u}(\cdot,t),\mathbf w\rangle_\T=\langle\widetilde{\mathbf u}(\cdot,t_0),\mathbf w\rangle_\T\quad \forall\, \mathbf w\in \mathbf H^0_{\#}, \ \forall\, t_0\in[0,T].
$$

(iii) There exists the associated pressure $p\in L_2(0,T;\dot{H}_{\#}^{-n/2+1})$ that for the given $\mathbf u$ is the unique solution of equation \eqref{NS-problem-div0} in this space. 
\end{lemma}

\begin{remark}\label{R4.3}
The initial condition \eqref{NS-eq:mik51in} should be understood for the function $\mathbf u$ re-defined as the function 
$\widetilde{\mathbf u}$ that was introduced in Lemma \ref{L4.2}(ii) and is $\mathbf H^0_{\#}$-weakly continuous in time.
\end{remark}

The following existence theorem was proved in \cite[Theorem 2]{Mikhailov2024}.
\begin{theorem}[Existence]
\label{NS-problemTh-sigma}
Let $n\ge 2$ and $T>0$. 
Let $a_{ij}^{\alpha \beta }\in L_\infty(0,T;L_{\infty\#})$
and the relaxed ellipticity condition \eqref{mu} hold.  
Let ${\mathbf f}\in L_2(0,T;\dot{\mathbf H}_\#^{-1})$, $\mathbf u^0\in \dot{\mathbf H}_{\#\sigma}^{0}$.

(i) Then there exists a weak solution 
${\mathbf u}\in  L_{\infty}(0,T;\dot{\mathbf H}_{\#\sigma}^{0})\cap L_2(0,T;\dot{\mathbf H}_{\#\sigma}^{1})$
of the anisotropic Navier-Stokes 
initial value problem  \eqref{NS-problem-div0}--\eqref{NS-problem-div0-IC} in the sense of Definition \ref{D6.1}.
Particularly,
$
\lim_{t\to 0}\langle{\mathbf u}(\cdot,t),\mathbf v\rangle_\T
=\langle{\mathbf u}^0,\mathbf v\rangle_\T\quad \forall\, \mathbf v\in \dot{\mathbf H}^0_{\#\sigma}.
$
There exists also the unique pressure $p\in L_2(0,T;\dot{H}_{\#}^{-n/2+1})$ associated with the obtained $\mathbf u$, that is the solution of equation \eqref{NS-problem-div0} in $L_2(0,T;\dot{H}_{\#}^{-n/2+1})$. 

(ii) Moreover, $\mathbf u$ satisfies  the following (strong) energy inequality,
\begin{align}
\frac12\| {\mathbf u}(\cdot,t)\| _{\mathbf L_{2\#}}^2
+\int_{t_0}^{t}a_{\T}({\mathbf u}(\cdot,\tau),{\mathbf u}(\cdot,\tau)) d\tau
\label{NS-eq:mik51ai=a}
\le \frac12 \| {\mathbf u}(\cdot,t_0)\| ^2_{\mathbf L_{2\#}}
+\int_{t_0}^{t}\langle\mathbf{f}(\cdot,\tau), {\mathbf u}(\cdot,\tau)\rangle_{\T}d\tau
\end{align}
for any $[t_0,t]\subset[0,T]$. It particularly implies
the standard energy inequality,
\begin{align}
\label{E4.9}
\frac12\| {\mathbf u}(\cdot,t)\| _{\mathbf L_{2\#}}^2
+\int_0^ta_{\T}({\mathbf u}(\cdot,\tau),{\mathbf u}(\cdot,\tau)) d\tau
\le \frac12\| {\mathbf u}^0\| ^2_{\mathbf L_{2\#}}
+\int_0^t\langle\mathbf{f}(\cdot,\tau), {\mathbf u}(\cdot,\tau)\rangle_{\T}\,d\tau
\quad \forall\,
t\in[0,T].
\end{align}
\end{theorem}

\section{Serrin-type solutions and their properties}\label{S3-Serrin-intro}
For the isotropic constant-coefficient homogeneous Navier-Stokes equations, it is well know  that the weak solution satisfying the famous Prodi-Serrin condition (accommodated here for the periodic vector setting)
\begin{align}\label{PScond}
\mathbf u\in L_{\widetilde q}(0,T;\dot{\mathbf L}_{\# q})
\end{align}
for some $\widetilde q$ and $ q$ such that 
\begin{align}\label{PScond1}
\frac{2}{\widetilde q}+\frac{n}{ q}=1, \quad n< q<\infty,
\end{align}
is unique in the class of weak solutions satisfying the energy inequality, for $n\le 4$;  the energy equality and the regularity results are also proved under the Prodi-Serrin conditions, see, e.g., \cite{Serrin1962}, \cite{Serrin1963}, \cite[Chapter 1, Theorem 6.9 and Remark 6.8]{Lions1969}, \cite[Section 14]{Lemarie-Rieusset2002}, \cite[Section 8.5]{RRS2016}, \cite[Theorem 7.17]{Seregin2015}, \cite[Section 1.5]{Sohr2001}. 

In this paper we limit ourself to the $L_2$-based Sobolev spaces with respect to the spatial variables and hence introduce a corresponding particular counterpart of the class of solutions satisfying the conditions close to \eqref{PScond}--\eqref{PScond1} and leading to the Serrin-type results.
\subsection{Serrin-type solutions and their properties for $n\ge 2$}
\begin{definition}\label{CSDef}
Let $n\ge 2$, $T>0$, ${\mathbf f}\in L_2(0,T;\dot{\mathbf H}_\#^{-1})$ and  
$\mathbf u^0\in \dot{\mathbf H}_{\#\sigma}^{0}$. 
If a solution $\mathbf u$ of the initial-variational problem \eqref{NS-eq:mik51a}--\eqref{NS-eq:mik51in}
belongs to $L_2(0,T;\dot{\mathbf H}_{\#\sigma}^{n/2})$, we will call it a {\bf Serrin-type solution}.
\end{definition}
The inclusion $\mathbf u \in L_2(0,T;\dot{\mathbf H}_{\#\sigma}^{n/2})$ can be considered as counterpart of the Prodi-Serrin  condition \eqref{PScond} in $L_2$-based Sobolev spaces. Indeed, by the Sobolev embedding theorem,
Theorem \ref{SET}, we obtain that the following continuous embeddings hold, 
$\dot{\mathbf H}_{\#\sigma}^{n/2}\subset {\mathbf H}_{\#}^{n/2}\subset \mathbf L_{\#q}$ for any $q\in(2,\infty)$.
Hence if $\mathbf u\in L_2(0,T;\dot{\mathbf H}_{\#\sigma}^{n/2})$, then for any $\epsilon>0$ there exists $q_\epsilon\in(2,\infty)$ such that  
$\mathbf u\in L_{\widetilde q}(0,T;\mathbf L_{\#q_\epsilon})$ 
for $\widetilde q=2$
and 
\begin{align}\label{PScond-eps}
\frac{2}{\widetilde q}+\frac{n}{ q_\epsilon}=1+\epsilon.
\end{align}
Condition \eqref{PScond-eps} is weaker than condition \eqref{PScond1} by the arbitrarily small $\epsilon>0$, but in spite of this, we will be able to prove that the inclusion $\mathbf u\in L_2(0,T;\dot{\mathbf H}_{\#\sigma}^{n/2})$ for the weak solution is sufficient to prove for it the Serrin-type results about the energy equality, uniqueness and regularity, which justifies the chosen Serrin-type solution name. 
We will also prove  the existence of such solutions, under appropriate conditions.

\begin{definition}\label{StrSDef}
Let $n\ge 2$, $T>0$, ${\mathbf f}\in L_2(0,T;\dot{\mathbf H}_\#^{-1})$ and  
$\mathbf u^0\in \dot{\mathbf H}_{\#\sigma}^{0}$. 
If a solution $\mathbf u$ of the initial-variational problem \eqref{NS-eq:mik51a}--\eqref{NS-eq:mik51in}
belongs to $L_2(0,T;\dot{\mathbf H}_{\#\sigma}^{2})$, we will call it a {\bf a strong solution}.
\end{definition}
The above definition of the strong solution is a bit weaker than, e.g., in \cite[Definition 6.1]{RRS2016}) or \cite[Chapter 1, Section 6.7]{Lions1969}, because it does not explicitly require the additional  inclusion $\mathbf u\in L_\infty(0,T;\dot{\mathbf H}_{\#\sigma}^{1})$ or $\mathbf u'\in L_\infty(0,T;\dot{\mathbf H}_{\#\sigma}^{0})$.

\begin{remark}
Definitions \ref{CSDef} and \ref{StrSDef} imply that  the strong solutions are also Serrin-type solutions if $n\in\{2,3,4\}$. 

The Serrin-type solutions are also strong solutions if $n\ge 4$. Some sufficient conditions for the Serrin-type solution existence are provided in  Section \ref{S5.3} further on in the paper.

If $n\in\{2,3\}$, then for a Serrin-type solution to be also a strong solution, the Serrin-type solution should have an additional regularity and the sufficient conditions for this are provided by the regularity Theorems and Corollaries in Sections \ref{S5ERC} and \ref{S5ERV},  with the parameter $r\ge 1$ there.
\end{remark} 
\begin{lemma}\label{L5.2}
Let $n\ge 2$, $T>0$, $a_{ij}^{\alpha \beta }\in L_\infty(0,T;L_{\infty\#})$,
 ${\mathbf f}\in L_2(0,T;\dot{\mathbf H}_\#^{-1})$ and  
$\mathbf u^0\in \dot{\mathbf H}_{\#\sigma}^{0}$.
Let $\mathbf{u}$ be a Serrin-type solution.
Then $({\mathbf u}\cdot \nabla ){\mathbf u}\in{L_2(0,T;\mathbf{H}_\#^{-1})}$,  
${\mathbf u}'\in{L_2(0,T;\mathbf{H}_{\#\sigma}^{-1})}$ and 
hence ${\mathbf u}\in W^1(\dot{\mathbf H}_{\#\sigma}^{n/2},\dot{\mathbf H}_{\#\sigma}^{-1})$ and
$
\mathbf u \in\mathcal C^0([0,T];\dot{\mathbf H}_{\#\sigma}^{n/4-1/2})\subset \mathcal C^0([0,T];\dot{\mathbf H}_{\#\sigma}^{0}).
$
Moreover,
\begin{align}
\label{NS-eq:mik51ar}
\left\langle {\mathbf u}'(\cdot,t),{\mathbf w}\right\rangle _{\T }
+\left\langle({\mathbf u}(\cdot,t)\cdot \nabla ){\mathbf u}(\cdot,t),{\mathbf w}\right\rangle _{\T }
+a_{\T}(t;{\mathbf u},{\mathbf w})
=\langle\mathbf{f}(\cdot,t), \mathbf w\rangle_{\T},\ 
\text{for a.e. } t\in(0,T),\
\forall\,\mathbf w\in \dot{\mathbf H}_{\#\sigma}^{1}.
\end{align}
The unique pressure  $p$ associated with the obtained $\mathbf u$ belongs to $L_2(0,T;\dot{H}_{\#}^{0})$. 
\end{lemma}
\begin{proof}
By relation  \eqref{eq:mik14}, the multiplication Theorem \ref{RS-T1-S4.6.1}(b) and the Sobolev interpolation inequality \eqref{SII2},
\begin{align}\label{E4.91-0T0EE}
\left\|({\mathbf u}\cdot \nabla ){\mathbf u}\right\|_{\mathbf{H}_\#^{-1}}
=\left\|\nabla\cdot({\mathbf u} \otimes{\mathbf u})\right\|_{\mathbf{H}_\#^{-1}}
\le \left\|{\mathbf u} \otimes{\mathbf u})\right\|_{({H}_\#^{0})^{n\times n}}^2
\le C_{*n} \|\mathbf{u}\|_{{\mathbf H}_\#^{n/4}}^2
\le C_{*n}\|\mathbf{u}\|_{{\mathbf H}_\#^0} \|\mathbf{u}\|_{{\mathbf H}_\#^{n/2}},
\end{align}
where $C_{*n}=C_*(n/4,n/4, n)$.
Hence 
\begin{align}\label{E4.91-0T0EE1}
\left\|({\mathbf u}\cdot \nabla ){\mathbf u}\right\|_{L_2(0,T;\mathbf{H}_\#^{-1})}
\le C_{*n}\|\mathbf{u}\|_{L_\infty(0,T;{\mathbf H}_\#^0)} \|\mathbf{u}\|_{L_2(0,T;{\mathbf H}_\#^{n/2})},
\end{align}
that is, $({\mathbf u}\cdot \nabla ){\mathbf u}\in{L_2(0,T;\mathbf{H}_\#^{-1})}$ and \eqref{E4.13} implies that 
${\mathbf u}'\in{L_2(0,T;\mathbf{H}_{\#\sigma}^{-1})}$.
Hence ${\mathbf u}\in W^1(\dot{\mathbf H}_{\#\sigma}^{n/2},\dot{\mathbf H}_{\#\sigma}^{-1})$, and Theorem \ref{LM-T3.1} implies that 
$
\mathbf u \in\mathcal C^0([0,T];\dot{\mathbf H}_{\#\sigma}^{n/4-1/2})\subset \mathcal C^0([0,T];\dot{\mathbf H}_{\#\sigma}^{0})
$.
Because of this,  equation  \eqref{NS-eq:mik51a} for  function ${\mathbf u}$ now reduces to \eqref{NS-eq:mik51ar}.

To prove the lemma claim about the associated pressure $p$ we remark that it satisfies \eqref{Eq-p}, 
where $\mathbf F\in L_2(0,T;\dot{\mathbf H}_{\# }^{-1})$
due to the lemma conditions and the inclusion in $({\mathbf u}\cdot \nabla ){\mathbf u}\in{L_2(0,T;\mathbf{H}_\#^{-1})}$.
By Lemma \ref{div-grad-is} for gradient, with $s=0$, equation \eqref{Eq-p} has a unique solution $p$ in 
$L_2(0,T;\dot{H}_{\#}^{0})$.
\end{proof}

Let us prove, in the variable-coefficient anisotropic setting, the energy equality and solution uniqueness for the Serrin-type solutions.
\begin{theorem}[Energy equality for Serrin-type solutions]\label{Th4.4}
Let $n\ge 2$, $T>0$,  $a_{ij}^{\alpha \beta }\in L_\infty(0,T;L_{\infty\#})$, ${\mathbf f}\in L_2(0,T;\dot{\mathbf H}_\#^{-1})$ and  
$\mathbf u^0\in \dot{\mathbf H}_{\#\sigma}^{0}$. 
If $\mathbf u$ is a Serrin-type solution 
of the initial-variational problem \eqref{NS-eq:mik51a}--\eqref{NS-eq:mik51in},
then the following energy equality holds for any $[t_0,t]\subset[0,T]$,
\begin{align}\label{NS-eq:mik51ai=aTEE}
\frac12\| {\mathbf u}(\cdot,t)\| _{\mathbf L_{2\#}}^2
+\int_{t_0}^{t}a_{\T}(\tau;{\mathbf u}(\cdot,\tau),{\mathbf u}(\cdot,\tau)) d\tau
=  \frac12\| {\mathbf u}(\cdot,t_0)\| ^2_{\mathbf L_{2\#}}
+\int_{t_0}^{t}\langle\mathbf{f}(\cdot,\tau), {\mathbf u}(\cdot,\tau)\rangle_{\T}d\tau.
\end{align}
It particularly implies the standard energy equality,
\begin{align}
\label{E4.9TEE}
\frac12\| {\mathbf u}(\cdot,t)\| _{\mathbf L_{2\#}}^2
+\int_0^ta_{\T}(\tau;{\mathbf u}(\cdot,\tau),{\mathbf u}(\cdot,\tau)) d\tau
= \frac12\| {\mathbf u}^0\| ^2_{\mathbf L_{2\#}}
+\int_0^t\langle\mathbf{f}(\cdot,\tau), {\mathbf u}(\cdot,\tau)\rangle_{\T}\,d\tau
\quad \forall\,
t\in[0,T].
\end{align}
\end{theorem}
\begin{proof}
By Lemma \ref{L5.2}, 
the function ${\mathbf u}$ satisfies equation  \eqref{NS-eq:mik51ar}, 
where for a.e. $t\in(0,T)$ we can employ $\mathbf u$ as $\mathbf w$ to obtain
\begin{align}
\label{NS-eq:mik51arAEE}
&\left\langle {\mathbf u}'(\cdot,t),{\mathbf u}(\cdot,t)\right\rangle _{\T }
+\left\langle({\mathbf u}(\cdot,t)\cdot \nabla ){\mathbf u}(\cdot,t),{\mathbf u}(\cdot,t)\right\rangle _{\T }
+a_{\T}(t;{\mathbf u},{\mathbf u}(\cdot,t))
=\langle\mathbf{f}(\cdot,t), {\mathbf u}(\cdot,t)\rangle_{\T},\ 
\text{for a.e. } t\in(0,T).
\end{align}
Taking into account Lemma \ref{L4.9} with $s=1$, $s'=-1$ for the first dual product and relation \eqref{eq:mik55}  for the second dual product, we get 
\begin{align}
\label{NS-eq:mik51arBEE}
\frac12\partial_t\left\| {\mathbf u}(\cdot,t)\right\|_{\mathbf L_{2\#}}^2
+a_{\T}(t;{\mathbf u},{\mathbf u}(\cdot,t))
=\langle\mathbf{f}(\cdot,t), {\mathbf u}(\cdot,t)\rangle_{\T},\ 
\text{for a.e. } t\in(0,T).
\end{align}
By the inclusions obtained in Lemma \ref{L5.2}, each dual product  and the bilinear form $a_\T$ in \eqref{NS-eq:mik51arAEE} and hence in \eqref{NS-eq:mik51arBEE} are integrable in $t$. 
After integrating \eqref{NS-eq:mik51arBEE}, we obtain \eqref{NS-eq:mik51ai=aTEE}
for a.e. $t_0$.

By Lemma \ref{L5.2}
$
\mathbf u \in \mathcal C^0([0,T];\dot{\mathbf H}_{\#\sigma}^{0}),
$
while the integrals in \eqref{NS-eq:mik51ai=aTEE} are continuous in  $t_0$ as well. 
Then we conclude that the energy equality \eqref{NS-eq:mik51ai=aTEE} holds for any $t_0\in[0,T)$, implying also \eqref{E4.9TEE}.
\end{proof}
\begin{theorem}[Uniqueness of Serrin-type solutions]\label{Th4.5}
Let $n\ge 2$, $T>0$,  $a_{ij}^{\alpha \beta }\in L_\infty(0,T;L_{\infty\#})$, ${\mathbf f}\in L_2(0,T;\dot{\mathbf H}_\#^{-1})$ and  
$\mathbf u^0\in \dot{\mathbf H}_{\#\sigma}^{0}$. 
Let $\mathbf u$ be a Serrin-type solution 
of the initial-variational problem \eqref{NS-eq:mik51a}--\eqref{NS-eq:mik51in} on the interval $[0,T]$
and $\mathbf v$ be any solution 
of the initial-variational problem \eqref{NS-eq:mik51a}--\eqref{NS-eq:mik51in} satisfying the energy inequality \eqref{E4.9} on the interval $[0,T]$.
Then $\mathbf u = \mathbf v$ on $[0,T]$.
\end{theorem}
\begin{proof}
We will here generalise the proof of  Theorem 6.10 in \cite{RRS2016}.

By Lemma \ref{L5.2}, the function ${\mathbf u}$ satisfies equation  \eqref{NS-eq:mik51ar}, 
where for a.e. $t\in(0,T)$ we can employ $\mathbf v$ as $\mathbf w$ to obtain
\begin{align}
\label{NS-eq:mik51arAEF}
\left\langle {\mathbf u}'(\cdot,t),{\mathbf v}(\cdot,t)\right\rangle_{\T }
+\left\langle({\mathbf u}(\cdot,t)\cdot \nabla ){\mathbf u}(\cdot,t),{\mathbf v}(\cdot,t)\right\rangle _{\T }
+a_{\T}(t;{\mathbf u}(\cdot,t),{\mathbf v}(\cdot,t))
=\langle\mathbf{f}(\cdot,t), {\mathbf v}(\cdot,t)\rangle_{\T}.
\end{align}
On the other hand, equation \eqref{NS-eq:mik51a} for $\mathbf v$ with $\mathbf u$ employed for $\mathbf w$ can be written for a.e. $t$ as
\begin{align}
\label{NS-eq:mik51arAEG}
\left\langle {\mathbf v}'(\cdot,t),{\mathbf u}(\cdot,t)\right\rangle_{\T }
+\left\langle({\mathbf v}(\cdot,t)\cdot \nabla ){\mathbf v}(\cdot,t),{\mathbf u}(\cdot,t)\right\rangle _{\T }
+a_{\T}(t;{\mathbf v}(\cdot,t),{\mathbf u}(\cdot,t))
=\langle\mathbf{f}(\cdot,t), {\mathbf u}(\cdot,t)\rangle_{\T},
\end{align}
where we took into account that 
${\mathbf u}(t)\in\dot{\mathbf H}_{\#\sigma}^{n/2}\subset\dot{\mathbf H}_{\#\sigma}^{1}$ for a.e. $t$.
Adding equations \eqref{NS-eq:mik51arAEF} and  \eqref{NS-eq:mik51arAEG} and integrating in time, we obtain
\begin{multline}
\label{NS-eq:mik51arAEH}
\int_0^t\left[\langle {\mathbf u}'(\cdot,\tau),{\mathbf v}(\cdot,\tau)\rangle_{\T } 
+ \langle {\mathbf v}'(\cdot,\tau),{\mathbf u}(\cdot,\tau)\rangle_{\T}\right]d\tau
+2\int_0^ta_{\T}(\tau;{\mathbf u}(\cdot,\tau),{\mathbf v}(\cdot,\tau))dt
\\
+\int_0^t\left[\langle({\mathbf u}(\cdot,\tau)\cdot \nabla ){\mathbf u}(\cdot,\tau),{\mathbf v}(\cdot,\tau)\rangle _{\T }
+\langle({\mathbf v}(\cdot,\tau)\cdot \nabla ){\mathbf v}(\cdot,\tau),{\mathbf u}(\cdot,\tau)\rangle _{\T }\right]d\tau
\\
=\int_0^t\langle\mathbf{f}(\cdot,\tau), {\mathbf u}(\cdot,\tau)\rangle_{\T}d\tau
+\int_0^t\langle\mathbf{f}(\cdot,\tau), {\mathbf v}(\cdot,\tau)\rangle_{\T}d\tau
\end{multline}

By Lemma \ref{L4.2}(i),  $\mathbf v\in W^1(0,T;\dot{\mathbf H}_{\#\sigma}^{1}, \dot{\mathbf H}_{\#\sigma}^{-n/2})$ and hence by Lemma \ref{LM-T3.1}  the traces  ${\mathbf v}(\cdot,0), {\mathbf v}(\cdot,t)\in \dot{\mathbf H}_{\#\sigma}^{1/2-n/4}$ are well defined. 
On the other hand, by Lemma \ref{L5.2} $\mathbf u\in W^1(0,T;\dot{\mathbf H}_{\#\sigma}^{n/2}, \dot{\mathbf H}_{\#\sigma}^{-1})$ and hence by Lemma \ref{LM-T3.1}  the traces  ${\mathbf u}(\cdot,0), {\mathbf u}(\cdot,t)\in \dot{\mathbf H}_{\#\sigma}^{n/4-1/2}$ are well defined.
Then due to Lemma \ref{L4.9}(ii) with $s=n/2$ and $s'=-1$,
\begin{align}
\label{E5.1}
\int_0^t\left[\langle {\mathbf u}'(\cdot,\tau),{\mathbf v}(\cdot,\tau)\rangle_{\T } 
+ \langle {\mathbf v}'(\cdot,\tau),{\mathbf u}(\cdot,\tau)\rangle_{\T}\right]d\tau
=\left\langle {\mathbf u}(\cdot,t),{\mathbf v}(\cdot,t)\right\rangle_{\T }
-\left\langle {\mathbf u}(\cdot,0),{\mathbf v}(\cdot,0)\right\rangle _{\T }.
\end{align}

Let us denote $\widetilde{\mathbf w}:=\mathbf u - \mathbf v$. 
Since 
${\mathbf u}(\cdot,0)={\mathbf v}(\cdot,0)=\mathbf u^0\in \dot{\mathbf H}_{\#\sigma}^{0}$
and $\mathbf u,\mathbf v, \widetilde{\mathbf w} \in L_\infty(0,T;\dot{\mathbf H}_{\#\sigma}^{0})$, we obtain
\begin{align}
\left\langle {\mathbf u}(\cdot,0),{\mathbf v}(\cdot,0)\right\rangle _{\T }
&=\left\|{\mathbf u}^0\right\|^2_{\dot{\mathbf H}_{\#\sigma}^{0}}
=\left\|{\mathbf v}^0\right\|^2_{\dot{\mathbf H}_{\#\sigma}^{0}},
\\
\left\langle {\mathbf u}(\cdot,t),{\mathbf v}(\cdot,t)\right\rangle_{\T }
&=\frac12\left\langle {\mathbf u}(\cdot,t),{\mathbf u}(\cdot,t)\right\rangle_{\T }
+\frac12\left\langle {\mathbf v}(\cdot,t),{\mathbf v}(\cdot,t)\right\rangle_{\T }
-\frac12\left\langle {\widetilde{\mathbf w}}(\cdot,t),{\widetilde{\mathbf w}}(\cdot,t)\right\rangle_{\T }
\nonumber\\
&=\frac12\left\|{\mathbf u}(\cdot,t)\right\|^2_{\dot{\mathbf H}_{\#\sigma}^{0}}
+\frac12\left\|{\mathbf v}(\cdot,t)\right\|^2_{\dot{\mathbf H}_{\#\sigma}^{0}}
-\frac12\left\|{\widetilde{\mathbf w}}(\cdot,t)\right\|^2_{\dot{\mathbf H}_{\#\sigma}^{0}}
\text{ for a.e. } t\in(0,T).
\end{align}
Due to \eqref{NS-a-vsigma},
$$
2a_{\T}(\tau;{\mathbf u}(\cdot,\tau),{\mathbf v}(\cdot,\tau))
=a_{\T}(\tau;{\mathbf u}(\cdot,\tau),{\mathbf u}(\cdot,\tau))
+a_{\T}(\tau;{\mathbf v}(\cdot,\tau),{\mathbf v}(\cdot,\tau))
-a_{\T}(\tau;{\widetilde{\mathbf w}}(\cdot,\tau),{\widetilde{\mathbf w}}(\cdot,\tau)).
$$
By relations \eqref{E-B.20} and \eqref{eq:mik55}, we obtain
\begin{align}\label{E5.4}
\langle({\mathbf u}(\cdot,\tau)\cdot \nabla )&{\mathbf u}(\cdot,\tau),{\mathbf v}(\cdot,\tau)\rangle _{\T }
+\langle({\mathbf v}(\cdot,\tau)\cdot \nabla ){\mathbf v}(\cdot,\tau),{\mathbf u}(\cdot,\tau)\rangle _{\T }
\nonumber\\
&=\langle({\mathbf v}(\cdot,\tau)\cdot \nabla ){\mathbf v}(\cdot,\tau)
-({\mathbf u}(\cdot,\tau)\cdot \nabla) {\mathbf v}(\cdot,\tau),{\mathbf u}(\cdot,\tau)\rangle _{\T }
\nonumber\\
&=\langle({\widetilde{\mathbf w}}(\cdot,\tau)\cdot \nabla ){\widetilde{\mathbf w}}(\cdot,\tau)
-({\widetilde{\mathbf w}}(\cdot,\tau)\cdot \nabla) {\mathbf u}(\cdot,\tau),{\mathbf u}(\cdot,\tau)\rangle _{\T }
=\langle({\widetilde{\mathbf w}}(\cdot,\tau)\cdot \nabla ){\widetilde{\mathbf w}}(\cdot,\tau),{\mathbf u}(\cdot,\tau)\rangle _{\T }.
\end{align}
Substituting \eqref{E5.1}--\eqref{E5.4} into \eqref{NS-eq:mik51arAEH}, we get
\begin{multline}\label{E5.15}
\frac12\left\|\widetilde{\mathbf w}(\cdot,t)\right\|^2_{\dot{\mathbf H}_{\#\sigma}^{0}}
+\int_0^ta_{\T}(\tau;\widetilde{\mathbf w}(\cdot,\tau),\widetilde{\mathbf w}(\cdot,\tau))dt
-\int_0^t\langle({\widetilde{\mathbf w}}(\cdot,\tau)\cdot \nabla ){\widetilde{\mathbf w}}(\cdot,\tau),{\mathbf u}(\cdot,\tau)\rangle _{\T }d\tau
\\
=A(\mathbf u;t) +A(\mathbf v;t)
\text{ for a.e. } t\in(0,T).
\end{multline}
Here 
\begin{align}
A(\mathbf u):=\frac12\left\|{\mathbf u}(\cdot,t)\right\|^2_{\dot{\mathbf H}_{\#\sigma}^{0}}
+\int_0^ta_{\T}(\tau;{\mathbf u}(\cdot,\tau),{\mathbf u}(\cdot,\tau))dt
-\frac12\left\|{\mathbf u}^0\right\|^2_{\dot{\mathbf H}_{\#\sigma}^{0}}
-\int_0^t\langle\mathbf{f}(\cdot,\tau), {\mathbf u}(\cdot,\tau)\rangle_{\T}d\tau=0
\end{align}
by the energy equality condition \eqref{E4.9TEE} for $\mathbf u$, while 
\begin{align}\label{E5.17}
A(\mathbf v):=\frac12\left\|{\mathbf v}(\cdot,t)\right\|^2_{\dot{\mathbf H}_{\#\sigma}^{0}}
+\int_0^ta_{\T}(\tau;{\mathbf v}(\cdot,\tau),{\mathbf v}(\cdot,\tau))dt
-\frac12\left\|{\mathbf v}^0\right\|^2_{\dot{\mathbf H}_{\#\sigma}^{0}}
-\int_0^t\langle\mathbf{f}(\cdot,\tau), {\mathbf v}(\cdot,\tau)\rangle_{\T}d\tau\le 0
\end{align}
by the energy inequality condition \eqref{E4.9} for $\mathbf v$.

Taking into account inequality \eqref{NS-a-1-v2-S-0} for the quadratic form $a$, \eqref{E5.15}--\eqref{E5.17} imply
\begin{align}\label{E5.18}
\frac12\left\|\widetilde{\mathbf w}(\cdot,t)\right\|^2_{\dot{\mathbf H}_{\#\sigma}^{0}}
+\frac{1}{4}C_{\mathbb A}^{-1}\int_0^t\|\widetilde{\mathbf w}(\cdot,\tau)\|^2_{\dot{\mathbf H}_{\#\sigma}^{1}}dt
\le\int_0^t\left|\langle({\widetilde{\mathbf w}}(\cdot,\tau)\cdot \nabla ){\widetilde{\mathbf w}}(\cdot,\tau),
{\mathbf u}(\cdot,\tau)\rangle _{\T }\right|d\tau.
\end{align}
By the multiplication Theorem \ref{RS-T1-S4.6.1}(b),  the Sobolev interpolation inequality \eqref{SII2} and Young's inequality, we obtain
\begin{multline}\label{E5.19}
\left|\langle({\widetilde{\mathbf w}}(\cdot,\tau)\cdot \nabla ){\widetilde{\mathbf w}}(\cdot,\tau),
{\mathbf u}(\cdot,\tau)\rangle _{\T }\right|
\le\|({\widetilde{\mathbf w}}(\cdot,\tau)\cdot \nabla ){\widetilde{\mathbf w}}(\cdot,\tau)\|_{\dot{\mathbf H}_{\#\sigma}^{-n/2}}
\|{\mathbf u}(\cdot,\tau)\|_{\dot{\mathbf H}_{\#\sigma}^{n/2}}
\\
\le\|\nabla\cdot[{\widetilde{\mathbf w}}(\cdot,\tau)\otimes {\widetilde{\mathbf w}}(\cdot,\tau)]\|_{\dot{\mathbf H}_{\#\sigma}^{-n/2}}
\|{\mathbf u}(\cdot,\tau)\|_{\dot{\mathbf H}_{\#\sigma}^{n/2}}
\le\|{\widetilde{\mathbf w}}(\cdot,\tau)\otimes {\widetilde{\mathbf w}}(\cdot,\tau)]\|_{({ H}_{\#}^{-n/2+1})^{n\times n}}
\|{\mathbf u}(\cdot,\tau)\|_{\dot{\mathbf H}_{\#\sigma}^{n/2}}
\\
\shoveright{
\le C_*\|{\widetilde{\mathbf w}}(\cdot,\tau)\|^2_{\dot{\mathbf H}_{\#\sigma}^{1/2}}
\|{\mathbf u}(\cdot,\tau)\|_{\dot{\mathbf H}_{\#\sigma}^{n/2}}
\le C_*\|{\widetilde{\mathbf w}}(\cdot,\tau)\|_{\dot{\mathbf H}_{\#\sigma}^{0}}
\|{\widetilde{\mathbf w}}(\cdot,\tau)\|_{\dot{\mathbf H}_{\#\sigma}^{1}}
\|{\mathbf u}(\cdot,\tau)\|_{\dot{\mathbf H}_{\#\sigma}^{n/2}}
}
\\
\le \frac{1}{4}C_{\mathbb A}^{-1}\|{\widetilde{\mathbf w}}(\cdot,\tau)\|^2_{\dot{\mathbf H}_{\#\sigma}^{1}}
+C_{\mathbb A}C_*^2\|{\widetilde{\mathbf w}}(\cdot,\tau)\|^2_{\dot{\mathbf H}_{\#\sigma}^{0}}
\|{\mathbf u}(\cdot,\tau)\|^2_{\dot{\mathbf H}_{\#\sigma}^{n/2}},
\end{multline}
where $C_*:=C_*(1/2,1/2,n)$ if from Theorem \ref{RS-T1-S4.6.1}(b). 
Implementing \eqref{E5.19} in \eqref{E5.18}, we obtain
\begin{align}\label{E4.20}
\frac12\left\|\widetilde{\mathbf w}(\cdot,t)\right\|^2_{\dot{\mathbf H}_{\#\sigma}^{0}}
\le C_{\mathbb A}C_*^2\int_0^t\|{\widetilde{\mathbf w}}(\cdot,\tau)\|^2_{\dot{\mathbf H}_{\#\sigma}^{0}}
\|{\mathbf u}(\cdot,\tau)\|^2_{\dot{\mathbf H}_{\#\sigma}^{n/2}}d\tau.
\end{align}
Since
$$
\int_0^T\|{\widetilde{\mathbf w}}(\cdot,\tau)\|^2_{\dot{\mathbf H}_{\#\sigma}^{0}}
\|{\mathbf u}(\cdot,\tau)\|^2_{\dot{\mathbf H}_{\#\sigma}^{n/2}}d\tau
\le \|{\widetilde{\mathbf w}}\|^2_{L_\infty(0,T;\dot{\mathbf H}_{\#\sigma}^{0})}
\|\mathbf u\|^2_{L_2(0,T;\dot{\mathbf H}_{\#\sigma}^{n/2})}<\infty,
$$
we can employ to \eqref{E4.20} the integral Gronwal's inequality from Lemma \ref{IGTL} to conclude that 
$\|{\widetilde{\mathbf w}}(\cdot,\tau)\|_{\dot{\mathbf H}_{\#\sigma}^{0}}=0$.
\end{proof}

\subsection{Serrin-type property of the two-dimensional weak solution}
By Definitions \ref{D6.1} and \ref{CSDef}, any week solution of the evolution space-periodic anisotropic Navier-Stokes initial value problem \eqref{NS-problem-div0}--\eqref{NS-problem-div0-IC} is a Serrin-type solution for $n=2$.
Then Lemmas \ref{L4.2} and \ref{L5.2} along with Theorems \ref{Th4.4} and \ref{Th4.5} lead to the following results for any $T>0$ and arbitrarily large data (unlike the higher dimensions discussed further on).

\begin{theorem}
\label{NS-problemTh-sigma2}
Let $n=2$, $T>0$, 
$a_{ij}^{\alpha \beta }\in L_\infty(0,T;L_{\infty\#})$
and the relaxed ellipticity condition \eqref{mu} hold.  
Let ${\mathbf f}\in L_2(0,T;\dot{\mathbf H}_\#^{-1})$, $\mathbf u^0\in \dot{\mathbf H}_{\#\sigma}^{0}$.

Then the solution $\mathbf u\in L_{\infty}(0,T_*;\dot{\mathbf H}_{\#\sigma}^{0})\cap L_2(0,T_*;\dot{\mathbf H}_{\#\sigma}^{1})$ of the anisotropic Navier-Stokes 
initial value problem  \eqref{NS-problem-div0}--\eqref{NS-problem-div0-IC} obtained in Theorem \ref{NS-problemTh-sigma} is of  Serrin-type and hence $({\mathbf u}\cdot \nabla ){\mathbf u}\in{L_2(0,T;\mathbf{H}_\#^{-1})}$,  
${\mathbf u}\in W^1(\dot{\mathbf H}_{\#\sigma}^{1},\dot{\mathbf H}_{\#\sigma}^{-1})$,
$\mathbf u$ is almost everywhere on $[0,T]$ equal to a function belonging to  $\mathcal C^0([0,T];\dot{\mathbf H}_{\#\sigma}^{0})$ and
$$
\lim_{t\to 0}\| \mathbf u(\cdot,t)-{\mathbf u}^0\| _{\dot{\mathbf H}_{\#\sigma}^{0}}= 0.
$$
In addition,
\begin{align*}
\left\langle {\mathbf u}'(\cdot,t),{\mathbf w}\right\rangle _{\T }
+\left\langle({\mathbf u}(\cdot,t)\cdot \nabla ){\mathbf u}(\cdot,t),{\mathbf w}\right\rangle _{\T }
+a_{\T}(t;{\mathbf u},{\mathbf w})
=\langle\mathbf{f}(\cdot,t), \mathbf w\rangle_{\T},\ 
\text{for a.e. } t\in(0,T),\
\forall\,\mathbf w\in \dot{\mathbf H}_{\#\sigma}^{1},
\end{align*}
and the following energy equality holds,
\begin{align*}
\frac12\|{\mathbf u}(\cdot,t)\| _{\mathbf L_{2\#}}^2
+\int_0^ta_{\T}(\tau;{\mathbf u}(\cdot,\tau),{\mathbf u}(\cdot,\tau)) d\tau
= \frac12\|{\mathbf u}^0\| ^2_{\mathbf L_{2\#}}
+\int_0^t\langle\mathbf{f}(\cdot,\tau), {\mathbf u}(\cdot,\tau)\rangle_{\T}\,d\tau
\quad \forall\,
t\in[0,T].
\end{align*}
Moreover, the solution $\mathbf u$ is unique in the class of solutions from
$ L_{\infty}(0,T;\dot{\mathbf H}_{\#\sigma}^{0})\cap L_2(0,T;\dot{\mathbf H}_{\#\sigma}^{1})$
satisfying the energy inequality \eqref{E4.9}. 
The unique pressure  $p$ associated with the obtained $\mathbf u$ belongs to $L_2(0,T;\dot{H}_{\#}^{0})$. 
\end{theorem}

\section{Serrin-type solution existence and regularity for constant anisotropic viscosity coefficients}\label{S5ERC}
In this section we analyse the existence and regularity of Serrin-type solutions for any $n\ge 2$ in the anisotropic variable-coefficient case.
This gives a generalisation of Theorem 10.1 in \cite{RRS2016}, where similar results were obtained for $n=3$, for the smoothness index $r=1/2$, and for the isotropic constant viscosity coefficients.

\subsection{Vector heat equation}
Let us first consider the spatially periodic Cauchy problem for the (vector) heat equation, 
\begin{align}
\label{Heat}
\partial_t\mathbf v -\Delta{\mathbf v}=\mathbf{0}
\quad& \mbox{in } \T\times(0,\infty), 
\\
\label{Heat-IC}
\mathbf v(\cdot,0)=\mathbf u^0\quad& \mbox{in } \T .
\end{align}
Its periodic solution can be written as
\begin{align}\label{K-def}
\mathbf v(\mathbf x, t)=(K{\mathbf u}^0)(\mathbf x, t):=\sum_{\bs\xi\in\Z^n}\widehat{\mathbf u}^0(\bs\xi)e^{-(2\pi |\bs\xi|)^2t+2\pi i \mathbf x\cdot\bs\xi}.
\end{align}

If $\div\,{\mathbf u}^0(\mathbf x)=0$ and $\widehat{\mathbf u}^0(\mathbf 0)=\mathbf 0$ then 
$\div\,{\mathbf v}(\mathbf x,t)=0$ and $\widehat{\mathbf v}(\mathbf 0,t)=\mathbf 0$.
Particularly, let us assume that $\mathbf u^0\in \dot{\mathbf H}_{\#\sigma}^{r}$  for some $r\in\R$. Then taking dual product of the both sides of equation \eqref{Heat} with $\Lambda^{2r}_\#{\mathbf v}$ gives
\begin{align*}
&\langle \partial_t\Lambda^{r}_\#{\mathbf v},\Lambda^{r}_\#{\mathbf v}\rangle _{\T }
+\langle{\nabla\Lambda^{r}_\#\mathbf v},\nabla\Lambda^{r}_\#{\bf v}\rangle_\T
=\mathbf 0,\ 
\end{align*}
implying
\begin{align*}
\frac12\frac{d}{d t}\left\|{\mathbf v}\right\|_{\mathbf H^{r}_{\#}}^2
+ \|\nabla{\mathbf v}\|_{({H}_\#^{r})^{n\times n}}^2=0.
\end{align*}
After integration this gives the energy-type equality
\begin{align}\label{E5.4-1}
\frac12\left\|{\mathbf v}(\cdot,t)\right\|_{\mathbf H^{r}_{\#}}^2
+\int_0^t \|\nabla{\mathbf v}\|_{({H}_\#^{r})^{n\times n}}^2 d\tau
=\frac12\left\|{\mathbf u}^0\right\|_{\mathbf H^{r}_{\#}}^2, \quad t\ge 0.
\end{align}
The solution representation \eqref{K-def} and the norm definition \eqref{eq:mik10} imply that 
\begin{align}\label{E4.121}
&\left\|{\mathbf v}(\cdot,t)\right\|_{\mathbf H^{r}_{\#}}^2
\le \left\|{\mathbf u}^0\right\|_{\mathbf H^{r}_{\#}}^2 \quad t\ge 0,
\end{align}
On the other hand, \eqref{eq:mik14} and \eqref{E5.4-1} lead to
\begin{align}
\label{E4.121a}
&\int_0^t\left\|K{\mathbf u}^0\right\|_{\dot{\mathbf H}^{r+1}_{\#}}^2 d\tau
=\int_0^t \|\mathbf v\|_{\dot{\mathbf H}_\#^{r+1}}^2 d\tau
\le 2\int_0^t \|\nabla{\mathbf v}\|_{({H}_\#^{r})^{n\times n}}^2 d\tau
\le \left\|{\mathbf u}^0\right\|_{\mathbf H^{r}_{\#}}^2, \quad t\ge 0.
\end{align}
Estimates \eqref{E4.121} and \eqref{E4.121a} 
mean that ${\mathbf v}\in{L_\infty(0,T;\dot{\mathbf H}^{r}_{\#\sigma})}\cup {L_2(0,T;\dot{\mathbf{H}}_{\#\sigma}^{r+1})}$ 
for any $T> 0$ 
and
\begin{align}\label{E4.122v}
&\left\|{\mathbf v}\right\|_{L_\infty(0,T;\mathbf H^{r}_{\#})}^2 
=\left\|K{\mathbf u}^0\right\|_{L_\infty(0,T;\mathbf H^{r}_{\#})}^2
\le \left\|{\mathbf u}^0\right\|_{\mathbf H^{r}_{\#}}^2, 
\\
\label{E4.122va}
&\|{\mathbf v}\|_{L_2(0,T;\mathbf{H}_\#^{r+1})}^2
=\left\|K{\mathbf u}^0\right\|_{L_2(0,T;\mathbf H^{r+1}_{\#})}^2
\le \left\|{\mathbf u}^0\right\|_{\mathbf H^{r}_{\#}}^2, \quad \forall\,T> 0.
\end{align}
This implies that 
the operator $K: \dot{\mathbf H}^{r}_{\#\sigma}\to L_\infty(0,\infty;\dot{\mathbf H}^{r}_{\#\sigma})\cup L_2(0,\infty;\dot{\mathbf{H}}_{\#\sigma}^{r+1})$ is continuous.

\subsection{Preliminary results 
for constant anisotropic viscosity coefficients}\label{S.5.2C}

For some  $n\ge 2$, $r\ge{n}/{2}-1$, and $T>0$, let  the coefficients 
$a_{ij}^{\alpha\beta}$ be constant and the relaxed ellipticity condition \eqref{mu} hold.
Let  also ${\mathbf f}\in L_2(0,T;\dot{\mathbf H}_\#^{r-1})$  and $\mathbf u^0\in \dot{\mathbf H}_{\#\sigma}^{r}$.

Let us employ, as usual,  the Galerkin approximation, with the sequence 
$\{\mathbf w_\ell\}\subset \dot{\mathbf C}^\infty_{\#\sigma}$ of eigenfunctions of the Bessel potential operator $\Lambda_\#$ in $\dot{\mathbf H}^0_{\#\sigma}$, corresponding to eigenvalues $\lambda_\ell$ and constituting an orthonormal basis in  $\dot{\mathbf H}_{\#\sigma}^{0}$, see Section \ref{BPOS}. 
They also constitute an orthogonal basis in $\dot{\mathbf H}_{\#\sigma}^{r}$ and $\dot{\mathbf H}_{\#\sigma}^{r+1}$, see Theorem \ref{TE.1r}.
Let us construct the  $m$-term approximation to $\mathbf u^0$, 
$$\mathbf u_m^0:=P_m\mathbf u_m^0=\sum_{\ell=1}^m \langle\mathbf u^0,\mathbf w_\ell\rangle_\T \mathbf w_\ell,$$ 
where $P_m$ is the orthogonal projector  from ${\mathbf H}^{r}_{\#\sigma}$ to ${\rm Span}\{\mathbf w_1,\ldots,\mathbf w_m\}$, 
cf. \eqref{m-Projector}, that converges in 
$\dot{\mathbf H}_{\#\sigma}^{0}$, $\dot{\mathbf H}_{\#\sigma}^{r}$ and 
$\dot{\mathbf H}_{\#\sigma}^{r+1}$ as $m\to\infty$. 
Due to the basis orthogonality,  we have the inequalities
\begin{align*}
\|\mathbf u^0_m\|_{\dot{\mathbf H}_{\#\sigma}^{0}}\le \|\mathbf u^0\|_{\dot{\mathbf H}_{\#\sigma}^{0}},\quad
\|\mathbf u^0_m\|_{\dot{\mathbf H}_{\#\sigma}^{r}}\le \|\mathbf u^0\|_{\dot{\mathbf H}_{\#\sigma}^{r}},\quad
\|\mathbf u^0_m\|_{\dot{\mathbf H}_{\#\sigma}^{r+1}}\le \|\mathbf u^0\|_{\dot{\mathbf H}_{\#\sigma}^{r+1}}.
\end{align*}
Let $\{\mathbf u_m\}$ be the sequence employed to prove Theorem 2 in \cite{Mikhailov2024}, given here as Theorem \ref{NS-problemTh-sigma}. 
The sequence $\{\mathbf u_m\}$ converges to a solution 
$\mathbf u \in L_\infty(0,T;\dot{\mathbf H}_{\#\sigma}^{0})\cap L_2(0,T;\dot{\mathbf H}_{\#\sigma}^{1})$ of the initial-variational problem \eqref{NS-eq:mik51a}--\eqref{NS-eq:mik51in} weakly in $L_2(0,T;\dot{\mathbf H}_{\#\sigma}^{1})$, weakly-star in $L_\infty(0,T;\dot{\mathbf H}_{\#\sigma}^{0})$ and strongly in $L_2(0,T;\dot{\mathbf H}_{\#\sigma}^{0})$. 
Particularly,
$
\mathbf u_m(\mathbf x,t)=\sum_{\ell=1}^m\eta_{\ell,m}(t)\mathbf w_\ell
$
and solves the following nonlinear ODE problem 
 from Theorem 2 in \cite{Mikhailov2024},
\begin{align}
\label{NS-eq:mik51adTv}
&\langle \partial_t{\mathbf u}_m,{\mathbf w}_k\rangle _{\T }
+a_{\T}(t;{\mathbf u}_m,{\bf w}_k)
+\langle({\mathbf u}_m\cdot \nabla ){\mathbf u}_m,{\bf w}_k\rangle _{\T}
=\langle\mathbf{f}, \mathbf w_k\rangle_{\T},\ \text{a.e. } t\in(0,T),\ \forall\, k\in \{1,\ldots,m\},
\\
\label{NS-eq:mik51indTv}
&\langle\mathbf u_m, \mathbf w_k\rangle_{\T}(\cdot,0)=\langle\mathbf u^0, \mathbf w_k\rangle_{\T},\quad \forall\, k\in \{1,\ldots,m\}.
\end{align}
Similarly to the proof of Theorem 10.1 in \cite{RRS2016}, 
let us define
$
\mathbf v_m(\mathbf x,t):=P_m \mathbf v=\sum_{k=1}^m\langle\mathbf v,\mathbf w_k\rangle_\T\mathbf w_k
$
for $\mathbf v$ given by \eqref{K-def}.
Acting by the projector $P_m$ on \eqref{Heat}--\eqref{Heat-IC} and then taking the dual product with $\mathbf w_k$, we obtain that for any $m>1$, $\mathbf v_m$
solves the initial value ODE problem
\begin{align}
\label{E4.113vm}
&\langle \partial_t{\mathbf v}_m,{\mathbf w}_k\rangle _{\T }
+\langle{\nabla\mathbf v}_m,\nabla{\bf w}_k\rangle_\T
=\mathbf 0,\ \forall\, t\in(0,T),\ \forall\, k\in \{1,\ldots,m\},
\\
\label{E4.114v}
&\langle\mathbf v_m, \mathbf w_k\rangle_{\T}(\cdot,0)=\langle\mathbf u^0, \mathbf w_k\rangle_{\T},\quad \forall\, k\in \{1,\ldots,m\},
\end{align}
and by \eqref{E4.122v} satisfies the estimates
\begin{align}
\label{E4.122v-mC}
&\left\|{\mathbf v}_m\right\|_{L_\infty(0,T;\mathbf H^{r}_{\#})}^2
\le\left\|{\mathbf v}\right\|_{L_\infty(0,T;\mathbf H^{r}_{\#})}^2
\le  \left\|{\mathbf u}^0\right\|_{\mathbf H^{r}_{\#}}^2,\quad \forall\,T> 0,
\\
\label{E4.122va-mC}
&\|{\mathbf v_m}\|_{L_2(0,T;\mathbf{H}_\#^{r+1})}^2
\le\|{\mathbf v}\|_{L_2(0,T;\mathbf{H}_\#^{r+1})}^2
\le\left\|{\mathbf u}^0\right\|_{\mathbf H^{r}_{\#}}^2, \quad \forall\,T> 0.
\end{align}

To reduce the problem \eqref{NS-eq:mik51adTv}-\eqref{NS-eq:mik51indTv} to the one with zero initial conditions, let us represent $\mathbf u_m=\mathbf v_m+\widetilde{\mathbf u}_m$.
Then due to \eqref{NS-eq:mik51adTv}--\eqref{NS-eq:mik51indTv}, the auxiliary function 
$
\widetilde{\mathbf u}_m(\mathbf x,t)=\sum_{\ell=1}^m\widetilde{\eta}_{\ell,m}(t)\mathbf w_\ell
$
 satisfies the ODE problem
\begin{align}
\label{E4.115var}
&\langle \partial_t\widetilde{\mathbf u}_m,{\mathbf w}_k\rangle _{\T }
+a_{\T}(t;\widetilde{\mathbf u}_m,{\bf w}_k)
+\langle({\mathbf u}_m\cdot \nabla ){\mathbf u}_m,{\bf w}_k\rangle _{\T}
\nonumber\\
&\hspace{10em}
=\langle\mathbf{f}, \mathbf w_k\rangle_{\T}
+\langle{\nabla\mathbf v}_m,\nabla{\bf w}_k\rangle_\T
-a_{\T}(t;{\mathbf v}_m,{\bf w}_k)
,\ \forall\, k\in \{1,\ldots,m\},
\\
\label{E4.116var}
&\langle\widetilde{\mathbf u}_m, \mathbf w_k\rangle_{\T}(\cdot,0)=\mathbf 0,\quad \forall\, k\in \{1,\ldots,m\}.
\end{align}
After multiplying by $\lambda^{2r}_k$ and taking into account the property $\Lambda^{2r}_\# \mathbf w_k=\lambda^{2r}_k \mathbf w_k$, relation  \eqref{NS-a-vsigma}, and that the operator $\Lambda^{r}_\#$ commutate with operators $\nabla$ and $E_{j\beta }$, equation \eqref{E4.115var} leads to
\begin{multline}\label{E4.93Tvar0C}
\langle\partial_t\widetilde{\boldsymbol u}_m, \Lambda^{2r}_\#\mathbf{w}_k\rangle_\T
+\left\langle E_{j\beta }(\widetilde{\mathbf u}_m),
a_{ij}^{\alpha \beta }\Lambda^{r}_\#E_{i\alpha }(\Lambda^{r}_\#{\mathbf w}_k)\right\rangle _{\T}
+\langle({\mathbf u}_m\cdot \nabla ){\mathbf u}_m,\Lambda^{2r}_\#{\bf w}_k\rangle _{\T}
\\
\hspace{1em}
=\langle\mathbf f, \Lambda^{2r}_\# \mathbf{w}_k\rangle_\T
+\langle{\nabla\mathbf v}_m,\Lambda^{r}_\#\nabla\Lambda^{r}_\#{\bf w}_k\rangle_\T
-\left\langle E_{j\beta }({\bf v}_m),
a_{ij}^{\alpha \beta }\Lambda^{r}_\#E_{i\alpha }(\Lambda^{r}_\#{\mathbf w}_k)\right\rangle _{\T},
\quad \forall\, k\in \{1,\ldots,m\} .
\end{multline}
These equations can be re-written as
\begin{multline}\label{E4.93TvarC}
\langle\partial_t\Lambda^{r}_\#\widetilde{\boldsymbol u}_m, \Lambda^{r}_\#\mathbf{w}_k\rangle_\T
+\left\langle a_{ij}^{\alpha \beta }E_{j\beta }(\Lambda^{r}_\#\widetilde{\mathbf u}_m),
E_{i\alpha }(\Lambda^{r}_\#{\mathbf w}_k)\right\rangle _{\T}
+\langle\Lambda^{r-1}_\#[({\mathbf u}_m\cdot \nabla ){\mathbf u}_m],\Lambda^{r+1}_\#{\bf w}_k\rangle _{\T}
\\
=\langle\Lambda^{r-1}_\#\mathbf f, \Lambda^{r+1}_\#  \mathbf{w}_k\rangle_\T
+\langle{\nabla\Lambda^{r}_\#\mathbf v}_m,\nabla\Lambda^{r}_\#{\bf w}_k\rangle_\T
-\left\langle E_{j\beta }({\bf v}_m),
a_{ij}^{\alpha \beta }\Lambda^{r}_\#E_{i\alpha }(\Lambda^{r}_\#{\mathbf w}_k)\right\rangle _{\T}
\quad \forall\, k\in \{1,\ldots,m\} .
\end{multline}
Multiplying equations in \eqref{E4.93TvarC} by $\widetilde{\eta}_{k,m}(t)$ and summing them up over $k\in \{1,\ldots,m\}$, we obtain
\begin{multline}\label{E4.88TvarC}
\frac{1}{2} \partial_t\left\|\Lambda^{r}_\#\widetilde{\mathbf u}_m\right\|_{\mathbf H^{0}_{\#}}^2
+a_{\T}(\Lambda^{r}_\#\widetilde{\mathbf u}_m,\Lambda^{r}_\#\widetilde{\mathbf u}_m)
=\langle\Lambda^{r-1}_\#\mathbf f, \Lambda^{r+1}_\# \widetilde{\mathbf u}_m\rangle_\T
+\langle{\nabla\Lambda^{r}_\#\mathbf v}_m,\nabla\Lambda^{r}_\#\widetilde{\mathbf u}_m\rangle_\T
\\
-\left\langle E_{j\beta }({\bf v}_m),
a_{ij}^{\alpha \beta }\Lambda^{r}_\#E_{i\alpha }(\Lambda^{r}_\#\widetilde{\mathbf u}_m)\right\rangle _{\T}
-\langle\Lambda^{r-1}_\#[({\mathbf u}_m\cdot \nabla ){\mathbf u}_m],\Lambda^{r+1}_\#\widetilde{\mathbf u}_m\rangle _{\T}.
\end{multline}

From \eqref{NS-a-1-v2-S-} we have
\begin{align}
\label{NS-a-1-v2-S-TvarC}
a_{\T}(\Lambda^{r}_\#\widetilde{\mathbf u}_m,\Lambda^{r}_\#\widetilde{\mathbf u}_m)
&\geq \frac14 C_{\mathbb A}^{-1}\|\Lambda^{r}_\#\widetilde{\mathbf u}_m\|_{\dot{\mathbf H}_{\#\sigma}^1}^2
=\frac14 C_{\mathbb A}^{-1}\|\widetilde{\mathbf u}_m\|^2_{\dot{\mathbf H}_{\#\sigma}^{r+1}}.
\end{align}

Let us now estimate the terms in the right hand side of \eqref{E4.88TvarC}.
First,
\begin{align}\label{E4.123cvar}
\langle\Lambda^{r-1}_\#\mathbf f, \Lambda^{r+1}_\# \widetilde{\mathbf u}_m\rangle_\T
\le \|\Lambda^{r-1}_\#\mathbf f\|_ {\dot{\mathbf H}_{\#}^{0}} \|\Lambda^{r+1}_\# \widetilde{\mathbf u}_m\|_{\dot{\mathbf H}_{\#}^{0}}
\le \|\mathbf f\|_ {\dot{\mathbf H}_{\#}^{r-1}} \|\widetilde{\mathbf u}_m\|_{\dot{\mathbf H}_{\#}^{r+1}}.
\end{align}
Next, inequality \eqref{eq:mik14} implies 
\begin{align}\label{E4.121cvar}
\langle\nabla\Lambda^{r}_\#\mathbf v_m,\nabla\Lambda^{r}_\#\widetilde{\mathbf u}_m\rangle_\T
&\le \|\nabla\Lambda^{r}_\# \mathbf v_m\|_{L_{2\#}^{n\times n}} \|\nabla\Lambda^{r}_\#\widetilde{\mathbf u}_m\|_{L_{2\#}^{n\times n}}
\nonumber\\
&\le \|\Lambda^{r}_\# \mathbf v_m\|_{\dot{\mathbf H}_{\#}^1} \|\Lambda^{r}_\#\widetilde{\mathbf u}_m\|_{\dot{\mathbf H}_{\#}^1}
\le \|\mathbf v_m\|_{\dot{\mathbf H}_{\#}^{r+1}} \|\widetilde{\mathbf u}_m\|_{\dot{\mathbf H}_{\#}^{r+1}}.
\end{align}
Further, since  $\tilde\sigma+1>n/2$, we obtain by Theorem \ref{RS-T1-S4.6.1}(a), relation \eqref{strain-r}, and inequality \eqref{eq:mik14},
\begin{multline}\label{E4.122varC}
\left|\left\langle E_{j\beta }({\bf v}_m),
a_{ij}^{\alpha \beta }\Lambda^{r}_\#E_{i\alpha }(\Lambda^{r}_\#\widetilde{\mathbf u}_m)\right\rangle _{\T}\right|
\le \|E_{j\beta }({\mathbf v}_m)\|_{({H}_{\#}^{r})^{n\times n}}  
\|a_{ij}^{\alpha \beta }\Lambda^{r}_\#E_{i\alpha }(\Lambda^{r}_\#\widetilde{\mathbf u}_m)\|_{({H}_{\#}^{-r})^{n\times n}}
\\
\le  \|\mathbb A\| 
\|\widetilde{\mathbf u}_m\|_{\dot{\mathbf H}_{\#}^{r+1}},
\end{multline}
where 
$\|\mathbb A\|
:= \left\|\left\{a_{ij}^{\alpha \beta }\right\}_{\alpha,\beta,i,j=1}^n\right\|_{F}$.
Finally,
\begin{multline}\label{E4.124C}
\left|\langle\Lambda^{r-1}_\#[({\mathbf u}_m\cdot \nabla ){\mathbf u}_m],\Lambda^{r+1}_\#\widetilde{\mathbf u}_m\rangle _{\T}\right|
\le \|\Lambda^{r-1}_\#[({\mathbf u}_m\cdot \nabla ){\mathbf u}_m]\|_{\dot{\mathbf H}_{\#}^{0}}
 \|\Lambda^{r+1}_\#\widetilde{\mathbf u}_m\|_{\dot{\mathbf H}_{\#}^{0}}
\\
\le \|({\mathbf u}_m\cdot \nabla ){\mathbf u}_m\|_{\dot{\mathbf H}_{\#}^{r-1}} \|\widetilde{\mathbf u}_m\|_{\dot{\mathbf H}_{\#}^{r+1}}.
\end{multline}

Implementing \eqref{NS-a-1-v2-S-TvarC}--\eqref{E4.124C} in \eqref{E4.88TvarC} and using Young's inequality, we obtain
\begin{multline*}
\frac{d}{d t}\left\|\widetilde{\mathbf u}_m\right\|_{\mathbf H^{r}_{\#}}^2
+\frac12 C_{\mathbb A}^{-1}\|\widetilde{\mathbf u}_m\|_{\mathbf{H}_\#^{r+1}}^2
\leq 2\Big(\|\mathbf f\|_{\mathbf{H}_\#^{r-1}}
+\big[\|\mathbb A\|+1\big] \|\mathbf v_m\|_{\dot{\mathbf H}_{\#}^{r+1}}
+\left\|({\mathbf u}_m\cdot \nabla ){\mathbf u}_m\right\|_{\mathbf{H}_\#^{r-1}}
\Big) \| \widetilde{\mathbf u}_m\|_{\mathbf{H}_\#^{r+1}}
\\
\leq 4 C_{\mathbb A}\Big(\|\mathbf f\|_{\mathbf{H}_\#^{r-1}}
+\big[\|\mathbb A\|+1\big]  \|\mathbf v_m\|_{\dot{\mathbf H}_{\#}^{r+1}}
+\left\|({\mathbf u}_m\cdot \nabla ){\mathbf u}_m\right\|_{\mathbf{H}_\#^{r-1}} 
\Big)^2
+  \frac14 C_{\mathbb A}^{-1}\| \widetilde{\mathbf u}_m\|_{\mathbf{H}_\#^{r+1}}^2.
\end{multline*}
Hence by the inequality $(\sum_{i=1}^k a_i)^2\le k\sum_{i=1}^k a_i^2$ (following from the Cauchy–Schwarz inequality), \begin{align}\label{E4.82TC}
\frac{d}{d t}\left\|\widetilde{\mathbf u}_m\right\|_{\mathbf H^{r}_{\#}}^2
+\frac14 C_{\mathbb A}^{-1}\|\widetilde{\mathbf u}_m\|_{\mathbf{H}_\#^{r+1}}^2
\leq 16 C_{\mathbb A}\Big(\|\mathbf f\|_{\mathbf{H}_\#^{r-1}}^2
+ \big[\|\mathbb A\|^2+1\big]\|\mathbf v_m\|_{\dot{\mathbf H}_{\#}^{r+1}}^2
+\left\|({\mathbf u}_m\cdot \nabla ){\mathbf u}_m\right\|_{\mathbf{H}_\#^{r-1}}^2
\Big).
\end{align}

Note that by the similar reasoning, but without employing in \eqref{NS-eq:mik51adTv}--\eqref{NS-eq:mik51indTv} the function $\mathbf v$, we obtain that $\mathbf u_m$ satisfies the differential inequality
\begin{align}\label{E4.82T0C}
\frac{d}{d t}\left\|\mathbf u_m\right\|_{\mathbf H^{r}_{\#}}^2
+\frac14 C_{\mathbb A}^{-1}\|\mathbf u_m\|_{\mathbf{H}_\#^{r+1}}^2
\leq 8 C_{\mathbb A}\Big(\|\mathbf f\|_{\mathbf{H}_\#^{r-1}}^2
+\left\|({\mathbf u}_m\cdot \nabla ){\mathbf u}_m\right\|_{\mathbf{H}_\#^{r-1}}^2
\Big).
\end{align}

Let us also estimate the last term in \eqref{E4.82TC} and \eqref{E4.82T0C} for the case $n/2-1\le r<n/2$. 
By relation  \eqref{eq:mik14},  the multiplication Theorem \ref{RS-T1-S4.6.1}(b), and the Sobolev interpolation inequality \eqref{SII2} we obtain,
\begin{align}\label{E4.91-0T0cC}
\left\|({\mathbf u}_m\cdot \nabla ){\mathbf u}_m\right\|_{\mathbf{H}_\#^{r-1}}^2
=\left\|\nabla\cdot({\mathbf u}_m \otimes{\mathbf u}_m)\right\|_{\mathbf{H}_\#^{r-1}}^2
&\le \left\|{\mathbf u}_m \otimes{\mathbf u}_m)\right\|_{({H}_\#^{r})^{n\times n}}^2
\nonumber
\\
&
\le C^2_{*rn} \|\mathbf{u}_m\|_{{\mathbf H}_\#^{r/2+n/4}}^4
\le C^2_{*rn}\|\mathbf{u}_m\|_{{\mathbf H}_\#^r}^2 \|\mathbf{u}_m\|^2_{{\mathbf H}_\#^{n/2}},
\end{align}
where $C_{*rn}=C_*(r/2+n/4,r/2+n/4, n)$.

\subsection{Serrin-type solution existence for constant anisotropic viscosity coefficients}\label{S5.3C}
Employing the results from Section \ref{S.5.2C} for $r=n/2-1$, we are now in the position to prove the existence of Serrin-type solutions.
\begin{theorem}
\label{NS-problemTh-sigma-Lv-crC}
Let $n\ge 2$, and $T>0$. 
Let the coefficients $a_{ij}^{\alpha\beta}$ be constant and the relaxed ellipticity condition \eqref{mu} hold.
Let ${\mathbf f}\in L_2(0,T;\dot{\mathbf H}_\#^{n/2-2})$  and $\mathbf u^0\in \dot{\mathbf H}_{\#\sigma}^{n/2-1}$.

(i) Then there  exist constants $A_{1}\ge 0$, $A_{2}\ge0$ and $A_{3}>0$ that are independent of ${\mathbf f}$ and $\mathbf u^0$ but may depend on $T$, $n$,  $\|\mathbb A\|$ and $C_{\mathbb A}$, such that
if ${\mathbf f}$, $\mathbf u^0$ and $T_*\in(0,T]$ satisfy the inequality
\begin{align}\label{E4.143T*varC}
\int_0^{T_*}\|\mathbf f(\cdot,t)\|_{\mathbf{H}_\#^{n/2-2}}^2dt
+ \left (A_{1}\|{\mathbf u}^0\|^2_{{\mathbf H}_\#^{n/2-1}} + A_{2}\right) 
\int_0^{T_*}\|(K{\mathbf u}^0)(\cdot,t)\|_{\dot{\mathbf H}_{\#}^{n/2}}^2dt
<A_{3},
\end{align}  
where $K$ is the operator defined in \eqref{K-def}, 
then there exists a solution $\mathbf u$ of the anisotropic Navier-Stokes 
initial value problem  \eqref{NS-problem-div0}--\eqref{NS-problem-div0-IC} in 
$L_{\infty}(0,T_*;\dot{\mathbf H}_{\#\sigma}^{n/2-1})\cap L_2(0,T_*;\dot{\mathbf H}_{\#\sigma}^{n/2})$, which is thus a Serrin-type solution.

(ii) In addition,
$
\mathbf u'\in L_2(0,T_*;\dot{\mathbf H}_{\#\sigma}^{n/2-2}),
$
$
\mathbf u \in\mathcal C^0([0,T_*];\dot{\mathbf H}_{\#\sigma}^{n/2-1})
$,
$
\lim_{t\to 0}\| \mathbf u(\cdot,t)-{\mathbf u}^0\| _{\dot{\mathbf H}_{\#\sigma}^{n/2-1}}= 0,
$
and $p\in L_2(0,T_*;\dot{H}_{\#}^{n/2-1})$. 

(iii) Moreover, $\mathbf u$ satisfies  the following energy equality for any $[t_0,t]\subset[0,T_*]$,
\begin{align}\label{NS-eq:mik51ai=aTvarC}
\frac12\| {\mathbf u}(\cdot,t)\| _{\mathbf L_{2\#}}^2
+\int_{t_0}^{t}a_{\T}({\mathbf u}(\cdot,\tau),{\mathbf u}(\cdot,\tau)) d\tau
=  \frac12\| {\mathbf u}(\cdot,t_0)\| ^2_{\mathbf L_{2\#}}
+\int_{t_0}^{t}\langle\mathbf{f}(\cdot,\tau), {\mathbf u}(\cdot,\tau)\rangle_{\T}d\tau.
\end{align}
It particularly implies the standard energy equality,
\begin{align}
\label{E4.9TvarC}
\frac12\ {\mathbf u}(\cdot,t)\| _{\mathbf L_{2\#}}^2
+\int_0^ta_{\T}({\mathbf u}(\cdot,\tau),{\mathbf u}(\cdot,\tau)) d\tau
= \frac12\|{\mathbf u}^0\| ^2_{\mathbf L_{2\#}}
+\int_0^t\langle\mathbf{f}(\cdot,\tau), {\mathbf u}(\cdot,\tau)\rangle_{\T}\,d\tau
\quad \forall\,
t\in[0,T_*].
\end{align}

(iv) The solution $\mathbf u$ is unique in the class of solutions from
$ L_{\infty}(0,T_*;\dot{\mathbf H}_{\#\sigma}^{0})\cap L_2(0,T_*;\dot{\mathbf H}_{\#\sigma}^{1})$
satisfying the energy inequality \eqref{E4.9} on the interval $[0,T_*]$. 
\end{theorem}
\begin{proof}
(i)
Let $r=n/2-1$.
Estimate \eqref{E4.91-0T0cC} implies
\begin{align}\label{E4.91-0TC}
\left\|({\mathbf u}_m\cdot \nabla ){\mathbf u}_m\right\|_{\mathbf{H}_\#^{r-1}}^2
&\le C^2_{*rn} \|\mathbf{u}_m\|_{{\mathbf H}_\#^{r/2+n/4}}^4
\le 8C^2_{*rn}(\|\widetilde{\mathbf u}_m\|^4_{{\mathbf H}_\#^{r/2+n/4}} + \|{\mathbf v}_m\|^4_{{\mathbf H}_\#^{r/2+n/4}})
\nonumber\\
&\le 8C^2_{*rn}\|\widetilde{\mathbf u}_m\|^2_{{\mathbf H}_\#^r} \|\widetilde{\mathbf u}_m\|^2_{{\mathbf H}_\#^{n/2}}
+ 8C^2_{*rn}\|{\mathbf v}_m\|^2_{{\mathbf H}_\#^r}\|{\mathbf v}_m\|^2_{{\mathbf H}_\#^{n/2}}
\nonumber\\
&\le 8C^2_{*rn}\|\widetilde{\mathbf u}_m\|^2_{{\mathbf H}_\#^{n/2-1}} \|\widetilde{\mathbf u}_m\|^2_{{\mathbf H}_\#^{n/2}}
+ 8C^2_{*rn}\|{\mathbf v}_m\|^2_{{\mathbf H}_\#^{n/2-1}}\|{\mathbf v}_m\|^2_{{\mathbf H}_\#^{n/2}},
\end{align}
where $C^2_{*rn}:=C_{*n/2-1,n}=C_*(n/2-1/2,n/2-1/2, n)$.
 Then by \eqref{E4.91-0TC} we obtain from \eqref{E4.82TC}, 
\begin{multline}\label{E4.93BTbC}
\frac{d}{d t}\left\|\widetilde{\mathbf u}_m\right\|_{\mathbf H^{n/2-1}_{\#}}^2
+\frac14 C_{\mathbb A}^{-1}\|\widetilde{\mathbf u}_m\|_{\mathbf{H}_\#^{n/2}}^2
\leq 
128C^2_{*rn}C_{\mathbb A}
\|\widetilde{\mathbf u}_m\|^2_{{\mathbf H}_\#^{n/2}}
\|\widetilde{\mathbf u}_m\|^2_{\dot{\mathbf H}_{\#}^{n/2-1}} 
\\
+ 16 C_{\mathbb A}\left(\|\mathbf f\|_{\mathbf{H}_\#^{n/2-2}}^2
+ 8C^2_{*rn}\|{\mathbf v}_m\|^2_{{\mathbf H}_\#^{n/2-1}}\|{\mathbf v}_m\|^2_{{\mathbf H}_\#^{n/2}}
+ \big[\|\mathbb A\|^2+1\big] \|\mathbf v_m\|_{\dot{\mathbf H}_{\#}^{n/2}}^2
\right).
\end{multline}
Let us now apply to \eqref{E4.93BTbC} Lemma \ref{RRS2016-L10.3alt} with 
\begin{align*}
&\eta=\left\|\widetilde{\mathbf u}_m\right\|_{\mathbf H^{n/2-1}_{\#}}^2,\
\eta_0=0,\
y=\|\widetilde{\mathbf u}_m\|_{\mathbf{H}_\#^{n/2}}^2,\ 
b=\frac14 C_{\mathbb A}^{-1},\ 
c=128C^2_{*rn}C_{\mathbb A},\
\\
&\psi=16 C_{\mathbb A}\left(\|\mathbf f\|_{\mathbf{H}_\#^{n/2-2}}^2
+ 8C^2_{*rn}\|{\mathbf v}_m\|^2_{{\mathbf H}_\#^{n/2-1}}\|{\mathbf v}_m\|^2_{{\mathbf H}_\#^{n/2}}
+\big[\|\mathbb A\|^2+1\big]\|\mathbf v_m\|_{\dot{\mathbf H}_{\#}^{n/2}}^2
\right)
\end{align*}
to conclude that if $T_*$ is such that 
\begin{multline}\label{E5.32bC}
\int_0^{T_*}\left(\|\mathbf f(\cdot,t)\|_{\mathbf{H}_\#^{n/2-2}}^2
+ \left(8C^2_{*rn}\|{\mathbf v}_m(\cdot,t)\|^2_{{\mathbf H}_\#^{n/2-1}}
+\big[\|\mathbb A\|^2+1\big]\right) 
\|\mathbf {\mathbf v}_m(\cdot,t)\|_{\dot{\mathbf H}_{\#}^{n/2}}^2\right)dt
\\
<\left(512 e C_{\mathbb A}^2C^2_{*rn}\right)^{-1},
\end{multline}
 then
\begin{align}
&\|{\mathbf u}_m\|_{L_\infty(0,T_*;\dot{\mathbf H}_{\#\sigma}^{n/2-1})}
\le\|\widetilde {\mathbf u}_m\|_{L_\infty(0,T_*;\dot{\mathbf H}_{\#\sigma}^{n/2-1})} 
+\|{\mathbf v}_m\|_{L_\infty(0,T_*;\dot{\mathbf H}_{\#\sigma}^{n/2-1})}
\nonumber\\
&\hspace{24.5em}\le \left(16\sqrt{2}C_{\mathbb A}{C}_{*rn}\right)^{-1}+\left\|{\mathbf u}^0\right\|_{\dot{\mathbf H}^{n/2-1}_{\#\sigma}},
\label{E4.143avarbC}\\
&\|{\mathbf u}_m\|_{L_2(0,T_*;\dot{\mathbf H}_{\#\sigma}^{n/2})}
\le\|\widetilde {\mathbf u}_m\|_{L_2(0,T_*;\dot{\mathbf H}_{\#\sigma}^{n/2})} 
+\|{\mathbf v}_m\|_{L_2(0,T_*;\dot{\mathbf H}_{\#\sigma}^{n/2})}
\le \left(8\sqrt{2C_{\mathbb A}}{C}_{*rn}\right)^{-1}
+\left\|{\mathbf u}^0\right\|_{\dot{\mathbf H}^{n/2-1}_{\#\sigma}}.
\label{E4.14varbC}
\end{align}
Estimates \eqref{E4.122v-mC} and \eqref{E4.122va-mC} were taken into account in \eqref{E4.143avarbC} and \eqref{E4.14varbC}.

Taking into account inequality \eqref{E4.122v-mC} again, we obtain that condition \eqref{E5.32bC} is satisfied  if $T_*$ is such that
\begin{align}\label{E4.143bC}
\int_0^{T_*}\|\mathbf f(\cdot,t)\|_{\mathbf{H}_\#^{n/2-2}}^2dt
+ \left (8C^2_{*rn}\|{\mathbf u}^0\|^2_{{\mathbf H}_\#^{n/2-1}}
+\big[\|\mathbb A\|^2+1\big]\right) \int_0^{T_*}\|\mathbf v(\cdot,t)\|_{\dot{\mathbf H}_{\#}^{n/2}}^2dt
<\left(512 e C_{\mathbb A}^2C^2_{*rn}\right)^{-1}.
\end{align}

Note that condition \eqref{E4.143bC} gives condition \eqref{E4.143T*varC} with 
\begin{align*}
&A_{1}=8C^2_{*rn}, \quad 
A_{2}=\|\mathbb A\|^2+1,\quad
A_{3}=\left(512 e C_{\mathbb A}^2C^2_{*rn}\right)^{-1}.
\end{align*}

Inequalities \eqref{E4.143avarbC} and \eqref{E4.14varbC} imply that  there exists a subsequence of $\{\mathbf u_m\}$ converging weakly  in $L_2(0, T_* ; \dot{\mathbf H}_{\#\sigma}^{n/2})$ and weakly-star in $L_\infty(0, T_* ; \dot{\mathbf H}_{\#\sigma}^{n/2-1})$ to a function 
${\mathbf u}^\dag\in L_2(0, T_* ; \dot{\mathbf H}_{\#\sigma}^{n/2})\cup L_\infty(0, T_* ; \dot{\mathbf H}_{\#\sigma}^{n/2-1})$.
Then the subsequence converges to ${\mathbf u}^\dag$ also weakly in $L_2(0,T_*;\dot{\mathbf H}_{\#\sigma}^{1})$ and weakly-star in $L_\infty(0,T_*;\dot{\mathbf H}_{\#\sigma}^{0})$.
Since $\{\mathbf u_m\}$ is the subsequence of the sequence that converges weakly in $L_2(0,T;\dot{\mathbf H}_{\#\sigma}^{1})$ and weakly-star in $L_\infty(0,T;\dot{\mathbf H}_{\#\sigma}^{0})$ to the weak solution, $\mathbf u$, of problem \eqref{NS-problem-div0}--\eqref{NS-problem-div0-IC} on $\left[0, T_*\right]$, we conclude that 
$\mathbf u={\mathbf u}^\dag\in L_\infty(0, T_* ; \dot{\mathbf H}_{\#\sigma}^{n/2-1})\cup L_2(0, T_* ; \dot{\mathbf H}_{\#\sigma}^{n/2})$. 

This implies that $\mathbf u$ is a Serrin-type solution on the interval $[0,T_*]$ and we thus proved item (i) of the theorem.

(ii) Repeating for $\mathbf u$ the reasoning related to inequality \eqref{E4.91-0T0cC}, we obtain
\begin{align*}
&\left\|({\mathbf u}\cdot \nabla ){\mathbf u}\right\|_{\mathbf{H}_\#^{n/2-2}}^2
\le  C^2_{*rn}\|{\mathbf u}\|^2_{{\mathbf H}_\#^{n/2-1}}\|{\mathbf u}\|^2_{{\mathbf H}_\#^{n/2}}.
\end{align*}
Hence 
\begin{align}\label{E4.91-0T0EE1r-uC1}
\left\|({\mathbf u}\cdot \nabla ){\mathbf u}\right\|_{L_2(0,T_*;\mathbf{H}_\#^{n/2-2})}
\le C_{*rn}\|\mathbf{u}\|_{L_\infty(0,T_*;{\mathbf H}_\#^{n/2-1})} \|\mathbf{u}\|_{L_2(0,T_*;{\mathbf H}_\#^{n/2})},
\end{align}
that is, $({\mathbf u}\cdot \nabla ){\mathbf u}\in{L_2(0,T_*;\dot{\mathbf H}_\#^{n/2-2})}$.
By \eqref{L-oper} and \eqref{TensNorm} we have
\begin{align*}
\left\|\bs{\mathfrak L}\mathbf u\right\|^2_{{\mathbf H}_{\#}^{n/2-2}}
\le\|a_{ij}^{\alpha \beta }E_{i\alpha }({\mathbf u})\|_{({H}_{\#}^{n/2-1})^{n\times n}}
\le \|\mathbb A\|^2\|\mathbf u\|^2_{{\mathbf H}_{\#}^{n/2}}
\end{align*}
and thus
\begin{align*}
\|\bs{\mathfrak L}\mathbf u\|^2_{L_2(0,T_*;\dot{\mathbf H}_{\#}^{n/2-2})}
&\le \|\mathbb A\|^2\|\mathbf u\|^2_{L_2(0,T_*;{\mathbf H}_{\#}^{n/2})},
\end{align*}
i.e., $\bs{\mathfrak L}\mathbf u\in{L_2(0,T_*;\dot{\mathbf H}_{\#\sigma}^{n/2-2})}$.
We also have ${\mathbf f}\in L_2(0,T;\dot{\mathbf H}_\#^{n/2-2})$.

Then \eqref{NS-problem-just} implies that 
${\mathbf u}'\in{L_2(0,T_*;\dot{\mathbf H}_{\#\sigma}^{n/2-2})}$ and hence by Theorem \ref{LM-T3.1} we obtain that 
$
\mathbf u \in\mathcal C^0([0,T_*];\dot{\mathbf H}_{\#\sigma}^{n/2-1})
$,
which also means 
that $
\| \mathbf u(\cdot,t)-{\mathbf u}^0\| _{\dot{\mathbf H}_{\#\sigma}^{n/2-1}}\to 0
$
as ${t\to 0}$.

To prove the theorem claim about the associated pressure $p$ we remark that it satisfies \eqref{Eq-p}, 
where  $\mathbf F\in L_2(0,T;\dot{\mathbf H}_{\#}^{n/2-2})$
due to the theorem conditions and the inclusion $({\mathbf u}\cdot \nabla ){\mathbf u}\in{L_2(0,T_*;\dot{\mathbf H}_\#^{n/2-2})}$.
By Lemma \ref{div-grad-is} for gradient, with $s=n/2-1$, equation \eqref{Eq-p} has a unique solution $p$ in 
$L_2(0,T_*;\dot{H}_{\#}^{n/2-1})$.

(iii) The energy equalities \eqref{NS-eq:mik51ai=aTvarC} and \eqref{E4.9TvarC} immediately follow from Theorem \ref{Th4.4}. 

(iv) The solution uniqueness follows from Theorem \ref{Th4.5}.
\end{proof}
 
\begin{remark}
Since $\|\mathbf f(\cdot,t)\|_{\dot{\mathbf H}_\#^{n/2-2}}^2$ is integrable on $(0,T]$ by the theorem condition and 
$\|(K{\mathbf u}^0)(\cdot,t)\|_{\dot{\mathbf H}_{\#}^{n/2}}^2$ is integrable on $(0,\infty)$ by the inequality \eqref{E4.121},  we conclude that due to the absolute continuity of the Lebesgue integrals, for arbitrarily large data ${\mathbf f}\in L_2(0,T;\dot{\mathbf H}_\#^{n/2-2})$  and $\mathbf u^0\in \dot{\mathbf H}_{\#\sigma}^{n/2-1}$ there exists $T_*>0$ such that condition \eqref{E4.143T*varC} holds.
\end{remark}

Estimating the integrand in the second integral in \eqref{E4.143T*varC} according to \eqref{E4.121}, we arrive at the following assertion allowing an explicit estimate of $T_*$ for arbitrarily large data if ${\mathbf f}\in L_\infty(0,T;\dot{\mathbf H}_\#^{n/2-2})$.
\begin{corollary}[Serrin-type solution  for arbitrarily large data but small time or vice versa.]
\label{NS-problemCor-sigma-2C}
Let $n\ge 2$ and $T>0$. 
Let the coefficients $a_{ij}^{\alpha\beta}$ be constant and the relaxed ellipticity condition hold.
Let ${\mathbf f}\in L_\infty(0,T;\dot{\mathbf H}_\#^{n/2-2})$  and $\mathbf u^0\in \dot{\mathbf H}_{\#\sigma}^{n/2}$.

Then there  exist constants $A_1, A_2, A_3>0$ that are independent of ${\mathbf f}$ and $\mathbf u^0$ but may depend on $T$, $n$,  $\|\mathbb A\|$ and $C_{\mathbb A}$, such that
if $T_*\in(0,T]$ satisfies the inequality
\begin{align}\label{E4.143T*C}
T_*\left[\|\mathbf f\|_{L_\infty(0,T;\dot{\mathbf H}_\#^{n/2-2})}^2
+ \left (A_1\|{\mathbf u}^0\|^2_{{\mathbf H}_\#^{n/2-1}} + A_2\right) \|{\mathbf u}^0\|^2_{{\mathbf H}_\#^{n/2}}\right]
<A_3,
\end{align}  
then there exists a Serrin-type solution  $\mathbf u\in L_{\infty}(0,T_*;\dot{\mathbf H}_{\#\sigma}^{n/2-1})\cap L_2(0,T_*;\dot{\mathbf H}_{\#\sigma}^{n/2})$  of the anisotropic Navier-Stokes  initial value problem.
This solution satisfies items (ii)-(iv) in Theorem \ref{NS-problemTh-sigma-Lv-crC}
\end{corollary}

Estimating the second integral in \eqref{E4.143T*varC} according to \eqref{E4.122va}, we arrive at the following assertion. 
\begin{corollary}[Existence of Serrin-type solution for arbitrary time but small data]
\label{NS-problemCor-sigma-1C}
Let $n\ge 2$ and $T>0$. 
Let the coefficients $a_{ij}^{\alpha\beta}$ be constant and the relaxed ellipticity condition hold.
Let ${\mathbf f}\in L_2(0,T;\dot{\mathbf H}_\#^{n/2-2})$  and $\mathbf u^0\in \dot{\mathbf H}_{\#\sigma}^{n/2-1}$.

Then there  exist constants $A_1, A_2, A_3>0$ that are independent of  ${\mathbf f}$ and $\mathbf u^0$ but may depend on $T$, $n$, $\|\mathbb A\|$ and $C_{\mathbb A}$, such that
if ${\mathbf f}$ and $\mathbf u^0$ satisfy the inequality
\begin{align}\label{E4.143TC}
\|\mathbf f\|_{L_2(0,T;\dot{\mathbf H}_\#^{n/2-2})}^2
+ \left (A_1\|{\mathbf u}^0\|^2_{{\mathbf H}_\#^{n/2-1}} + A_2\right) \|{\mathbf u}^0\|^2_{{\mathbf H}_\#^{n/2-1}}
<A_3,
\end{align}  
then there exists a Serrin-type solution  $\mathbf u\in L_{\infty}(0,T;\dot{\mathbf H}_{\#\sigma}^{n/2-1})\cap L_2(0,T;\dot{\mathbf H}_{\#\sigma}^{n/2})$  of the anisotropic Navier-Stokes 
initial value problem.
This solution satisfies items (ii)-(iv) in Theorem \ref{NS-problemTh-sigma-Lv-crC} with $T_*=T$.
\end{corollary}



\subsection{Spatial regularity of Serrin-type solutions for constant anisotropic viscosity coefficients}\label{S5.4C}

\begin{theorem}[Spatial regularity of Serrin-type solution for arbitrarily large data.]
\label{NS-problemTh-sigma-Lv-erC}
Let $n\ge 2$, $r>{n}/{2}-1$, and $T>0$. 
Let the coefficients $a_{ij}^{\alpha\beta}$ be constant and the relaxed ellipticity condition \eqref{mu} hold.
Let ${\mathbf f}\in L_2(0,T;\dot{\mathbf H}_\#^{r-1})$  and $\mathbf u^0\in \dot{\mathbf H}_{\#\sigma}^{r}$, while ${\mathbf f}$, $\mathbf u^0$ and $T_*\in(0,T]$ satisfy  inequality \eqref{E4.143T*varC} from Theorem \ref{NS-problemTh-sigma-Lv-crC}.

Then the Serrin-type solution $\mathbf u$ of the anisotropic Navier-Stokes 
initial value problem  \eqref{NS-problem-div0}--\eqref{NS-problem-div0-IC} 
belongs to
$L_{\infty}(0,T_*;\dot{\mathbf H}_{\#\sigma}^{r})\cap L_2(0,T_*;\dot{\mathbf H}_{\#\sigma}^{r+1})$.
In addition, 
$
\mathbf u'\in L_2(0,T_*;\dot{\mathbf H}_{\#\sigma}^{r-1}),
$
$
\mathbf u \in\mathcal C^0([0,T_*];\dot{\mathbf H}_{\#\sigma}^{r})
$,
$
\lim_{t\to 0}\| \mathbf u(\cdot,t)-{\mathbf u}^0\| _{\dot{\mathbf H}_{\#\sigma}^{r}}= 0
$
and $p\in L_2(0,T_*;\dot{H}_{\#}^{r})$. 
\end{theorem}
\begin{proof}
The existence of the Serrin-type solution $\mathbf u\in L_{\infty}(0,T_*;\dot{\mathbf H}_{\#\sigma}^{n/2-1})\cap L_2(0,T_*;\dot{\mathbf H}_{\#\sigma}^{n/2})$ is proved in Theorem \ref{NS-problemTh-sigma-Lv-crC}(i), and we will prove that it has a higher smoothness. 
We will employ the same Galerkin approximation used in Section \ref{S.5.2C} and in the proof of Theorem \ref{NS-problemTh-sigma-Lv-crC}(i).

\paragraph{Step (a).} 
Let us estimate the last term in  \eqref{E4.82T0C} for the case $n/2-1<r<n/2$. 
By \eqref{E4.91-0T0cC} we obtain from \eqref{E4.82T0C}, 
\begin{align}\label{E4.93BTcC}
\frac{d}{d t}\left\|\mathbf u_m\right\|_{\mathbf H^{r}_{\#}}^2
+\frac14 C_{\mathbb A}^{-1}\|\mathbf u_m\|_{\mathbf{H}_\#^{r+1}}^2
\leq 8 C_{\mathbb A}C^2_{*rn} \|\mathbf{u}_m\|^2_{{\mathbf H}_\#^{n/2}}\|\mathbf u_m\|^2_{\dot{\mathbf H}_{\#}^{r}}
+8 C_{\mathbb A}\|\mathbf f\|_{\mathbf{H}_\#^{r-1}}^2
\end{align}
implying
\begin{align}\label{E4.93BTc1C}
\frac{d}{d t}\left\|\mathbf u_m\right\|_{\mathbf H^{r}_{\#}}^2
\leq 8 C_{\mathbb A}C^2_{*rn} \|\mathbf{u}_m\|^2_{{\mathbf H}_\#^{n/2}}\|\mathbf u_m\|^2_{\dot{\mathbf H}_{\#}^{r}}
+8 C_{\mathbb A}\|\mathbf f\|_{\mathbf{H}_\#^{r-1}}^2.
\end{align}
By Gronwall's inequality \eqref{E17}, we obtain from \eqref{E4.93BTc1C} that
\begin{align}\label{E5.69C}
\|\mathbf u_m\|^2_{L_\infty(0,T_*;\dot{\mathbf H}_{\#\sigma}^{r})}
\le \exp\left(8 C_{\mathbb A}C^2_{*rn} \|\mathbf{u}_m\|^2_{L_2(0,T_*;\mathbf H_\#^{n/2})}\right)
 \left[\|\mathbf u_m(\cdot,0)\|_{\mathbf H^{r}_{\#}}^2
+8 C_{\mathbb A}\|\mathbf f\|_{L_2(0,T_*;\mathbf{H}_\#^{r-1})}^2\right].
\end{align}
We have $\|\mathbf u_m(\cdot,0)\|_{\mathbf H^{r}_{\#}}\le \|\mathbf u^0\|_{\mathbf H^{r}_{\#}}$ and
by \eqref{E4.14varbC}, the sequence $\|\mathbf{u}_m\|_{L_2(0,T_*;\mathbf H_\#^{n/2})}$ is bounded. 
Then \eqref{E5.69C} implies that the sequence $\|\mathbf u_m\|^2_{L_\infty(0,T_*;\dot{\mathbf H}_{\#\sigma}^{r})}$ is bounded as well.
Integrating \eqref{E4.93BTcC}, we conclude that 
\begin{multline}\label{E5.70C}
\|\mathbf u_m\|^2_{L_2(0,T_*;\dot{\mathbf H}_{\#\sigma}^{r+1})}
\le 32 C^2_{\mathbb A}C^2_{*rn} \|\mathbf{u}_m\|^2_{L_2(0,T_*;\mathbf H_\#^{n/2})}
\|\mathbf u_m\|^2_{L_\infty(0,T_*;\dot{\mathbf H}_{\#\sigma}^{r})}
\\
+4C_{\mathbb A}\|\mathbf u_m(\cdot,0)\|_{\mathbf H^{r}_{\#}}^2
+32 C^2_{\mathbb A}\|\mathbf f\|_{L_2(0,T_*;\mathbf{H}_\#^{r-1})}^2.
\end{multline}
Inequalities \eqref{E5.69C} and \eqref{E5.70C} mean that the sequences
\begin{align}\label{E5.70aC}
\{\|\mathbf u_m\|_{L_\infty(0,T_*;\dot{\mathbf H}_{\#\sigma}^{r})}\}_{m=1}^\infty
\text { and } \{\|\mathbf u_m\|_{L_2(0,T_*;\dot{\mathbf H}_{\#\sigma}^{r+1})}\}_{m=1}^\infty
\text{ are bounded for } n/2-1<r<n/2.
\end{align}

\paragraph{Step (b).} 
Let now $r=n/2$. 
Then by the multiplication Theorem \ref{RS-T1-S4.6.1}(a)  and
relation  \eqref{eq:mik14},
\begin{align}\label{E4.91-0T0dC}
\left\|({\mathbf u}_m\cdot \nabla ){\mathbf u}_m\right\|^2_{\mathbf{H}_\#^{n/2-1}} 
=\left\|\nabla\cdot({\mathbf u}_m \otimes{\mathbf u}_m)\right\|_{\mathbf{H}_\#^{n/2-1}}^2
&\le \left\|{\mathbf u}_m \otimes{\mathbf u}_m)\right\|_{({H}_\#^{n/2})^{n\times n}}^2
\nonumber\\
&\le C^2_{*rn} \|\mathbf{u}_m\|_{{\mathbf H}_\#^{n/2}}^2\|\mathbf{u}_m\|_{{\mathbf H}_\#^{n/2+1/2}}^2,
\end{align}
where 
$C_{*rn}=C_*(n/2,n/2+1/2, n)$.

Then by \eqref{E4.91-0T0dC} we obtain from \eqref{E4.82T0C}, 
\begin{align}\label{E4.93BTdC}
\frac{d}{d t}\left\|\mathbf u_m\right\|_{\mathbf H^{n/2}_{\#}}^2
+\frac14 C_{\mathbb A}^{-1}\|\mathbf u_m\|_{\mathbf{H}_\#^{n/2+1}}^2
\leq 8 C_{\mathbb A} C^2_{*rn} \|\mathbf{u}_m\|^2_{{\mathbf H}_\#^{n/2+1/2}}\|\mathbf u_m\|^2_{\dot{\mathbf H}_{\#}^{n/2}}
+8 C_{\mathbb A}\|\mathbf f\|_{\mathbf{H}_\#^{n/2-1}}^2,
\end{align}
implying
\begin{align}\label{E4.93BTd1C}
\frac{d}{d t}\left\|\mathbf u_m\right\|_{\mathbf H^{n/2}_{\#}}^2
\leq 8 C_{\mathbb A} C^2_{*rn} \|\mathbf{u}_m\|^2_{{\mathbf H}_\#^{n/2+1/2}}\|\mathbf u_m\|^2_{\dot{\mathbf H}_{\#}^{n/2}}
+8 C_{\mathbb A}\|\mathbf f\|_{\mathbf{H}_\#^{n/2-1}}^2\, .
\end{align}
By Gronwall's inequality \eqref{E17}, we obtain from \eqref{E4.93BTd1C} that
\begin{align}\label{E5.74C}
\|\mathbf u_m\|^2_{L_\infty(0,T_*;\dot{\mathbf H}_{\#\sigma}^{n/2})}
\le \exp\left(8 C_{\mathbb A}C^2_{*rn} \|\mathbf{u}_m\|^2_{L_2(0,T_*;\mathbf H_\#^{n/2+1/2})}\right)
\left[\|\mathbf u_m(\cdot,0)\|_{\mathbf H^{n/2}_{\#}}^2
+8 C_{\mathbb A}\|\mathbf f\|_{L_2(0,T_*;\mathbf{H}_\#^{n/2-1})}^2\right].
\end{align}
We have $\|\mathbf u_m(\cdot,0)\|_{\mathbf H^{n/2}_{\#}}\le \|\mathbf u^0\|_{\mathbf H^{n/2}_{\#}}$ and
by \eqref{E5.70aC}, the sequence $\|\mathbf{u}_m\|_{L_2(0,T_*;\mathbf H_\#^{n/2+1/2})}$ is bounded as well. 
Then \eqref{E5.74C} implies that the sequence $\|\mathbf u_m\|^2_{L_\infty(0,T_*;\dot{\mathbf H}_{\#\sigma}^{n/2})}$ is also bounded.
Integrating \eqref{E4.93BTdC}, we conclude that 
\begin{multline}\label{E5.75C}
\|\mathbf u_m\|^2_{L_2(0,T_*;\dot{\mathbf H}_{\#\sigma}^{n/2+1})}
\le 32 C^2_{\mathbb A} C^2_{*rn} \|\mathbf{u}_m\|^2_{L_2(0,T_*;\mathbf H_\#^{n/2+1/2})}
\|\mathbf u_m\|^2_{L_\infty(0,T_*;\dot{\mathbf H}_{\#\sigma}^{n/2})}
\\
+4C_{\mathbb A}\|\mathbf u_m(\cdot,0)\|_{\mathbf H^{n/2}_{\#}}^2
+32 C^2_{\mathbb A}\|\mathbf f\|_{L_2(0,T_*;\mathbf{H}_\#^{n/2-1})}^2
\le C<\infty.
\end{multline}
Inequalities \eqref{E5.74C} and \eqref{E5.75C} mean that the sequences
\begin{align}\label{E5.75aC}
\{\|\mathbf u_m\|_{L_\infty(0,T_*;\dot{\mathbf H}_{\#\sigma}^{r})}\}_{m=1}^\infty
\text { and } \{\|\mathbf u_m\|_{L_2(0,T_*;\dot{\mathbf H}_{\#\sigma}^{r+1})}\}_{m=1}^\infty
\text{ are bounded for } r=n/2.
\end{align}

\paragraph{Step(c).}
Let now $kn/2< r\le(k+1)n/2$, $k=1,2,3,\ldots$
By the multiplication Theorem \ref{RS-T1-S4.6.1}(a) and
relation  \eqref{eq:mik14}, 
\begin{align}\label{r>n/2E4.910C}
\|({\mathbf u}_m\cdot& \nabla ){\mathbf u}_m\|_{\mathbf{H}_\#^{r-1}}^2 
\le C_{*rn}^2\|\mathbf{u}_m\|_{{\mathbf H}_\#^r}^2 
\|\nabla\mathbf{u}_m\|_{({H}_\#^{r-1})^{n\times n}}^2
\le C_{*rn}^2\|\mathbf{u}_m\|^2_{{\mathbf H}_\#^r}\|\mathbf{u}_m\|^2_{{\mathbf H}_\#^r},
\end{align}
where 
$C_{*rn}=C_*(r-1,r, n)$.

Then by \eqref{r>n/2E4.910C} we obtain from \eqref{E4.82T0C}, 
\begin{align}\label{E4.93BTeC}
\frac{d}{d t}\left\|\mathbf u_m\right\|_{\mathbf H^{r}_{\#}}^2
+\frac14 C_{\mathbb A}^{-1}\|\mathbf u_m\|_{\mathbf{H}_\#^{r+1}}^2
\leq 8 C_{\mathbb A} C^2_{*rn} \|\mathbf{u}_m\|^2_{{\mathbf H}_\#^{r}} \|\mathbf u_m\|^2_{\dot{\mathbf H}_{\#}^{r}}
+ 8C_{\mathbb A}\|\mathbf f\|_{\mathbf{H}_\#^{r-1}}^2
\end{align}
implying
\begin{align}\label{E4.93BTe1C}
\frac{d}{d t}\left\|\mathbf u_m\right\|_{\mathbf H^{r}_{\#}}^2
\leq 8 C_{\mathbb A} C^2_{*rn} \|\mathbf{u}_m\|^2_{{\mathbf H}_\#^{r}} \|\mathbf u_m\|^2_{\dot{\mathbf H}_{\#}^{r}}
+ 8C_{\mathbb A}\|\mathbf f\|_{\mathbf{H}_\#^{r-1}}^2.
\end{align}
By Gronwall's inequality \eqref{E17}, we obtain from \eqref{E4.93BTe1C} that
\begin{align}\label{E5.79C}
\|\mathbf u_m\|^2_{L_\infty(0,T_*;\dot{\mathbf H}_{\#\sigma}^{r})}
\le \exp\left( 8C_{\mathbb A} C^2_{*rn} \|\mathbf{u}_m\|^2_{L_2(0,T_*;\mathbf H_\#^{r})}\right)
 \left[\|\mathbf u_m(\cdot,0)\|_{\mathbf H^{r}_{\#}}^2
+8 C_{\mathbb A}\|\mathbf f\|_{L_2(0,T_*;\mathbf{H}_\#^{r-1})}^2\right],
\end{align}
where $\|\mathbf u_m(\cdot,0)\|_{\mathbf H^{r}_{\#}}\le \|\mathbf u^0\|_{\mathbf H^{r}_{\#}}$.

If $k=1$, then
 the sequence $\|\mathbf{u}_m\|_{L_2(0,T_*;\mathbf H_\#^{r})}$ in \eqref{E5.79C} is bounded due to \eqref{E5.70aC} and \eqref{E5.75aC}. 
Then \eqref{E5.79C} implies that the sequence $\|\mathbf u_m\|^2_{L_\infty(0,T_*;\dot{\mathbf H}_{\#\sigma}^{r})}$ is bounded as well.
Integrating \eqref{E4.93BTeC}, we also conclude that for $k=1$,
\begin{multline}\label{E5.80C}
\|\mathbf u_m\|^2_{L_2(0,T_*;\dot{\mathbf H}_{\#\sigma}^{r+1})}
\le 32 C^2_{\mathbb A} C^2_{*rn} \|\mathbf{u}_m\|^2_{L_2(0,T_*;\mathbf H_\#^{r})}
\|\mathbf u_m\|^2_{L_\infty(0,T_*;\dot{\mathbf H}_{\#\sigma}^{r})}
\\
+4C_{\mathbb A}\|\mathbf u_m(\cdot,0)\|_{\mathbf H^{r}_{\#}}^2
+32 C^2_{\mathbb A}\|\mathbf f\|_{L_2(0,T_*;\mathbf{H}_\#^{r-1})}^2,
\quad kn/2<r\le (k+1)n/2.
\end{multline}
Inequalities \eqref{E5.79C} and \eqref{E5.80C} mean that for $k=1$ the sequences
\begin{align}\label{E5.80aC}
\{\|\mathbf u_m\|_{L_\infty(0,T_*;\dot{\mathbf H}_{\#\sigma}^{r})}\}_{m=1}^\infty
\text { and } \{\|\mathbf u_m\|_{L_2(0,T_*;\dot{\mathbf H}_{\#\sigma}^{r+1})}\}_{m=1}^\infty
\text{ are bounded for } kn/2<r\le (k+1)n/2.
\end{align}

If we assume that properties \eqref{E5.80aC} hold for some integer $k\ge 1$, then by the similar argument,  properties \eqref{E5.80aC} hold with $k$ replaced by $k+1$ and thus, by induction, they hold for any integer $k$.
Hence, collecting properties  \eqref{E5.70aC}, \eqref{E5.75aC}, and \eqref{E5.80aC} we conclude that 
the sequences
\begin{align}\label{E5.80bC}
\{\|\mathbf u_m\|_{L_\infty(0,T_*;\dot{\mathbf H}_{\#\sigma}^{r})}\}_{m=1}^\infty
\text { and } \{\|\mathbf u_m\|_{L_2(0,T_*;\dot{\mathbf H}_{\#\sigma}^{r+1})}\}_{m=1}^\infty
\text{ are bounded for } n/2-1<r.
\end{align}

Properties \eqref{E5.80bC} imply that  there exists a subsequence of $\{\mathbf u_m\}$ converging weakly  in $L_2(0, T_* ; \dot{\mathbf H}_{\#\sigma}^{r+1})$ and weakly-star in $L_\infty(0, T_* ; \dot{\mathbf H}_{\#\sigma}^r)$ to a function 
${\mathbf u}^\dag\in L_2(0, T_* ; \dot{\mathbf H}_{\#\sigma}^{r+1})\cup L_\infty(0, T_* ; \dot{\mathbf H}_{\#\sigma}^r)$.
Then the subsequence converges to ${\mathbf u}^\dag$ also weakly in $L_2(0,T_*;\dot{\mathbf H}_{\#\sigma}^{1})$ and weakly-star in $L_\infty(0,T_*;\dot{\mathbf H}_{\#\sigma}^{0})$.
Since $\{\mathbf u_m\}$ is the subsequence of the sequence that converges weakly in $L_2(0,T;\dot{\mathbf H}_{\#\sigma}^{1})$ and weakly-star in $L_\infty(0,T;\dot{\mathbf H}_{\#\sigma}^{0})$ to the weak solution, $\mathbf u$, of problem \eqref{NS-problem-div0}--\eqref{NS-problem-div0-IC} on $\left[0, T_*\right]$, we conclude that 
$\mathbf u={\mathbf u}^\dag\in L_2(0, T_* ; \dot{\mathbf H}_{\#\sigma}^{r+1})\cup L_\infty(0, T_* ; \dot{\mathbf H}_{\#\sigma}^r)$, for any $r>n/2-1$ and we thus finished proving that
\begin{align}\label{E5.28C}
\mathbf u\in L_{\infty}(0,T_*;\dot{\mathbf H}_{\#\sigma}^{r})\cap L_2(0,T_*;\dot{\mathbf H}_{\#\sigma}^{r+1}).
\end{align}

\paragraph{Step(d).} Repeating for $\mathbf u$ the reasoning related to inequalities \eqref{E4.91-0T0cC}, \eqref{E4.91-0T0dC}, \eqref{r>n/2E4.910C}, corresponding to the considered $r$, we obtain
\begin{align*}
&\left\|({\mathbf u}\cdot \nabla ){\mathbf u}\right\|_{\mathbf{H}_\#^{r-1}}^2
\le  C^2_{*rn}\|{\mathbf u}\|^2_{{\mathbf H}_\#^r}\|{\mathbf u}\|^2_{{\mathbf H}_\#^{r+1}}.
\end{align*}
Hence 
\begin{align}\label{E4.91-0T0EE1r-uC}
\left\|({\mathbf u}\cdot \nabla ){\mathbf u}\right\|_{L_2(0,T;\mathbf{H}_\#^{r-1})}
\le C_{*rn}\|\mathbf{u}\|_{L_\infty(0,T_*;{\mathbf H}_\#^r)} \|\mathbf{u}\|_{L_2(0,T_*;{\mathbf H}_\#^{r+1})}
\end{align}
Due to \eqref{E5.28C} then $({\mathbf u}\cdot \nabla ){\mathbf u}\in{L_2(0,T_*;\mathbf{H}_\#^{r-1})}$.
By \eqref{L-oper} and \eqref{TensNorm} we have
\begin{align*}
\left\|\bs{\mathfrak L}\mathbf u\right\|^2_{{\mathbf H}_{\#}^{r-1}}
\le\|a_{ij}^{\alpha \beta }E_{i\alpha }({\mathbf u})\|_{({H}_{\#}^{r})^{n\times n}}
\le \|\mathbb A\|^2_{H_\#^{\tilde\sigma+1},F}\|\mathbf u\|^2_{{\mathbf H}_{\#}^{r+1}}\quad
\text{for a.e. } t\in(0,T),
\end{align*}
and thus
\begin{align*}
\|\bs{\mathfrak L}\mathbf u\|^2_{L_2(0,T_*;\dot{\mathbf H}_{\#}^{r-1})}
&\le \|\mathbb A\|^2_{L_\infty(0,T_*;H_\#^{\tilde\sigma+1}),F}\|\mathbf u\|^2_{L_2(0,T_*;{\mathbf H}_{\#}^{r+1})},
\end{align*}
i.e., $\bs{\mathfrak L}\mathbf u\in{L_2(0,T_*;\dot{\mathbf H}_{\#}^{r-1})}$.
We also have ${\mathbf f}\in L_2(0,T;\dot{\mathbf H}_\#^{r-1})$.
Thus $\mathbf F$ defined by \eqref{Eq-F} belongs to $L_2(0,T;\dot{\mathbf H}_{\#}^{r-1})$.
Then \eqref{NS-problem-just} implies that 
${\mathbf u}'\in{L_2(0,T_*;\mathbf{H}_{\#\sigma}^{r-1})}$ and since ${\mathbf u}\in{L_2(0,T_*;\mathbf{H}_{\#\sigma}^{r+1})}$, we obtain by Theorem \ref{LM-T3.1} that 
$
\mathbf u \in\mathcal C^0([0,T_*];\dot{\mathbf H}_{\#\sigma}^{r})
$,
which also means 
that $
\| \mathbf u(\cdot,t)-{\mathbf u}^0\| _{\dot{\mathbf H}_{\#\sigma}^{r}}\to 0
$
as ${t\to 0}$.

To prove the theorem claim about the associated pressure $p$, we remark that $p$ satisfies \eqref{Eq-p}.
By Lemma \ref{div-grad-is} for gradient, with $s=r$, equation \eqref{Eq-p} has a unique solution $p$ in 
$L_2(0,T_*;\dot{H}_{\#}^{r})$.
\end{proof}

As in Corollaries \ref{NS-problemCor-sigma-2C} and \ref{NS-problemCor-sigma-1C},  condition  \eqref{E4.143T*varC} in Theorem \ref{NS-problemTh-sigma-Lv-erC} can be replaced by simpler conditions for particular cases, which leads to the following two assertions.
\begin{corollary}[Serrin-type solution  for arbitrarily large data but small time or vice versa.]
\label{NS-problemCor-sigma-2>C}
Let $n\ge 2$ and $T>0$. 
Let the coefficients $a_{ij}^{\alpha\beta}$ be constant and the relaxed ellipticity condition hold.
Let ${\mathbf f}\in L_2(0,T;\dot{\mathbf H}_\#^{r-1})\cap L_\infty(0,T;\dot{\mathbf H}_\#^{n/2-2})$  and 
$\mathbf u^0\in \dot{\mathbf H}_{\#\sigma}^{r}\cap \dot{\mathbf H}_{\#\sigma}^{n/2}$, $r>{n}/{2}-1$.
Let $T_*\in (0,T)$ satisfies inequality \eqref{E4.143T*C} in Corollary \ref{NS-problemCor-sigma-2C}.

Then the Serrin-type solution $\mathbf u$ of the anisotropic Navier-Stokes 
initial value problem  \eqref{NS-problem-div0}--\eqref{NS-problem-div0-IC} 
belongs to
$L_{\infty}(0,T_*;\dot{\mathbf H}_{\#\sigma}^{r})\cap L_2(0,T_*;\dot{\mathbf H}_{\#\sigma}^{r+1})$.
In addition, 
$
\mathbf u'\in L_2(0,T_*;\dot{\mathbf H}_{\#\sigma}^{r-1}),
$
$
\mathbf u \in\mathcal C^0([0,T_*];\dot{\mathbf H}_{\#\sigma}^{r})
$, 
$
\lim_{t\to 0}\| \mathbf u(\cdot,t)-{\mathbf u}^0\| _{\dot{\mathbf H}_{\#\sigma}^{r}}= 0
$
and $p\in L_2(0,T_*;\dot{H}_{\#}^{r})$.
\end{corollary}

\begin{corollary}[Serrin-type solution for arbitrary time but small data]
\label{NS-problemCor-sigma-1>C}
Let $n\ge 2$,   $r\ge{n}/{2}-1$, and $T>0$. 
Let the coefficients $a_{ij}^{\alpha\beta}$ be constant and the relaxed ellipticity condition hold.
Let the data ${\mathbf f}\in L_2(0,T;\dot{\mathbf H}_\#^{r-1})$  and $\mathbf u^0\in \dot{\mathbf H}_{\#\sigma}^{r}$ satisfy inequality \eqref{E4.143TC} in Corollary \ref{NS-problemCor-sigma-1C}.

Then the Serrin-type solution $\mathbf u$ of the anisotropic Navier-Stokes 
initial value problem  \eqref{NS-problem-div0}--\eqref{NS-problem-div0-IC} 
belongs to
$L_{\infty}(0,T;\dot{\mathbf H}_{\#\sigma}^{r})\cap L_2(0,T;\dot{\mathbf H}_{\#\sigma}^{r+1})$.
In addition, 
$
\mathbf u'\in L_2(0,T;\dot{\mathbf H}_{\#\sigma}^{r-1}),
$
$
\mathbf u \in\mathcal C^0([0,T];\dot{\mathbf H}_{\#\sigma}^{r})
$,
$
\lim_{t\to 0}\| \mathbf u(\cdot,t)-{\mathbf u}^0\| _{\dot{\mathbf H}_{\#\sigma}^{r}}= 0
$
and $p\in L_2(0,T;\dot{H}_{\#}^{r})$.
\end{corollary}

Theorem \ref{NS-problemTh-sigma-Lv-erC} leads also to the following infinite regularity assertion.
\begin{corollary}
\label{Cor5.11spC}
Let $T>0$ and $n\ge 2$. 
Let the coefficients $a_{ij}^{\alpha\beta}$ be constant and the relaxed ellipticity condition \eqref{mu} hold.
Let ${\mathbf f}\in L_2(0,T;\dot{\mathbf C}_\#^{\infty})$
and $\mathbf u^0\in \dot{\mathbf C}_{\#\sigma}^{\infty}$, while ${\mathbf f}$, $\mathbf u^0$ and $T_*\in(0,T]$ satisfy  inequality \eqref{E4.143T*varC} from Theorem \ref{NS-problemTh-sigma-Lv-crC}.

Then the Serrin-type solution $\mathbf u$ of the anisotropic Navier-Stokes 
initial value problem  \eqref{NS-problem-div0}--\eqref{NS-problem-div0-IC} 
is such that
${\mathbf u}\in \mathcal C^0([0,T_*];\dot{\mathbf C}_{\#\sigma}^{\infty})$, 
$\mathbf u'\in L_2(0,T_*;\dot{\mathbf C}_{\#\sigma}^{\infty})$ and 
$p\in L_2(0,T_*;\dot{\mathcal C}_{\#}^{\infty})$.
\end{corollary}
\begin{proof}
Taking into account that $\dot{\mathbf C}_\#^{\infty}=\bigcap_{r\in\R} \dot{\mathbf{H}}_{\#\sigma}^r$, Theorem \ref{NS-problemTh-sigma-Lv-erC} implies that
${\mathbf u}\in{L_\infty(0,T_*;\dot{\mathbf{H}}_{\#\sigma}^{r})}\cap L_2(0,T_*;\dot{\mathbf H}_\#^{r+1})$, 
$\mathbf u'\in L_2(0,T_*;\dot{\mathbf H}_{\#\sigma}^{r-1}),$
$p\in L_2(0,T;\dot{H}_{\#}^{r})$,
 $\forall\, r\in\R$.
Hence ${\mathbf u}\in \mathcal C^0([0,T_*];\dot{\mathbf C}_{\#\sigma}^{\infty})$, 
$\mathbf u'\in L_2(0,T_*;\dot{\mathbf C}_{\#\sigma}^{\infty})$
and $p\in L_2(0,T_*;\dot{\mathcal C}_{\#}^{\infty})$.
\end{proof}

\subsection{Spatial-temporal regularity of Serrin-type solutions  for constant anisotropic viscosity coefficients}\label{S5C}

\begin{theorem}
\label{NS-problemTh-sigma-Lv-er-t1C}
Let $T>0$ and $n\ge 2$. Let $r\ge {n}/{2}-1$ if $n\ge 3$, while $r> {n}/{2}-1$  if $n=2$. 
Let the coefficients $a_{ij}^{\alpha\beta}$ be constant and the relaxed ellipticity condition \eqref{mu} hold.
Let ${\mathbf f}\in L_\infty(0,T;\dot{\mathbf H}_\#^{r-2})\cap L_2(0,T;\dot{\mathbf H}_\#^{r-1})$
and $\mathbf u^0\in \dot{\mathbf H}_{\#\sigma}^{r}$, while ${\mathbf f}$, $\mathbf u^0$ and $T_*\in(0,T]$ satisfy  inequality \eqref{E4.143T*varC} from Theorem \ref{NS-problemTh-sigma-Lv-crC}.

Then the Serrin-type solution $\mathbf u$ of the anisotropic Navier-Stokes 
initial value problem  \eqref{NS-problem-div0}--\eqref{NS-problem-div0-IC} 
is such that
$\mathbf u'\in L_{\infty}(0,T_*;\dot{\mathbf H}_{\#\sigma}^{r-2})\cup L_2(0,T_*;\dot{\mathbf H}_{\#\sigma}^{r-1})$, while
$p\in L_\infty(0,T_*;\dot{H}_{\#}^{r-1})\cap L_2(0,T_*;\dot{H}_{\#}^{r})$.
\end{theorem}
\begin{proof}
By Theorems \ref{NS-problemTh-sigma-Lv-crC} and \ref{NS-problemTh-sigma-Lv-erC},  we have the inclusions 
$\mathbf u\in L_{\infty}(0,T_*;\dot{\mathbf H}_{\#\sigma}^{r})$,
$\mathbf u'\in L_2(0,T_*;\dot{\mathbf H}_{\#\sigma}^{r-1})$ and $p\in L_2(0,T_*;\dot{H}_{\#}^{r})$. 
Then we only need to prove the inclusions $\mathbf u'\in L_{\infty}(0,T_*;\dot{\mathbf H}_{\#\sigma}^{r-2})$ and $p\in L_\infty(0,T_*;\dot{H}_{\#}^{r-1})$.

Let, first, $n/2-1\le r<n/2+1$ (and also $0<r$ if $n=2$). 
By relation  \eqref{eq:mik14} and  the multiplication Theorem \ref{RS-T1-S4.6.1}(b), 
we have
\begin{align*}
\left\|({\mathbf u}\cdot \nabla ){\mathbf u}\right\|_{\mathbf{H}_\#^{r-2}}
=\left\|\nabla\cdot({\mathbf u} \otimes{\mathbf u})\right\|_{\mathbf{H}_\#^{r-2}}
&\le \left\|{\mathbf u} \otimes{\mathbf u})\right\|_{({H}_\#^{r-1})^{n\times n}}
\le C'_{*rn} \|\mathbf{u}\|_{{\mathbf H}_\#^{r/2+n/4-1/2}}^2
\le C'_{*rn}\|\mathbf{u}\|^2_{{\mathbf H}_\#^r},
\end{align*}
where $C'_{*rn}=C_*(r/2+n/4,r/2+n/4, n)$.

Let now $r\ge n/2+1$. Again, by relation  \eqref{eq:mik14} and  by the multiplication Theorem \ref{RS-T1-S4.6.1}(a), 
we have
\begin{align*}
\left\|({\mathbf u}\cdot \nabla ){\mathbf u}\right\|_{\mathbf{H}_\#^{r-2}}
=\left\|\nabla\cdot({\mathbf u} \otimes{\mathbf u})\right\|_{\mathbf{H}_\#^{r-2}}
&\le \left\|{\mathbf u} \otimes{\mathbf u})\right\|_{({H}_\#^{r-1})^{n\times n}}
\le C''_{*rn} \|\mathbf{u}\|_{{\mathbf H}_\#^{r-1}}\|\mathbf{u}\|_{{\mathbf H}_\#^{r}}
\le C''_{*rn} \|\mathbf{u}\|_{{\mathbf H}_\#^{r}}^2,
\end{align*}
where $C''_{*rn}=C_*(r-1,r, n)$.
Hence in both cases 
\begin{align}\label{E4.91-0T0c*C}
\left\|({\mathbf u}\cdot \nabla ){\mathbf u}\right\|_{L_{\infty}(0,T_*;\mathbf{H}_\#^{r-2})}
\le C_{*rn} \|\mathbf{u}\|_{L_{\infty}(0,T_*;\mathbf{H}_\#^{r})}^2,
\end{align}
where $C_{*rn}$ is $C'_{*rn}$ or $C''_{*rn}$, respectively.

By \eqref{L-oper} and \eqref{TensNorm}, we have
\begin{align*}
\left\|\bs{\mathfrak L}\mathbf u\right\|_{{\mathbf H}_{\#}^{r-2}}
\le\|a_{ij}^{\alpha \beta }E_{i\alpha }({\mathbf u})\|_{({H}_{\#}^{r-1})^{n\times n}}
\le \|\mathbb A\|\|\mathbf u\|_{{\mathbf H}_{\#}^{r}}.
\end{align*}
Thus
\begin{align*}
\|\bs{\mathfrak L}\mathbf u\|_{L_\infty(0,T_*;\dot{\mathbf H}_{\#}^{r-2})}
&\le \|\mathbb A\|\|\mathbf u\|_{L_\infty(0,T_*;{\mathbf H}_{\#}^{r})},
\end{align*}
i.e., $\bs{\mathfrak L}\mathbf u\in{L_\infty(0,T_*;\dot{\mathbf H}_{\#}^{r-2})}$.
We also have ${\mathbf f}\in L_\infty(0,T;\dot{\mathbf H}_\#^{r-2})$.

Then \eqref{NS-problem-just} implies that 
${\mathbf u}'\in{L_\infty(0,T_*;\mathbf{H}_{\#\sigma}^{r-2})}$, while \eqref{Eq-p} and Lemma \ref{div-grad-is} for gradient, with $s=r-1$, imply that  $p\in L_\infty(0,T_*;\dot{H}_{\#}^{r-1})$. 
\end{proof}

\begin{theorem}
\label{NS-problemTh-sigma-Lv-er-tC}
Let $T>0$ and $n\ge 2$. 
Let $r> {n}/{2}-1$. 
Let the coefficients $a_{ij}^{\alpha\beta}$ be constant and the relaxed ellipticity condition \eqref{mu} hold.
Let $k\in [1, r+1)$ be an integer. 
Let 
${\mathbf f}^{(l)}\in L_\infty(0,T;\dot{\mathbf H}_\#^{r-2-2l})\cap L_2(0,T;\dot{\mathbf H}_\#^{r-1-2l})$, $l=0,1,\ldots,k-1$,
and $\mathbf u^0\in \dot{\mathbf H}_{\#\sigma}^{r}$, while ${\mathbf f}$, $\mathbf u^0$ and $T_*\in(0,T]$ satisfy  inequality \eqref{E4.143T*varC} from Theorem \ref{NS-problemTh-sigma-Lv-crC}.

Then the Serrin-type solution $\mathbf u$ of the anisotropic Navier-Stokes 
initial value problem  \eqref{NS-problem-div0}--\eqref{NS-problem-div0-IC} 
is such that
$\mathbf u^{(l)}\in L_{\infty}(0,T_*;\dot{\mathbf H}_{\#\sigma}^{r-2l})\cap L_2(0,T_*;\dot{\mathbf H}_{\#\sigma}^{r+1-2l})$, $l=0,\ldots,k$, while
$p^{(l)}\in L_\infty(0,T_*;\dot{H}_{\#}^{r-1-2l})\cap L_2(0,T_*;\dot{H}_{\#\sigma}^{r-2l})$, $l=0,\ldots,k-1$.
\end{theorem}
\begin{proof}
Some parts of the following proof are inspired by \cite[Theorem 3.1]{Temam1982} and \cite[Chapter 3, Section 3.6]{Temam2001}, see also \cite[Section 7.2]{RRS2016}. 

We will employ the mathematical induction argument in the proof. 
We first remark that by Theorems \ref{NS-problemTh-sigma-Lv-crC}, \ref{NS-problemTh-sigma-Lv-erC} and \ref{NS-problemTh-sigma-Lv-er-t1C}, if 
${\mathbf f}\in L_\infty(0,T;\dot{\mathbf H}_\#^{r-2})\cap L_2(0,T;\dot{\mathbf H}_\#^{r-1})$, then
$\mathbf u\in L_{\infty}(0,T_*;\dot{\mathbf H}_{\#\sigma}^{r})\cap L_2(0,T_*;\dot{\mathbf H}_{\#\sigma}^{r+1})$,
$\mathbf u'\in L_{\infty}(0,T_*;\dot{\mathbf H}_{\#\sigma}^{r-2})\cap L_2(0,T_*;\dot{\mathbf H}_{\#\sigma}^{r-1})$,
while $p\in L_\infty(0,T_*;\dot{H}_{\#}^{r-1})\cap L_2(0,T_*;\dot{H}_{\#}^{r})$.
This means that the theorem holds true for $k=1$.

Let us assume that the theorem holds true for some $k'=k-1\in[1, r)$, i.e., 
\begin{align}\label{E5.75tC}
\mathbf u^{(l)}\in L_{\infty}(0,T_*;\dot{\mathbf H}_{\#\sigma}^{r-2l})\cap L_2(0,T_*;\dot{\mathbf H}_{\#\sigma}^{r+1-2l}),\quad 
l=0,\ldots,k-1,
\end{align}
and prove that it hold also for $l=k$.
To this end, let us differentiate  equation  \eqref{NS-problem-just} $k-1$ times in $t$ (in the distribution sense) to obtain
\begin{align}
\label{NS-problem-just-kC}
\mathbf u^{(k)} 
=\mathbb P_\sigma\mathbf F^{(k-1)}
\quad \mbox{in } \T\times(0,T),
\end{align}
where 
\begin{align}\label{E5.59C}
\mathbf F^{(l)}:=\mathbf{f}^{(l)}
+\partial_t^{l}\bs{\mathfrak L}{\mathbf u}
-\partial_t^{l}[({\mathbf u}\cdot \nabla ){\mathbf u}] \quad \forall\, l\in\N.
\end{align}

 Let us denote
$
s_{l1}:=r-2\max\{(k-1-l),l\}, \ 
s_{l2}:=r-2\min\{(k-1-l),l\}.
$
Then
\begin{align*}
&s_{l1}\le r-k+1\le s_{l2},\\
&s_{l1}+s_{l2}=r-2(k-1-l)+r-2l=2(r-k+1),\quad \forall\ l=0,\ldots,k-1.
\end{align*} 
The theorem condition  $r-k+1>0$ implies $s_{l1}+s_{l2}>0$.

\paragraph{Step 1: Convection term} 
\begin{align}\label{E5.60C}
\partial_t^{k-1}[({\mathbf u}\cdot \nabla ){\mathbf u}]=\sum_{l=0}^{k-1}C_{k-1}^l({\mathbf u}^{(k-1-l)}\cdot \nabla ){\mathbf u}^{(l)}
\end{align}
where $C_{k-1}^l$ are the binomial coefficients.
\\[1ex]
{\it Case (A).} Let $0\le l\le (k-1)/2$. Then $s_{l1}=r-2(k-1-l)$, $s_{l2}=r-2l$. 

{\it Subcase (A1).}  Let ${n}/{2}-1<r\le n/2+2l$.  Then $s_{l2}=r-2l\le n/2$.
By the theorem conditions, $r+1-n/2>0$ and $r-k+1>0$, and hence there exists $\epsilon\in(0,\min\{r+1-n/2, 2(r-k+1)\})$ and thus
$$
(r+1-n/2)-\epsilon>0,\quad s_{l1}-\epsilon/2\le s_{l2}-\epsilon/2<n/2, \quad s_{l1}+s_{l2}-\epsilon=2(r-k+1)-\epsilon>0.
$$
By relation  \eqref{eq:mik14} and  the multiplication Theorem \ref{RS-T1-S4.6.1}(b), 
we have
\begin{multline*}
\left\|({\mathbf u}^{(k-1-l)}\cdot \nabla ){\mathbf u}^{(l)}\right\|_{\mathbf{H}_\#^{r-2k}}
=\left\|\nabla\cdot({\mathbf u}^{(k-1-l)} \otimes{\mathbf u}^{(l)})\right\|_{\mathbf{H}_\#^{r-2k}}
\le \left\|{\mathbf u}^{(k-1-l)} \otimes{\mathbf u}^{(l)})\right\|_{({H}_\#^{r+1-2k})^{n\times n}}
\\
\le  \left\|{\mathbf u}^{(k-1-l)} \otimes{\mathbf u}^{(l)})\right\|_{({H}_\#^{r+1-2k+(r+1-n/2)-\epsilon})^{n\times n}}
= \left\|{\mathbf u}^{(k-1-l)} \otimes{\mathbf u}^{(l)})\right\|_{({H}_\#^{s_{l1}+s_{l2}-\epsilon-n/2})^{n\times n}}
\\
\le  C'_{*rn} \left\|{\mathbf u}^{(k-1-l)}\right\|_{\mathbf H_\#^{s_{l1}-\epsilon/2}}
\left\|{\mathbf u}^{(l)})\right\|_{\mathbf H_\#^{s_{l2}-\epsilon/2}}  
\le C'_{*rn} \left\|{\mathbf u}^{(k-1-l)}\right\|_{\mathbf H_\#^{r-2(k-1-l)}} 
\left\|{\mathbf u}^{(l)}\right\|_{\mathbf H_\#^{r-2l}},
\end{multline*}
\begin{multline*}
\left\|({\mathbf u}^{(k-1-l)}\cdot \nabla ){\mathbf u}^{(l)}\right\|_{\mathbf{H}_\#^{r+1-2k}}
\le  \left\|({\mathbf u}^{(k-1-l)}\cdot \nabla ){\mathbf u}^{(l)}\right\|_{\mathbf{H}_\#^{r+1-2k+(r+1-n/2)-\epsilon}}
= \left\|({\mathbf u}^{(k-1-l)}\cdot \nabla ){\mathbf u}^{(l)}\right\|_{\mathbf{H}_\#^{s_{l1}+s_{l2}-\epsilon-n/2}}
\\
\le  C'_{*rn} \left\|{\mathbf u}^{(k-1-l)}\right\|_{\mathbf H_\#^{s_{l1}-\epsilon/2}}
\left\|\nabla{\mathbf u}^{(l)})\right\|_{(H_\#^{s_{l2}-\epsilon/2})^{n\times n}}  
\le  C'_{*rn} \left\|{\mathbf u}^{(k-1-l)}\right\|_{\mathbf H_\#^{r-2(k-1-l)}} 
\left\|{\mathbf u}^{(l)}\right\|_{\mathbf H_\#^{r+1-2l}},
\end{multline*}
where $C'_{*rn}=C_*(s_{l1}-\epsilon/2,s_{l2}-\epsilon/2, n)$.

{\it Subcase (A2).}  Let $r>n/2+2l$.  Then $s_{l2}=r-2l>n/2$.

Hence by relation  \eqref{eq:mik14} and  the multiplication Theorem \ref{RS-T1-S4.6.1}(a), 
we have
\begin{multline*}
\left\|({\mathbf u}^{(k-1-l)}\cdot \nabla ){\mathbf u}^{(l)}\right\|_{\mathbf{H}_\#^{r-2k}}
=\left\|\nabla\cdot({\mathbf u}^{(k-1-l)} \otimes{\mathbf u}^{(l)})\right\|_{\mathbf{H}_\#^{r-2k}}
\le \left\|{\mathbf u}^{(k-1-l)} \otimes{\mathbf u}^{(l)})\right\|_{({H}_\#^{r+1-2k})^{n\times n}}
\\
\le \left\|{\mathbf u}^{(k-1-l)} \otimes{\mathbf u}^{(l)})\right\|_{({H}_\#^{r-2(k-1-l)})^{n\times n}}
\le  C''_{*rn} \left\|{\mathbf u}^{(k-1-l)}\right\|_{\mathbf H_\#^{r-2(k-1-l)}} \left\|{\mathbf u}^{(l)})\right\|_{\mathbf H_\#^{r-2l}},
\end{multline*}
\begin{multline*}
\left\|({\mathbf u}^{(k-1-l)}\cdot \nabla ){\mathbf u}^{(l)}\right\|_{\mathbf{H}_\#^{r+1-2k}}
\le \left\|({\mathbf u}^{(k-1-l)}\cdot \nabla ){\mathbf u}^{(l)}\right\|_{\mathbf{H}_\#^{r-2(k-1-l)}}
\\
\le C''_{*rn} \left\|{\mathbf u}^{(k-1-l)}\right\|_{\mathbf H_\#^{r-2(k-1-l)}} \left\|\nabla{\mathbf u}^{(l)})\right\|_{\mathbf H_\#^{r-2l}}
\le C''_{*rn} \left\|{\mathbf u}^{(k-1-l)}\right\|_{\mathbf H_\#^{r-2(k-1-l)}} \left\|{\mathbf u}^{(l)})\right\|_{\mathbf H_\#^{r+1-2l}},
\end{multline*}
where $C''_{*rn}=C_*(s_{l1},s_{l2}, n)$.

Thus, combining the cases (A1) and (A2), we obtain that for any $r>n/2-1$ and for $0\le l\le (k-1)/2$,
\begin{align}
\label{E5.77aC}
&\left\|({\mathbf u}^{(k-1-l)}\cdot \nabla ){\mathbf u}^{(l)}\right\|_{\mathbf{H}_\#^{r-2k}}
\le C_{*rn} \left\|{\mathbf u}^{(k-1-l)}\right\|_{\mathbf H_\#^{r-2(k-1-l)}} 
\left\|{\mathbf u}^{(l)}\right\|_{\mathbf H_\#^{r-2l}},
\\
\label{E5.77a+C}
&\left\|({\mathbf u}^{(k-1-l)}\cdot \nabla ){\mathbf u}^{(l)}\right\|_{\mathbf{H}_\#^{r+1-2k}}
\le C_{*rn} \left\|{\mathbf u}^{(k-1-l)}\right\|_{\mathbf H_\#^{r-2(k-1-l)}} \left\|{\mathbf u}^{(l)})\right\|_{\mathbf H_\#^{r+1-2l}},
\end{align}
where $C_{*rn}$ is $C'_{*rn}$ or $C''_{*rn}$, respectively.
\\[1ex]
{\it Case (B).} Let $(k-1)/2\le l\le k-1$. Then taking into account that 
\begin{align*}
\left\|({\mathbf u}^{(k-1-l)}\cdot \nabla ){\mathbf u}^{(l)}\right\|_{\mathbf{H}_\#^{r-2k}}
=\left\|\nabla\cdot({\mathbf u}^{(k-1-l)} \otimes{\mathbf u}^{(l)})\right\|_{\mathbf{H}_\#^{r-2k}}
=\left\|({\mathbf u}^{(l)}\cdot \nabla ){\mathbf u}^{(k-1-l)}\right\|_{\mathbf{H}_\#^{r-2k}},
\end{align*}
we arrive at the case (A) for $l'=k-1-l$ and finally to the same estimates \eqref{E5.77aC}--\eqref{E5.77a+C}.

Thus  for any $r>n/2-1$ and any integer $l\in [0, k-1]$, we obtain  the estimates
\begin{align*}
&\left\|({\mathbf u}^{(k-1-l)}\cdot \nabla ){\mathbf u}^{(l)}\right\|_{L_\infty(0,T_*;\mathbf{H}_\#^{r-2k})}
\le C_{*rn} \left\|{\mathbf u}^{(k-1-l)}\right\|_{L_\infty(0,T_*;\mathbf H_\#^{r-2(k-1-l)})} 
\left\|{\mathbf u}^{(l)}\right\|_{L_\infty(0,T_*;\mathbf H_\#^{r-2l})},
\\
&\left\|({\mathbf u}^{(k-1-l)}\cdot \nabla ){\mathbf u}^{(l)}\right\|_{L_2(0,T_*;\mathbf{H}_\#^{r+1-2k})}
\le C_{*rn} \left\|{\mathbf u}^{(k-1-l)}\right\|_{L_\infty(0,T_*;\mathbf H_\#^{r-2(k-1-l)})} 
\left\|{\mathbf u}^{(l)}\right\|_{L_2(0,T_*;\mathbf H_\#^{r+1-2l})},
\end{align*}
Hence by \eqref{E5.60C} and \eqref{E5.75tC},
\begin{align}\label{E5.63C}
\partial_t^{k-1}[({\mathbf u}\cdot \nabla ){\mathbf u}]\in L_\infty(0,T_*;\mathbf H_\#^{r-2k})\cap L_2(0,T_*;\mathbf H_\#^{r+1-2k}).
\end{align}

\paragraph{Step 2: Linear terms and right-hand side}\ 

Due to \eqref{E5.75tC},
\begin{align}\label{E5.65C}
\partial_t^{l}\bs{\mathfrak L}{\mathbf u}=\bs{\mathfrak L}{\mathbf u}^{(k-1)}\in L_\infty(0,T_*;\mathbf H_\#^{r-2k})\cap L_2(0,T_*;\mathbf H_\#^{r+1-2k}).
\end{align}
We also have ${\mathbf f}^{(k-1)}\in L_\infty(0,T;\dot{\mathbf H}_\#^{r-2k})\cap L_2(0,T;\dot{\mathbf H}_\#^{r+1-2k})$.
Then  \eqref{E5.59C}, \eqref{E5.63C} and \eqref{E5.65C} 
imply that 
\begin{align}\label{E5.66C}
{\mathbf F}^{(k-1)}\in L_\infty(0,T;\dot{\mathbf H}_\#^{r-2k})\cap L_2(0,T;\dot{\mathbf H}_\#^{r+1-2k}).
\end{align}
Thus by \eqref{NS-problem-just-kC},
${\mathbf u}^{(k)}\in{L_\infty(0,T_*;\dot{\mathbf{H}}_{\#\sigma}^{r-2k})}\cap L_2(0,T_*;\dot{\mathbf H}_\#^{r+1-2k})$.

\paragraph{Step 3: Pressure}\ 

The associated pressure $p$ satisfies \eqref{Eq-p}. Differentiating it in time, we obtain
\begin{align}
\label{Eq-p-tC}
\nabla p^{(l)}=\mathbb P_g\mathbf{F}^{(l)}\quad 
\quad \mbox{in } \T\times(0,T), \quad l=0,1,\ldots,k-1.
\end{align}
By the same reasoning as in the proof of \eqref{E5.66C}, the similar inclusions for junior derivatives also hold,
\begin{align*}
{\mathbf F}^{(l)}\in L_\infty(0,T;\dot{\mathbf H}_\#^{r-2-2l})\cap L_2(0,T;\dot{\mathbf H}_\#^{r-1-2l}), \quad l=0,1,\ldots,k-1.
\end{align*}
By Lemma \ref{div-grad-is} for gradient, with $s=r-1-2l$ and $s=r-2l$, respectively, equation \eqref{Eq-p-tC} implies that $p^{(l)}\in L_\infty(0,T_*;\dot{H}_{\#}^{r-1-2l})\cap L_2(0,T_*;\dot{H}_{\#}^{r-2l})$.
\end{proof}

\begin{corollary}
\label{Cor5.11C}
Let $T>0$ and $n\ge 2$. 
Let the coefficients $a_{ij}^{\alpha\beta}$ be constant and the relaxed ellipticity condition \eqref{mu} hold.
Let ${\mathbf f}\in \mathcal C^\infty(0,T;\dot{\mathbf C}_\#^{\infty})$,  
and $\mathbf u^0\in \dot{\mathbf C}_{\#\sigma}^{\infty}$, while ${\mathbf f}$, $\mathbf u^0$ and $T_*\in(0,T]$ satisfy  inequality \eqref{E4.143T*varC} from Theorem \ref{NS-problemTh-sigma-Lv-crC}.

Then the Serrin-type solution $\mathbf u$ of the anisotropic Navier-Stokes 
initial value problem  \eqref{NS-problem-div0}--\eqref{NS-problem-div0-IC} 
is such that
$\mathbf u\in \mathcal C^{\infty}(0,T_*;\dot{\mathbf C}_{\#\sigma}^{\infty})$,
$p\in \mathcal C^\infty(0,T_*;\dot{\mathcal C}_{\#}^\infty)$.
\end{corollary}
\begin{proof}
Taking into account that $\dot{\mathbf C}_\#^{\infty}=\bigcap_{r\in\R} \dot{\mathbf{H}}_{\#\sigma}^r$, Theorem \ref{NS-problemTh-sigma-Lv-er-tC} implies that for any integer $k\ge 0$,
${\mathbf u}^{(k)}\in{L_\infty(0,T_*;\dot{\mathbf{H}}_{\#\sigma}^{r-2k})}\cap L_2(0,T_*;\dot{\mathbf H}_\#^{r+1-2k})$,
$p^{(k)}\in L_\infty(0,T_*;\dot{H}_{\#}^{r-1-2k})\cap L_2(0,T_*;\dot{H}_{\#\sigma}^{r-2k})$, for any $r\in\R$.
Hence $\mathbf u\in \mathcal C^{\infty}(0,T_*;\dot{\mathbf C}_{\#\sigma}^{\infty})$,
$p\in \mathcal C^\infty(0,T_*;\dot{\mathcal C}_{\#}^\infty)$.
\end{proof}

\subsection{Regularity of two-dimensional weak solution for constant viscosity coefficients}\label{S2DC}
The regularity results of Section \ref{S5.4C} and \ref{S5C} hold for $n=2$, but  as for the the isotropic constant-coefficient case (cf. e.g., \cite[Chapter 3, Sections 3.3, 3.5.1]{Temam2001}, \cite[Section 6.5]{RRS2016}) these results can be essentially improved for $n=2$ also in the anisotropic setting with constant coefficients.

Let us give a counterpart of Theorem \ref{NS-problemTh-sigma-Lv-erC} that for $n=2$ is valid on any time interval $[0,T]$ (and not only on its special subinterval $[0,T_*]$).
\begin{theorem}[Spatial regularity of solution for arbitrarily large data.]
\label{NS-problemTh-sigma-Lv-er2C}
Let $n= 2$, $r>0$, and $T>0$. 
Let the coefficients $a_{ij}^{\alpha\beta}$ be constant and the relaxed ellipticity condition \eqref{mu} hold.
Let ${\mathbf f}\in L_2(0,T;\dot{\mathbf H}_\#^{r-1})$  and $\mathbf u^0\in \dot{\mathbf H}_{\#\sigma}^{r}$.

Then the solution $\mathbf u$ of the anisotropic Navier-Stokes 
initial value problem  \eqref{NS-problem-div0}--\eqref{NS-problem-div0-IC}
obtained in Theorem \ref{NS-problemTh-sigma}  is of Serrin-type and 
belongs to
$L_{\infty}(0,T;\dot{\mathbf H}_{\#\sigma}^{r})\cap L_2(0,T;\dot{\mathbf H}_{\#\sigma}^{r+1})$.
In addition, 
$
\mathbf u'\in L_2(0,T;\dot{\mathbf H}_{\#\sigma}^{r-1}),
$
$
\mathbf u \in\mathcal C^0([0,T];\dot{\mathbf H}_{\#\sigma}^{r})
$
$
\lim_{t\to 0}\| \mathbf u(\cdot,t)-{\mathbf u}^0\| _{\dot{\mathbf H}_{\#\sigma}^{r}}= 0
$
and $p\in L_2(0,T_*;\dot{H}_{\#}^{r})$. 
\end{theorem}
\begin{proof}
The proof coincides word-for-word with the proof of Theorem \ref{NS-problemTh-sigma-Lv-erC} if we take there $n=2$ while replacing $T_*$ by $T$ and the reference to \eqref{E4.14varbC} for the boundedness of the sequence $\|\mathbf{u}_m\|_{L_2(0,T_*;\mathbf H_\#^{n/2})}$ for $n=2$ by the reference to the corresponding inequality
\begin{align*}
&\|\mathbf u_m\|^2_{L_2(0,T;\dot{\mathbf H}_\#^{1})}
\le 4C_{\mathbb A}\left( \| {\mathbf u}^0\| ^2_{\mathbf L_{2\#}}
+4 C_{\mathbb A}\|\mathbf f\|^2_{L_2(0,T;\dot{\mathbf H}_\#^{-1})}\right).
\end{align*}
 obtained as inequality (59) in our paper \cite{Mikhailov2024}. 
\end{proof}

The following assertion can be proved similarly to Corollary \ref{Cor5.11spC}. 
\begin{corollary}
\label{Cor5.11-2C}
Let $T>0$ and $n= 2$. 
Let the coefficients $a_{ij}^{\alpha\beta}$ be constant and the relaxed ellipticity condition \eqref{mu} hold.
Let ${\mathbf f}\in L_2(0,T;\dot{\mathbf C}_\#^{\infty})$  
and $\mathbf u^0\in \dot{\mathbf C}_{\#\sigma}^{\infty}$.

Then the solution $\mathbf u$ of the anisotropic Navier-Stokes 
initial value problem  \eqref{NS-problem-div0}--\eqref{NS-problem-div0-IC}
obtained in Theorem \ref{NS-problemTh-sigma}  is of  Serrin-type and 
is such that
${\mathbf u}\in \mathcal C^0([0,T];\dot{\mathbf C}_{\#\sigma}^{\infty})$, 
$\mathbf u'\in L_2(0,T;\dot{\mathbf C}_{\#\sigma}^{\infty})$ and 
$p\in L_2(0,T;\dot{\mathcal C}_{\#}^{\infty})$.
\end{corollary}

The next three assertions on spatial-temporal regularity for $n=2$ are the corresponding counterparts of Theorems \ref{NS-problemTh-sigma-Lv-er-t1C}, \ref{NS-problemTh-sigma-Lv-er-tC} and Corollary \ref{Cor5.11C} and are proved in a similar way after replacing there $T_*$ by $T$.
\begin{theorem}
\label{NS-problemTh-sigma-Lv-er-t1-2C}
Let $T>0$, $n\ge 2$ and  $r> 0$. 
Let the coefficients $a_{ij}^{\alpha\beta}$ be constant the relaxed ellipticity condition \eqref{mu} hold.
Let ${\mathbf f}\in L_\infty(0,T;\dot{\mathbf H}_\#^{r-2})\cap L_2(0,T;\dot{\mathbf H}_\#^{r-1})$
and $\mathbf u^0\in \dot{\mathbf H}_{\#\sigma}^{r}$.

Then the solution $\mathbf u$ of the anisotropic Navier-Stokes 
initial value problem  \eqref{NS-problem-div0}--\eqref{NS-problem-div0-IC}
obtained in Theorem \ref{NS-problemTh-sigma}  is of Serrin-type and 
is such that
$\mathbf u'\in L_{\infty}(0,T;\dot{\mathbf H}_{\#\sigma}^{r-2})\cup L_2(0,T;\dot{\mathbf H}_{\#\sigma}^{r-1})$,
$p\in L_\infty(0,T;\dot{H}_{\#}^{r-1})\cap L_2(0,T;\dot{H}_{\#}^{r})$.
\end{theorem}

\begin{theorem}
\label{NS-problemTh-sigma-Lv-er-t-2C}
Let $T>0$, $n= 2$ and  $r> 0$. 
Let the coefficients $a_{ij}^{\alpha\beta}$ be constant the relaxed ellipticity condition \eqref{mu} hold.
Let $k\in [1, r+1)$ be an integer. 
Let  
${\mathbf f}^{(l)}\in L_\infty(0,T;\dot{\mathbf H}_\#^{r-2-2l})\cap L_2(0,T;\dot{\mathbf H}_\#^{r-1-2l})$, $l=0,1,\ldots,k-1$,
and $\mathbf u^0\in \dot{\mathbf H}_{\#\sigma}^{r}$.

Then the solution $\mathbf u$ of the anisotropic Navier-Stokes 
initial value problem  \eqref{NS-problem-div0}--\eqref{NS-problem-div0-IC}
obtained in Theorem \ref{NS-problemTh-sigma}  is of  Serrin-type and 
is such that
$\mathbf u^{(l)}\in L_{\infty}(0,T;\dot{\mathbf H}_{\#\sigma}^{r-2l})\cap L_2(0,T;\dot{\mathbf H}_{\#\sigma}^{r+1-2l})$, $l=0,\ldots,k$;
$p^{(l)}\in L_\infty(0,T;\dot{H}_{\#}^{r-1-2l})\cap L_2(0,T;\dot{H}_{\#\sigma}^{r-2l})$, $l=0,\ldots,k-1$.
\end{theorem}

\begin{corollary}
Let $T>0$ and $n= 2$. 
Let the coefficients $a_{ij}^{\alpha\beta}$ be constant the relaxed ellipticity condition \eqref{mu} hold.
Let ${\mathbf f}\in \mathcal C^\infty(0,T;\dot{\mathbf C}_\#^{\infty})$,  
and $\mathbf u^0\in \dot{\mathbf C}_{\#\sigma}^{\infty}$.

Then the solution $\mathbf u$ of the anisotropic Navier-Stokes 
initial value problem  \eqref{NS-problem-div0}--\eqref{NS-problem-div0-IC}
obtained in Theorem \ref{NS-problemTh-sigma}  is of Serrin-type and 
is such that
$\mathbf u\in \mathcal C^{\infty}(0,T;\dot{\mathbf C}_{\#\sigma}^{\infty})$,
$p\in \mathcal C^\infty(0,T;\dot{\mathcal C}_{\#}^\infty)$.
\end{corollary}

\section{Serrin-type solution existence and regularity for variable anisotropic viscosity coefficients}\label{S5ERV}
In this section we generalise to the anisotropic variable  viscosity coefficients the analysis of the existence and regularity of Serrin-type solutions for any $n\ge 2$ given in Section  \ref{S5ERC} for the anisotropic constant viscosity coefficients. 

\subsection{Preliminary results 
for variable anisotropic viscosity coefficients}\label{S.5.2}

For some  $n\ge 2$, $r\ge{n}/{2}-1$, and $T>0$, let  
$a_{ij}^{\alpha \beta }\in L_\infty([0,T];H_\#^{\tilde\sigma+1})$, 
$\tilde\sigma>\frac{n}{2}+\max\{|r-1|,|n/2-2|\}$,
 and the relaxed ellipticity condition \eqref{mu} hold.
Let  also ${\mathbf f}\in L_2(0,T;\dot{\mathbf H}_\#^{r-1})$  and $\mathbf u^0\in \dot{\mathbf H}_{\#\sigma}^{r}$.

We employ the Galerkin approximation as in Section \ref{S.5.2C} and repeating the same arguments arrive at the same  equations \eqref{E4.93Tvar0C} but now with the variable coefficients $a_{ij}^{\alpha \beta }(\mathbf x,t)$.
These equations can be now re-written as
\begin{multline}\label{E4.93Tvar}
\langle\partial_t\Lambda^{r}_\#\widetilde{\boldsymbol u}_m, \Lambda^{r}_\#\mathbf{w}_k\rangle_\T
+\left\langle a_{ij}^{\alpha \beta }E_{j\beta }(\Lambda^{r}_\#\widetilde{\mathbf u}_m),
E_{i\alpha }(\Lambda^{r}_\#{\mathbf w}_k)\right\rangle _{\T}
+\langle\Lambda^{r-1}_\#[({\mathbf u}_m\cdot \nabla ){\mathbf u}_m],\Lambda^{r+1}_\#{\bf w}_k\rangle _{\T}
\\
\hspace{1em}
=\langle\Lambda^{r-1}_\#\mathbf f, \Lambda^{r+1}_\#  \mathbf{w}_k\rangle_\T
+\langle{\nabla\Lambda^{r}_\#\mathbf v}_m,\nabla\Lambda^{r}_\#{\bf w}_k\rangle_\T
-\left\langle E_{j\beta }({\bf v}_m),
a_{ij}^{\alpha \beta }\Lambda^{r}_\#E_{i\alpha }(\Lambda^{r}_\#{\mathbf w}_k)\right\rangle _{\T}
\\
-\left\langle E_{j\beta }(\widetilde{\mathbf u}_m),
a_{ij}^{\alpha \beta }\Lambda^{r}_\#E_{i\alpha }(\Lambda^{r}_\#{\mathbf w}_k)
-\Lambda^{r}_\#[a_{ij}^{\alpha \beta }E_{i\alpha }(\Lambda^{r}_\#{\mathbf w}_k)]\right\rangle _{\T},
\quad \forall\, k\in \{1,\ldots,m\} .
\end{multline}
Multiplying equations in \eqref{E4.93Tvar} by $\widetilde{\eta}_{k,m}(t)$ and summing them up over $k\in \{1,\ldots,m\}$, we obtain
\begin{multline}\label{E4.88Tvar}
\frac{1}{2} \partial_t\left\|\Lambda^{r}_\#\widetilde{\mathbf u}_m\right\|_{\mathbf H^{0}_{\#}}^2
+a_{\T}(t;\Lambda^{r}_\#\widetilde{\mathbf u}_m,\Lambda^{r}_\#\widetilde{\mathbf u}_m)
=\langle\Lambda^{r-1}_\#\mathbf f, \Lambda^{r+1}_\# \widetilde{\mathbf u}_m\rangle_\T
+\langle{\nabla\Lambda^{r}_\#\mathbf v}_m,\nabla\Lambda^{r}_\#\widetilde{\mathbf u}_m\rangle_\T
\\
-\left\langle E_{j\beta }({\bf v}_m),
a_{ij}^{\alpha \beta }\Lambda^{r}_\#E_{i\alpha }(\Lambda^{r}_\#\widetilde{\mathbf u}_m)\right\rangle _{\T}
-\left\langle E_{j\beta }(\widetilde{\mathbf u}_m),
a_{ij}^{\alpha \beta }\Lambda^{r}_\#E_{i\alpha }(\Lambda^{r}_\#\widetilde{\mathbf u}_m)
-\Lambda^{r}_\#[a_{ij}^{\alpha \beta }E_{i\alpha }(\Lambda^{r}_\#\widetilde{\mathbf u}_m)]\right\rangle _{\T}
\\
-\langle\Lambda^{r-1}_\#[({\mathbf u}_m\cdot \nabla ){\mathbf u}_m],\Lambda^{r+1}_\#\widetilde{\mathbf u}_m\rangle _{\T}.
\end{multline}

From \eqref{NS-a-1-v2-S-} we have
\begin{align}
\label{NS-a-1-v2-S-Tvar}
a_{\T}(t;\Lambda^{r}_\#\widetilde{\mathbf u}_m,\Lambda^{r}_\#\widetilde{\mathbf u}_m)
&\geq \frac14 C_{\mathbb A}^{-1}\|\Lambda^{r}_\#\widetilde{\mathbf u}_m\|_{\dot{\mathbf H}_{\#\sigma}^1}^2
=\frac14 C_{\mathbb A}^{-1}\|\widetilde{\mathbf u}_m\|^2_{\dot{\mathbf H}_{\#\sigma}^{r+1}}.
\end{align}

Let us now estimate the terms in the right hand side of \eqref{E4.88Tvar}.
For the first two terms and for the last one, estimates \eqref{E4.123cvar}, \eqref{E4.121cvar} and \eqref{E4.124C} still hold.

Further, since  $\tilde\sigma+1>n/2$, we obtain by Theorem \ref{RS-T1-S4.6.1}(a), relation \eqref{strain-r}, and inequality \eqref{eq:mik14},
\begin{multline}\label{E4.122var}
\left|\left\langle E_{j\beta }({\bf v}_m),
a_{ij}^{\alpha \beta }\Lambda^{r}_\#E_{i\alpha }(\Lambda^{r}_\#\widetilde{\mathbf u}_m)\right\rangle _{\T}\right|
\le \|E_{j\beta }({\mathbf v}_m)\|_{({H}_{\#}^{r})^{n\times n}}  
\|a_{ij}^{\alpha \beta }\Lambda^{r}_\#E_{i\alpha }(\Lambda^{r}_\#\widetilde{\mathbf u}_m)\|_{({H}_{\#}^{-r})^{n\times n}}
\\
\le \|E_{j\beta }({\mathbf v}_m)\|_{({H}_{\#}^{r})^{n\times n}} C_{*\tilde\sigma rn} \|a_{ij}^{\alpha \beta }\|_{H_\#^{\tilde\sigma+1}} \|\Lambda^{r}_\# E_{i\alpha }(\Lambda^{r}_\#\widetilde{\mathbf u}_m)\|_{({H}_{\#}^{-r})^{n\times n}}
\\
\le  C_{*\tilde\sigma rn} \|\mathbb A\|_{H_\#^{\tilde\sigma+1},F}\, \|\mathbf v_m\|_{\dot{\mathbf H}_{\#}^{r+1}} 
\|\widetilde{\mathbf u}_m\|_{\dot{\mathbf H}_{\#}^{r+1}},
\end{multline}
where 
$C_{*\tilde\sigma rn}:=C_*(-r,\tilde\sigma+1,n)$,
$\|\mathbb A(\cdot,t)\|_{H_\#^{\tilde\sigma+1},F}
:= \left|\left\{\|a_{ij}^{\alpha \beta }(\cdot,t)\|_{H_\#^{\tilde\sigma+1}}\right\}_{\alpha,\beta,i,j=1}^n\right|_{F}$.

By Theorem \ref{Jcomm} with $\theta=r$, $s=0$,
\begin{multline}\label{E4.122var-}
\left|\left\langle E_{j\beta }(\widetilde{\mathbf u}_m),
a_{ij}^{\alpha \beta }\Lambda^{r}_\#E_{i\alpha }(\Lambda^{r}_\#\widetilde{\mathbf u}_m)
-\Lambda^{r}_\#[a_{ij}^{\alpha \beta }E_{i\alpha }(\Lambda^{r}_\#\widetilde{\mathbf u}_m)]\right\rangle _{\T}\right|
\\
\le \|E_{j\beta }(\widetilde{\mathbf u}_m)\|_{({ H}_{\#}^{r-1})^{n\times n}}  
\|a_{ij}^{\alpha \beta }\Lambda^{r}_\#E_{i\alpha }(\Lambda^{r}_\#\widetilde{\mathbf u}_m)
-\Lambda^{r}_\#[a_{ij}^{\alpha \beta }E_{i\alpha }(\Lambda^{r}_\#\widetilde{\mathbf u}_m)]\|_{({H}_{\#}^{-r+1})^{n\times n}}
\\
\le \bar C_{0,r,\tilde\sigma}\|E_{j\beta }(\widetilde{\mathbf u}_m)\|_{({H}_{\#}^{r-1})^{n\times n}} |a_{ij}^{\alpha \beta }|_{H_\#^{\tilde\sigma+1}}
\|E_{i\alpha }(\Lambda^{r}_\#\widetilde{\mathbf u}_m)\|_{(H_\#^0))^{n\times n}} 
\le \bar C_{0,r,\tilde\sigma} |\mathbb A|_{H_\#^{\tilde\sigma+1},F} 
\|\widetilde{\mathbf u}_m\|_{\dot{\mathbf H}_{\#}^{r}} 
\|\widetilde{\mathbf u}_m\|_{\dot{\mathbf H}_{\#}^{r+1}},
\end{multline}
where 
$|\mathbb A(\cdot,t)|_{H_\#^{\tilde\sigma+1},F}
:= \left|\left\{|a_{ij}^{\alpha \beta }(\cdot,t)|_{H_\#^{\tilde\sigma+1}}\right\}_{\alpha,\beta,i,j=1}^n\right|_{F}
\le \|\mathbb A(\cdot,t)\|_{H_\#^{\tilde\sigma+1},F}$.

Implementing estimates \eqref{E4.123cvar}, \eqref{E4.121cvar}, \eqref{E4.124C}, \eqref{NS-a-1-v2-S-Tvar}, \eqref{E4.122var}, \eqref{E4.122var-} in \eqref{E4.88Tvar} and using Young's inequality, we obtain
\begin{multline*}
\frac{d}{d t}\left\|\widetilde{\mathbf u}_m\right\|_{\mathbf H^{r}_{\#}}^2
+\frac12 C_{\mathbb A}^{-1}\|\widetilde{\mathbf u}_m\|_{\mathbf{H}_\#^{r+1}}^2
\leq 2\Big(\|\mathbf f\|_{\mathbf{H}_\#^{r-1}}
+\big[C_{*\tilde\sigma rn}\|\mathbb A\|_{H_\#^{\tilde\sigma+1},F}+1\big] \|\mathbf v_m\|_{\dot{\mathbf H}_{\#}^{r+1}}
\\
\shoveright{
+\bar C_{0,r,\tilde\sigma} |\mathbb A|_{H_\#^{\tilde\sigma+1},F} 
\|\widetilde{\mathbf u}_m\|_{\dot{\mathbf H}_{\#}^{r}} 
+\left\|({\mathbf u}_m\cdot \nabla ){\mathbf u}_m\right\|_{\mathbf{H}_\#^{r-1}}
\Big) \| \widetilde{\mathbf u}_m\|_{\mathbf{H}_\#^{r+1}}
}
\\
\leq 4 C_{\mathbb A}\Big(\|\mathbf f\|_{\mathbf{H}_\#^{r-1}}
+\big[C_{*\tilde\sigma rn}\|\mathbb A\|_{H_\#^{\tilde\sigma+1},F}+1\big]  \|\mathbf v_m\|_{\dot{\mathbf H}_{\#}^{r+1}}
\\
+\bar C_{0,r,\tilde\sigma} |\mathbb A|_{H_\#^{\tilde\sigma+1},F} \|\widetilde{\mathbf u}_m\|_{\dot{\mathbf H}_{\#}^{r}}
+\left\|({\mathbf u}_m\cdot \nabla ){\mathbf u}_m\right\|_{\mathbf{H}_\#^{r-1}} 
\Big)^2
+  \frac14 C_{\mathbb A}^{-1}\| \widetilde{\mathbf u}_m\|_{\mathbf{H}_\#^{r+1}}^2.
\end{multline*}
Hence by the inequality $(\sum_{i=1}^k a_i)^2\le k\sum_{i=1}^k a_i^2$ (following from the Cauchy–Schwarz inequality), \begin{multline}\label{E4.82T}
\frac{d}{d t}\left\|\widetilde{\mathbf u}_m\right\|_{\mathbf H^{r}_{\#}}^2
+\frac14 C_{\mathbb A}^{-1}\|\widetilde{\mathbf u}_m\|_{\mathbf{H}_\#^{r+1}}^2
\leq 20 C_{\mathbb A}\Big(\|\mathbf f\|_{\mathbf{H}_\#^{r-1}}^2
+ \big[C_{*\tilde\sigma rn}^2\|\mathbb A\|^2_{H_\#^{\tilde\sigma+1},F}+1\big]\|\mathbf v_m\|_{\dot{\mathbf H}_{\#}^{r+1}}^2
\\
+\bar C_{0,r,\tilde\sigma}^2 |\mathbb A|^2_{H_\#^{\tilde\sigma+1},F}\|\widetilde{\mathbf u}_m\|^2_{\dot{\mathbf H}_{\#}^{r}}
+\left\|({\mathbf u}_m\cdot \nabla ){\mathbf u}_m\right\|_{\mathbf{H}_\#^{r-1}}^2
\Big).
\end{multline}

Note that by the similar reasoning, but without employing in \eqref{NS-eq:mik51adTv}--\eqref{NS-eq:mik51indTv} the function $\mathbf v$, we obtain that $\mathbf u_m$ satisfies the differential inequality
\begin{multline}\label{E4.82T0}
\frac{d}{d t}\left\|\mathbf u_m\right\|_{\mathbf H^{r}_{\#}}^2
+\frac14 C_{\mathbb A}^{-1}\|\mathbf u_m\|_{\mathbf{H}_\#^{r+1}}^2
\leq 12 C_{\mathbb A}\Big(\|\mathbf f\|_{\mathbf{H}_\#^{r-1}}^2
+\bar C_{0,r,\tilde\sigma}^2 |\mathbb A|^2_{H_\#^{\tilde\sigma+1},F}\|\mathbf u_m\|^2_{\dot{\mathbf H}_{\#}^{r}}
+\left\|({\mathbf u}_m\cdot \nabla ){\mathbf u}_m\right\|_{\mathbf{H}_\#^{r-1}}^2
\Big).
\end{multline}

\subsection{Serrin-type solution existence for variable anisotropic viscosity coefficients}\label{S5.3}
Employing the results from Section \ref{S.5.2} for $r=n/2-1$, we are now in the position to prove the existence of Serrin-type solutions.
\begin{theorem}
\label{NS-problemTh-sigma-Lv-cr}
Let $n\ge 2$, and $T>0$. 
Let  
$a_{ij}^{\alpha \beta }\in L_\infty([0,T];H_\#^{\tilde\sigma+1})$, 
$\tilde\sigma>\frac{n}{2}+|n/2-2|$,
 and the relaxed ellipticity condition \eqref{mu} hold.
Let ${\mathbf f}\in L_2(0,T;\dot{\mathbf H}_\#^{n/2-2})$  and $\mathbf u^0\in \dot{\mathbf H}_{\#\sigma}^{n/2-1}$.

(i) Then there  exist constants $A_{1}\ge 0$, $A_{2}\ge0$ and $A_{3}>0$ that are independent of ${\mathbf f}$ and $\mathbf u^0$ but may depend on $T$, $n$,  $\|\mathbb A\|$ and $C_{\mathbb A}$, such that
if ${\mathbf f}$, $\mathbf u^0$ and $T_*\in(0,T]$ satisfy the inequality
\begin{align}\label{E4.143T*var}
\int_0^{T_*}\|\mathbf f(\cdot,t)\|_{\mathbf{H}_\#^{n/2-2}}^2dt
+ \left (A_{1}\|{\mathbf u}^0\|^2_{{\mathbf H}_\#^{n/2-1}} + A_{2}\right) 
\int_0^{T_*}\|(K{\mathbf u}^0)(\cdot,t)\|_{\dot{\mathbf H}_{\#}^{n/2}}^2dt
<A_{3},
\end{align}  
where $K$ is the operator defined in \eqref{K-def}, 
then there exists a solution $\mathbf u$ of the anisotropic Navier-Stokes 
initial value problem  \eqref{NS-problem-div0}--\eqref{NS-problem-div0-IC} in 
$L_{\infty}(0,T_*;\dot{\mathbf H}_{\#\sigma}^{n/2-1})\cap L_2(0,T_*;\dot{\mathbf H}_{\#\sigma}^{n/2})$, which is thus a Serrin-type solution.

(ii) In addition,
$
\mathbf u'\in L_2(0,T_*;\dot{\mathbf H}_{\#\sigma}^{n/2-2}),
$
$
\mathbf u \in\mathcal C^0([0,T_*];\dot{\mathbf H}_{\#\sigma}^{n/2-1})
$,
$
\lim_{t\to 0}\| \mathbf u(\cdot,t)-{\mathbf u}^0\| _{\dot{\mathbf H}_{\#\sigma}^{n/2-1}}= 0,
$
and $p\in L_2(0,T_*;\dot{H}_{\#}^{n/2-1})$. 

(iii) Moreover, $\mathbf u$ satisfies  the following energy equality for any $[t_0,t]\subset[0,T_*]$,
\begin{align}\label{NS-eq:mik51ai=aTvar}
\frac12\| {\mathbf u}(\cdot,t)\| _{\mathbf L_{2\#}}^2
+\int_{t_0}^{t}a_{\T}(\tau;{\mathbf u}(\cdot,\tau),{\mathbf u}(\cdot,\tau)) d\tau
=  \frac12\| {\mathbf u}(\cdot,t_0)\| ^2_{\mathbf L_{2\#}}
+\int_{t_0}^{t}\langle\mathbf{f}(\cdot,\tau), {\mathbf u}(\cdot,\tau)\rangle_{\T}d\tau.
\end{align}
It particularly implies the standard energy equality,
\begin{align}
\label{E4.9Tvar}
\frac12\ {\mathbf u}(\cdot,t)\| _{\mathbf L_{2\#}}^2
+\int_0^ta_{\T}(\tau;{\mathbf u}(\cdot,\tau),{\mathbf u}(\cdot,\tau)) d\tau
= \frac12\|{\mathbf u}^0\| ^2_{\mathbf L_{2\#}}
+\int_0^t\langle\mathbf{f}(\cdot,\tau), {\mathbf u}(\cdot,\tau)\rangle_{\T}\,d\tau
\quad \forall\,
t\in[0,T_*].
\end{align}

(iv) The solution $\mathbf u$ is unique in the class of solutions from
$ L_{\infty}(0,T_*;\dot{\mathbf H}_{\#\sigma}^{0})\cap L_2(0,T_*;\dot{\mathbf H}_{\#\sigma}^{1})$
satisfying the energy inequality \eqref{E4.9} on the interval $[0,T_*]$. 
\end{theorem}

\begin{proof}
(i)
Let $r=n/2-1$.
The estimate \eqref{E4.91-0T0cC} still holds.
Let us fix any small  $\tilde\sigma_n$ such that $\tilde\sigma\ge\tilde\sigma_n>{n}/{2}+|n/2-2|=\max\{2,n-2\}$.  
Then by \eqref{E4.91-0TC} we obtain from \eqref{E4.82T}, 
\begin{multline}\label{E4.93BTb}
\frac{d}{d t}\left\|\widetilde{\mathbf u}_m\right\|_{\mathbf H^{n/2-1}_{\#}}^2
+\frac14 C_{\mathbb A}^{-1}\|\widetilde{\mathbf u}_m\|_{\mathbf{H}_\#^{n/2}}^2
\leq 
\left(160C^2_{*rn}C_{\mathbb A}
\|\widetilde{\mathbf u}_m\|^2_{{\mathbf H}_\#^{n/2}}
+20C_{\mathbb A}\bar C_{n}^2 |\mathbb A|^2_{H_\#^{\tilde\sigma_n+1},F}\right)
\|\widetilde{\mathbf u}_m\|^2_{\dot{\mathbf H}_{\#}^{n/2-1}} 
\\
+ 20 C_{\mathbb A}\left(\|\mathbf f\|_{\mathbf{H}_\#^{n/2-2}}^2
+ 8C^2_{*rn}\|{\mathbf v}_m\|^2_{{\mathbf H}_\#^{n/2-1}}\|{\mathbf v}_m\|^2_{{\mathbf H}_\#^{n/2}}
+ \big[\widetilde{C}_{*n}^2\|\mathbb A\|^2_{H_\#^{\tilde\sigma_n+1},F}+1\big] \|\mathbf v_m\|_{\dot{\mathbf H}_{\#}^{n/2}}^2
\right),
\end{multline}
where 
$\bar C_{n}:=\bar C_{0,n/2-1,\tilde\sigma_n}$, $\widetilde{C}_{*n}:=C_{*\tilde\sigma_n, n/2-1,n}=C_*(-n/2+1,\tilde\sigma_n+1,n)$.

Let us now apply to \eqref{E4.93BTb} Lemma \ref{RRS2016-L10.3phi} with 
\begin{align*}
&\eta=\left\|\widetilde{\mathbf u}_m\right\|_{\mathbf H^{n/2-1}_{\#}}^2,\
\eta_0=0,\
y=\|\widetilde{\mathbf u}_m\|_{\mathbf{H}_\#^{n/2}}^2,\ 
b=\frac14 C_{\mathbb A}^{-1},\ 
c=160C^2_{*rn}C_{\mathbb A},\
\phi=20 C_{\mathbb A}\bar C_{n}^2 |\mathbb A|^2_{H_\#^{\tilde\sigma_n+1},F}\\
&\psi=20 C_{\mathbb A}\left(\|\mathbf f\|_{\mathbf{H}_\#^{n/2-2}}^2
+ 8C^2_{*rn}\|{\mathbf v}_m\|^2_{{\mathbf H}_\#^{n/2-1}}\|{\mathbf v}_m\|^2_{{\mathbf H}_\#^{n/2}}
+\big[\widetilde{C}_{*n}^2\|\mathbb A\|^2_{H_\#^{\tilde\sigma_n+1},F}+1\big]\|\mathbf v_m\|_{\dot{\mathbf H}_{\#}^{n/2}}^2
\right)
\end{align*}
to conclude that if $T_*$ is such that 
\begin{multline}\label{E5.32b}
\int_0^{T_*}e^{\Phi(T_*)-\Phi(t)}\left(\|\mathbf f(\cdot,t)\|_{\mathbf{H}_\#^{n/2-2}}^2
+ \left(8C^2_{*rn}\|{\mathbf v}_m(\cdot,t)\|^2_{{\mathbf H}_\#^{n/2-1}}
+\big[\widetilde{C}_{*n}^2\|\mathbb A\|^2_{H_\#^{\tilde\sigma_n+1},F}+1\big]\right) 
\|\mathbf {\mathbf v}_m(\cdot,t)\|_{\dot{\mathbf H}_{\#}^{n/2}}^2\right)dt
\\
<\left(640eC_{\mathbb A}^2C^2_{*rn}\right)^{-1},
\end{multline}
where 
$$\Phi(s):=\int_0^s \phi(\tau)d\tau=20C_{\mathbb A}\bar C_{n}^2\int_0^s  |\mathbb A(\cdot,\tau)|^2_{H_\#^{\tilde\sigma_n+1},F}d\tau,$$
 then
\begin{align}
&\|{\mathbf u}_m\|_{L_\infty(0,T_*;\dot{\mathbf H}_{\#\sigma}^{n/2-1})}
\le\|\widetilde {\mathbf u}_m\|_{L_\infty(0,T_*;\dot{\mathbf H}_{\#\sigma}^{n/2-1})} 
+\|{\mathbf v}_m\|_{L_\infty(0,T_*;\dot{\mathbf H}_{\#\sigma}^{n/2-1})}
\le \left(8\sqrt{10}C_{\mathbb A}{C}_{*rn}\right)^{-1}+\left\|{\mathbf u}^0\right\|_{\dot{\mathbf H}^{n/2-1}_{\#\sigma}},
\label{E4.143avarb}\\
&\|{\mathbf u}_m\|_{L_2(0,T_*;\dot{\mathbf H}_{\#\sigma}^{n/2})}
\le\|\widetilde {\mathbf u}_m\|_{L_2(0,T_*;\dot{\mathbf H}_{\#\sigma}^{n/2})} 
+\|{\mathbf v}_m\|_{L_2(0,T_*;\dot{\mathbf H}_{\#\sigma}^{n/2})}
\le \left(4\sqrt{10C_{\mathbb A}}{C}_{*rn}\right)^{-1}
+\left\|{\mathbf u}^0\right\|_{\dot{\mathbf H}^{n/2-1}_{\#\sigma}}.
\label{E4.14varb}
\end{align}
Estimates \eqref{E4.122v-mC} and \eqref{E4.122va-mC} were taken into account in \eqref{E4.143avarb} and \eqref{E4.14varb}.

Taking into account inequality \eqref{E4.122v-mC}, we obtain that condition \eqref{E5.32b} is satisfied if $T_*$ is such that 
\begin{multline}\label{E4.143b}
\int_0^{T_*}\|\mathbf f(\cdot,t)\|_{\mathbf{H}_\#^{n/2-2}}^2dt
+ \left (8C^2_{*rn}\|{\mathbf u}^0\|^2_{{\mathbf H}_\#^{n/2-1}}
+\big[\widetilde{C}_{*n}^2\|\mathbb A\|^2_{L_\infty(0,T;H_\#^{\tilde\sigma_n+1}),F}+1\big]\right) \int_0^{T_*}\|\mathbf v(\cdot,t)\|_{\dot{\mathbf H}_{\#}^{n/2}}^2dt
\\
<\left(640C_{\mathbb A}^2C^2_{*rn}\right)^{-1}
\exp\left(-1-20 C_{\mathbb A}\bar C_{n}^2 \|\mathbb A\|^2_{L_\infty(0,T;H_\#^{\tilde\sigma_n+1}),F}\,T\right).
\end{multline}

Note that condition \eqref{E4.143b} gives condition \eqref{E4.143T*var} with 
\begin{align*}
&A_{1}=8C^2_{*rn}, \quad 
A_{2}=\widetilde{C}_{*n}^2\|\mathbb A\|^2_{L_\infty(0,T;H_\#^{\tilde\sigma_n+1}),F}+1,\\
&A_{3}=\left(640C_{\mathbb A}^2C^2_{*rn}\right)^{-1}
\exp\left(-1-20 C_{\mathbb A}\bar C_{n}^2 \|\mathbb A\|^2_{L_\infty(0,T;H_\#^{\tilde\sigma+1}),F}\,T\right).
\end{align*}

Inequalities \eqref{E4.143avarb} and \eqref{E4.14varb} imply that  there exists a subsequence of $\{\mathbf u_m\}$ converging weakly  in $L_2(0, T_* ; \dot{\mathbf H}_{\#\sigma}^{n/2})$ and weakly-star in $L_\infty(0, T_* ; \dot{\mathbf H}_{\#\sigma}^{n/2-1})$ to a function 
${\mathbf u}^\dag\in L_2(0, T_* ; \dot{\mathbf H}_{\#\sigma}^{n/2})\cup L_\infty(0, T_* ; \dot{\mathbf H}_{\#\sigma}^{n/2-1})$.
Then the subsequence converges to ${\mathbf u}^\dag$ also weakly in $L_2(0,T_*;\dot{\mathbf H}_{\#\sigma}^{1})$ and weakly-star in $L_\infty(0,T_*;\dot{\mathbf H}_{\#\sigma}^{0})$.
Since $\{\mathbf u_m\}$ is the subsequence of the sequence that converges weakly in $L_2(0,T;\dot{\mathbf H}_{\#\sigma}^{1})$ and weakly-star in $L_\infty(0,T;\dot{\mathbf H}_{\#\sigma}^{0})$ to the weak solution, $\mathbf u$, of problem \eqref{NS-problem-div0}--\eqref{NS-problem-div0-IC} on $\left[0, T_*\right]$, we conclude that 
$\mathbf u={\mathbf u}^\dag\in L_\infty(0, T_* ; \dot{\mathbf H}_{\#\sigma}^{n/2-1})\cup L_2(0, T_* ; \dot{\mathbf H}_{\#\sigma}^{n/2})$. 

This implies that $\mathbf u$ is a Serrin-type solution on the interval $[0,T_*]$ and we thus proved item (i) of the theorem.

(ii) As in step (ii) of the proof of Theorem \ref{NS-problemTh-sigma-Lv-cr}, estimate \eqref{E4.91-0T0EE1r-uC1} implies that 
$({\mathbf u}\cdot \nabla ){\mathbf u}\in{L_2(0,T_*;\dot{\mathbf H}_\#^{n/2-2})}$.
By \eqref{L-oper} and \eqref{TensNorm} we have
\begin{align*}
\left\|\bs{\mathfrak L}\mathbf u\right\|^2_{{\mathbf H}_{\#}^{n/2-2}}
\le\|a_{ij}^{\alpha \beta }E_{i\alpha }({\mathbf u})\|_{({H}_{\#}^{n/2-1})^{n\times n}}
\le \|\mathbb A\|^2_{H_\#^{\tilde\sigma+1},F}\|\mathbf u\|^2_{{\mathbf H}_{\#}^{n/2}}
\end{align*}
and thus
\begin{align*}
\|\bs{\mathfrak L}\mathbf u\|^2_{L_2(0,T_*;\dot{\mathbf H}_{\#}^{n/2-2})}
&\le \|\mathbb A\|^2_{L_\infty(0,T_*;H_\#^{\tilde\sigma+1}),F}\|\mathbf u\|^2_{L_2(0,T_*;{\mathbf H}_{\#}^{n/2})},
\end{align*}
i.e., $\bs{\mathfrak L}\mathbf u\in{L_2(0,T_*;\dot{\mathbf H}_{\#\sigma}^{n/2-2})}$.
We also have ${\mathbf f}\in L_2(0,T;\dot{\mathbf H}_\#^{n/2-2})$.

Then \eqref{NS-problem-just} implies that 
${\mathbf u}'\in{L_2(0,T_*;\dot{\mathbf H}_{\#\sigma}^{n/2-2})}$ and hence by Theorem \ref{LM-T3.1} we obtain that 
$
\mathbf u \in\mathcal C^0([0,T_*];\dot{\mathbf H}_{\#\sigma}^{n/2-1})
$,
which also means 
that $
\| \mathbf u(\cdot,t)-{\mathbf u}^0\| _{\dot{\mathbf H}_{\#\sigma}^{n/2-1}}\to 0
$
as ${t\to 0}$.

To prove the theorem claim about the associated pressure $p$ we remark that it satisfies \eqref{Eq-p}, 
where  $\mathbf F\in L_2(0,T;\dot{\mathbf H}_{\#}^{n/2-2})$
due to the theorem conditions and the inclusion $({\mathbf u}\cdot \nabla ){\mathbf u}\in{L_2(0,T_*;\dot{\mathbf H}_\#^{n/2-2})}$.
By Lemma \ref{div-grad-is} for gradient, with $s=n/2-1$, equation \eqref{Eq-p} has a unique solution $p$ in 
$L_2(0,T_*;\dot{H}_{\#}^{n/2-1})$.

(iii) The energy equalities \eqref{NS-eq:mik51ai=aTvar} and \eqref{E4.9Tvar} immediately follow from Theorem \ref{Th4.4}. 

(iv) The solution uniqueness follows from Theorem \ref{Th4.5}.
\end{proof}

\begin{remark}
Note that by the Sobolev embedding theorem, the condition $a_{ij}^{\alpha \beta }\in L_\infty([0,T];H_\#^{\tilde\sigma+1})$, $\tilde\sigma>|n/2-2|+\frac{n}{2}$, in Theorem \ref{NS-problemTh-sigma-Lv-cr} and further on implies $a_{ij}^{\alpha \beta }\in L_\infty([0,T];\mathcal C_\#^{0})\subset L_\infty([0,T];L_{\infty\#})$.
\end{remark}

\begin{remark}
Since $\|\mathbf f(\cdot,t)\|_{\dot{\mathbf H}_\#^{n/2-2}}^2$ is integrable on $(0,T]$ by the theorem condition and 
$\|(K{\mathbf u}^0)(\cdot,t)\|_{\dot{\mathbf H}_{\#}^{n/2}}^2$ is integrable on $(0,\infty)$ by the inequality \eqref{E4.121},  we conclude that due to the absolute continuity of the Lebesgue integrals, for arbitrarily large data ${\mathbf f}\in L_2(0,T;\dot{\mathbf H}_\#^{n/2-2})$  and $\mathbf u^0\in \dot{\mathbf H}_{\#\sigma}^{n/2-1}$ there exists $T_*>0$ such that condition \eqref{E4.143T*var} holds.
\end{remark}
Estimating the integrand in the second integral in \eqref{E4.143T*var} according to \eqref{E4.121}, we arrive at the following assertion allowing an explicit estimate of $T_*$ for arbitrarily large data if ${\mathbf f}\in L_\infty(0,T;\dot{\mathbf H}_\#^{n/2-2})$.
\begin{corollary}[Serrin-type solution  for arbitrarily large data but small time or vice versa.]
\label{NS-problemCor-sigma-2}
Let $n\ge 2$ and $T>0$. 
Let $a_{ij}^{\alpha \beta }\in L_\infty([0,T];H_\#^{\tilde\sigma+1})$, 
$\tilde\sigma>|n/2-2|+\frac{n}{2}$,
 and the relaxed ellipticity condition hold.
Let ${\mathbf f}\in L_\infty(0,T;\dot{\mathbf H}_\#^{n/2-2})$  and $\mathbf u^0\in \dot{\mathbf H}_{\#\sigma}^{n/2}$.

Then there  exist constants $A_1, A_2, A_3>0$ that are independent of ${\mathbf f}$ and $\mathbf u^0$ but may depend on $T$, $n$,  $\|\mathbb A\|$ and $C_{\mathbb A}$, such that
if $T_*\in(0,T]$ satisfies the inequality
\begin{align}\label{E4.143T*}
T_*\left[\|\mathbf f\|_{L_\infty(0,T;\dot{\mathbf H}_\#^{n/2-2})}^2
+ \left (A_1\|{\mathbf u}^0\|^2_{{\mathbf H}_\#^{n/2-1}} + A_2\right) \|{\mathbf u}^0\|^2_{{\mathbf H}_\#^{n/2}}\right]
<A_3,
\end{align}  
then there exists a Serrin-type solution  $\mathbf u\in L_{\infty}(0,T_*;\dot{\mathbf H}_{\#\sigma}^{n/2-1})\cap L_2(0,T_*;\dot{\mathbf H}_{\#\sigma}^{n/2})$  of the anisotropic Navier-Stokes  initial value problem.
This solution satisfies items (ii)-(iv) in Theorem \ref{NS-problemTh-sigma-Lv-cr}
\end{corollary}

Estimating the second integral in \eqref{E4.143T*var} according to \eqref{E4.122va}, we arrive at the following assertion. 
\begin{corollary}[Existence of Serrin-type solution for arbitrary time but small data]
\label{NS-problemCor-sigma-1}
Let $n\ge 2$ and $T>0$. 
Let  $a_{ij}^{\alpha \beta }\in L_\infty([0,T];H_\#^{\tilde\sigma+1})$, 
$\tilde\sigma>|n/2-2|+\frac{n}{2}$,
 and the relaxed ellipticity condition hold.
Let ${\mathbf f}\in L_2(0,T;\dot{\mathbf H}_\#^{n/2-2})$  and $\mathbf u^0\in \dot{\mathbf H}_{\#\sigma}^{n/2-1}$.

Then there  exist constants $A_1, A_2, A_3>0$ that are independent of  ${\mathbf f}$ and $\mathbf u^0$ but may depend on $T$, $n$, $\|\mathbb A\|$ and $C_{\mathbb A}$, such that
if ${\mathbf f}$ and $\mathbf u^0$ satisfy the inequality
\begin{align}\label{E4.143T}
\|\mathbf f\|_{L_2(0,T;\dot{\mathbf H}_\#^{n/2-2})}^2
+ \left (A_1\|{\mathbf u}^0\|^2_{{\mathbf H}_\#^{n/2-1}} + A_2\right) \|{\mathbf u}^0\|^2_{{\mathbf H}_\#^{n/2-1}}
<A_3,
\end{align}  
then there exists a Serrin-type solution  $\mathbf u\in L_{\infty}(0,T;\dot{\mathbf H}_{\#\sigma}^{n/2-1})\cap L_2(0,T;\dot{\mathbf H}_{\#\sigma}^{n/2})$  of the anisotropic Navier-Stokes 
initial value problem.
This solution satisfies items (ii)-(iv) in Theorem \ref{NS-problemTh-sigma-Lv-cr} with $T_*=T$.
\end{corollary}



\subsection{Spatial regularity of Serrin-type solutions for variable anisotropic viscosity coefficients}\label{S5.4}

\begin{theorem}[Spatial regularity of Serrin-type solution for arbitrarily large data.]
\label{NS-problemTh-sigma-Lv-er}
Let $n\ge 2$, $r>{n}/{2}-1$, and $T>0$. 
Let  
$a_{ij}^{\alpha \beta }\in L_\infty([0,T];H_\#^{\tilde\sigma+1})$, 
$\tilde\sigma>\frac{n}{2}+\max\{|r-1|,|n/2-2|\}$,
 and the relaxed ellipticity condition \eqref{mu} hold.
Let ${\mathbf f}\in L_2(0,T;\dot{\mathbf H}_\#^{r-1})$  and $\mathbf u^0\in \dot{\mathbf H}_{\#\sigma}^{r}$, while ${\mathbf f}$, $\mathbf u^0$ and $T_*\in(0,T]$ satisfy  inequality \eqref{E4.143T*var} from Theorem \ref{NS-problemTh-sigma-Lv-cr}.

Then the Serrin-type solution $\mathbf u$ of the anisotropic Navier-Stokes 
initial value problem  \eqref{NS-problem-div0}--\eqref{NS-problem-div0-IC} 
belongs to
$L_{\infty}(0,T_*;\dot{\mathbf H}_{\#\sigma}^{r})\cap L_2(0,T_*;\dot{\mathbf H}_{\#\sigma}^{r+1})$.
In addition, 
$
\mathbf u'\in L_2(0,T_*;\dot{\mathbf H}_{\#\sigma}^{r-1}),
$
$
\mathbf u \in\mathcal C^0([0,T_*];\dot{\mathbf H}_{\#\sigma}^{r})
$,
$
\lim_{t\to 0}\| \mathbf u(\cdot,t)-{\mathbf u}^0\| _{\dot{\mathbf H}_{\#\sigma}^{r}}= 0
$
and $p\in L_2(0,T_*;\dot{H}_{\#}^{r})$. 
\end{theorem}
\begin{proof}
The existence of the Serrin-type solution $\mathbf u\in L_{\infty}(0,T_*;\dot{\mathbf H}_{\#\sigma}^{n/2-1})\cap L_2(0,T_*;\dot{\mathbf H}_{\#\sigma}^{n/2})$ is proved in Theorem \ref{NS-problemTh-sigma-Lv-cr}(i), and we will prove that it has a higher smoothness. 
We will employ the same Galerkin approximation used in Section \ref{S.5.2} and in the proof of Theorem \ref{NS-problemTh-sigma-Lv-cr}(i).

\paragraph{Step (a).} 
Let us estimate the last term in  \eqref{E4.82T0} for the case $n/2-1<r<n/2$. 
By \eqref{E4.91-0T0cC} we obtain from \eqref{E4.82T0}, 
\begin{align}\label{E4.93BTc}
\frac{d}{d t}\left\|\mathbf u_m\right\|_{\mathbf H^{r}_{\#}}^2
+\frac14 C_{\mathbb A}^{-1}\|\mathbf u_m\|_{\mathbf{H}_\#^{r+1}}^2
\leq 12 C_{\mathbb A}\Big(\bar C_{0,r,\tilde\sigma}^2 |\mathbb A|^2_{H_\#^{\tilde\sigma+1},F}
+C^2_{*rn} \|\mathbf{u}_m\|^2_{{\mathbf H}_\#^{n/2}}\Big)\|\mathbf u_m\|^2_{\dot{\mathbf H}_{\#}^{r}}
+12 C_{\mathbb A}\|\mathbf f\|_{\mathbf{H}_\#^{r-1}}^2
\end{align}
implying
\begin{align}\label{E4.93BTc1}
\frac{d}{d t}\left\|\mathbf u_m\right\|_{\mathbf H^{r}_{\#}}^2
\leq 12 C_{\mathbb A}\Big(\bar C_{0,r,\tilde\sigma}^2 |\mathbb A|^2_{H_\#^{\tilde\sigma+1},F}
+C^2_{*rn} \|\mathbf{u}_m\|^2_{{\mathbf H}_\#^{n/2}}\Big)\|\mathbf u_m\|^2_{\dot{\mathbf H}_{\#}^{r}}
+12 C_{\mathbb A}\|\mathbf f\|_{\mathbf{H}_\#^{r-1}}^2.
\end{align}
By Gronwall's inequality \eqref{E17}, we obtain from \eqref{E4.93BTc1} that
\begin{multline}\label{E5.69}
\|\mathbf u_m\|^2_{L_\infty(0,T_*;\dot{\mathbf H}_{\#\sigma}^{r})}
\le \exp\Big[12 C_{\mathbb A}\Big(\bar C_{0,r,\tilde\sigma}^2T_*
\| |\mathbb A|_{H_\#^{\tilde\sigma+1},F}\|^2_{L_\infty(0,T_*)}
+C^2_{*rn} \|\mathbf{u}_m\|^2_{L_2(0,T_*;\mathbf H_\#^{n/2})}\Big)\Big]
\\
\times \left[\|\mathbf u_m(\cdot,0)\|_{\mathbf H^{r}_{\#}}^2
+12 C_{\mathbb A}\|\mathbf f\|_{L_2(0,T_*;\mathbf{H}_\#^{r-1})}^2\right].
\end{multline}
We have $\|\mathbf u_m(\cdot,0)\|_{\mathbf H^{r}_{\#}}\le \|\mathbf u^0\|_{\mathbf H^{r}_{\#}}$ and
by \eqref{E4.14varb}, the sequence $\|\mathbf{u}_m\|_{L_2(0,T_*;\mathbf H_\#^{n/2})}$ is bounded. 
Then \eqref{E5.69} implies that the sequence $\|\mathbf u_m\|^2_{L_\infty(0,T_*;\dot{\mathbf H}_{\#\sigma}^{r})}$ is bounded as well.
Integrating \eqref{E4.93BTc}, we conclude that 
\begin{multline}\label{E5.70}
\|\mathbf u_m\|^2_{L_2(0,T_*;\dot{\mathbf H}_{\#\sigma}^{r+1})}
\le 48 C^2_{\mathbb A}\Big(\bar C_{0,r,\tilde\sigma}^2T_*
\| |\mathbb A|_{H_\#^{\tilde\sigma+1},F}\|^2_{L_\infty(0,T_*)}
+C^2_{*rn} \|\mathbf{u}_m\|^2_{L_2(0,T_*;\mathbf H_\#^{n/2})}\Big)
\|\mathbf u_m\|^2_{L_\infty(0,T_*;\dot{\mathbf H}_{\#\sigma}^{r})}
\\
+4C_{\mathbb A}\|\mathbf u_m(\cdot,0)\|_{\mathbf H^{r}_{\#}}^2
+48 C^2_{\mathbb A}\|\mathbf f\|_{L_2(0,T_*;\mathbf{H}_\#^{r-1})}^2.
\end{multline}
Inequalities \eqref{E5.69} and \eqref{E5.70} mean that the sequences
\begin{align}\label{E5.70a}
\{\|\mathbf u_m\|_{L_\infty(0,T_*;\dot{\mathbf H}_{\#\sigma}^{r})}\}_{m=1}^\infty
\text { and } \{\|\mathbf u_m\|_{L_2(0,T_*;\dot{\mathbf H}_{\#\sigma}^{r+1})}\}_{m=1}^\infty
\text{ are bounded for } n/2-1<r<n/2.
\end{align}

\paragraph{Step (b).} 
Let now $r=n/2$. 
By \eqref{E4.91-0T0dC} we obtain from \eqref{E4.82T0}, 
\begin{multline}\label{E4.93BTd}
\frac{d}{d t}\left\|\mathbf u_m\right\|_{\mathbf H^{n/2}_{\#}}^2
+\frac14 C_{\mathbb A}^{-1}\|\mathbf u_m\|_{\mathbf{H}_\#^{n/2+1}}^2
\leq 12 C_{\mathbb A}\Big(\bar C_{0,r,\tilde\sigma}^2 |\mathbb A|^2_{H_\#^{\tilde\sigma+1},F}
+C^2_{*rn} \|\mathbf{u}_m\|^2_{{\mathbf H}_\#^{n/2+1/2}}\Big)\|\mathbf u_m\|^2_{\dot{\mathbf H}_{\#}^{n/2}}
\\
+12 C_{\mathbb A}\|\mathbf f\|_{\mathbf{H}_\#^{n/2-1}}^2,
\end{multline}
implying
\begin{align}\label{E4.93BTd1}
\frac{d}{d t}\left\|\mathbf u_m\right\|_{\mathbf H^{n/2}_{\#}}^2
\leq 12 C_{\mathbb A}\Big(\bar C_{0,r,\tilde\sigma}^2 |\mathbb A|^2_{H_\#^{\tilde\sigma+1},F}
+C^2_{*rn} \|\mathbf{u}_m\|^2_{{\mathbf H}_\#^{n/2+1/2}}\Big)\|\mathbf u_m\|^2_{\dot{\mathbf H}_{\#}^{n/2}}
+12 C_{\mathbb A}\|\mathbf f\|_{\mathbf{H}_\#^{n/2-1}}^2\, .
\end{align}
By Gronwall's inequality \eqref{E17}, we obtain from \eqref{E4.93BTd1} that
\begin{multline}\label{E5.74}
\|\mathbf u_m\|^2_{L_\infty(0,T_*;\dot{\mathbf H}_{\#\sigma}^{n/2})}
\le \exp\Big[12 C_{\mathbb A}\Big(\bar C_{0,r,\tilde\sigma}^2T_*
\| |\mathbb A|_{H_\#^{\tilde\sigma+1},F}\|^2_{L_\infty(0,T_*)}
+C^2_{*rn} \|\mathbf{u}_m\|^2_{L_2(0,T_*;\mathbf H_\#^{n/2+1/2})}\Big)\Big]
\\
\times \left[\|\mathbf u_m(\cdot,0)\|_{\mathbf H^{n/2}_{\#}}^2
+12 C_{\mathbb A}\|\mathbf f\|_{L_2(0,T_*;\mathbf{H}_\#^{n/2-1})}^2\right].
\end{multline}
We have $\|\mathbf u_m(\cdot,0)\|_{\mathbf H^{n/2}_{\#}}\le \|\mathbf u^0\|_{\mathbf H^{n/2}_{\#}}$ and
by \eqref{E5.70a}, the sequence $\|\mathbf{u}_m\|_{L_2(0,T_*;\mathbf H_\#^{n/2+1/2})}$ is bounded as well. 
Then \eqref{E5.74} implies that the sequence $\|\mathbf u_m\|^2_{L_\infty(0,T_*;\dot{\mathbf H}_{\#\sigma}^{n/2})}$ is also bounded.
Integrating \eqref{E4.93BTd}, we conclude that 
\begin{multline}\label{E5.75}
\|\mathbf u_m\|^2_{L_2(0,T_*;\dot{\mathbf H}_{\#\sigma}^{n/2+1})}
\le 48 C^2_{\mathbb A}\Big(\bar C_{0,r,\tilde\sigma}^2T_*
\| |\mathbb A|_{H_\#^{\tilde\sigma+1},F}\|^2_{L_\infty(0,T_*)}
+C^2_{*rn} \|\mathbf{u}_m\|^2_{L_2(0,T_*;\mathbf H_\#^{n/2+1/2})}\Big)
\|\mathbf u_m\|^2_{L_\infty(0,T_*;\dot{\mathbf H}_{\#\sigma}^{n/2})}
\\
+4C_{\mathbb A}\|\mathbf u_m(\cdot,0)\|_{\mathbf H^{n/2}_{\#}}^2
+48 C^2_{\mathbb A}\|\mathbf f\|_{L_2(0,T_*;\mathbf{H}_\#^{n/2-1})}^2
\le C<\infty.
\end{multline}
Inequalities \eqref{E5.74} and \eqref{E5.75} mean that the sequences
\begin{align}\label{E5.75a}
\{\|\mathbf u_m\|_{L_\infty(0,T_*;\dot{\mathbf H}_{\#\sigma}^{r})}\}_{m=1}^\infty
\text { and } \{\|\mathbf u_m\|_{L_2(0,T_*;\dot{\mathbf H}_{\#\sigma}^{r+1})}\}_{m=1}^\infty
\text{ are bounded for } r=n/2.
\end{align}

\paragraph{Step(c).}
Let now $kn/2< r\le(k+1)n/2$, $k=1,2,3,\ldots$
By \eqref{r>n/2E4.910C} we obtain from \eqref{E4.82T0}, 
\begin{align}\label{E4.93BTe}
\frac{d}{d t}\left\|\mathbf u_m\right\|_{\mathbf H^{r}_{\#}}^2
+\frac14 C_{\mathbb A}^{-1}\|\mathbf u_m\|_{\mathbf{H}_\#^{r+1}}^2
\leq 12 C_{\mathbb A}\Big(\bar C_{0,r,\tilde\sigma}^2 |\mathbb A|^2_{H_\#^{\tilde\sigma+1},F}
+C^2_{*rn} \|\mathbf{u}_m\|^2_{{\mathbf H}_\#^{r}}\Big)\|\mathbf u_m\|^2_{\dot{\mathbf H}_{\#}^{r}}
+12 C_{\mathbb A}\|\mathbf f\|_{\mathbf{H}_\#^{r-1}}^2
\end{align}
implying
\begin{align}\label{E4.93BTe1}
\frac{d}{d t}\left\|\mathbf u_m\right\|_{\mathbf H^{r}_{\#}}^2
\leq 12 C_{\mathbb A}\Big(\bar C_{0,r,\tilde\sigma}^2 |\mathbb A|^2_{H_\#^{\tilde\sigma+1},F}
+C^2_{*rn} \|\mathbf{u}_m\|^2_{{\mathbf H}_\#^{r}}\Big)\|\mathbf u_m\|^2_{\dot{\mathbf H}_{\#}^{r}}
+12 C_{\mathbb A}\|\mathbf f\|_{\mathbf{H}_\#^{r-1}}^2.
\end{align}
By Gronwall's inequality \eqref{E17}, we obtain from \eqref{E4.93BTe1} that
\begin{multline}\label{E5.79}
\|\mathbf u_m\|^2_{L_\infty(0,T_*;\dot{\mathbf H}_{\#\sigma}^{r})}
\le \exp\Big[12 C_{\mathbb A}\Big(\bar C_{0,r,\tilde\sigma}^2T_*
\| |\mathbb A|_{H_\#^{\tilde\sigma+1},F}\|^2_{L_\infty(0,T_*)}
+C^2_{*rn} \|\mathbf{u}_m\|^2_{L_2(0,T_*;\mathbf H_\#^{r})}\Big)\Big]
\\
\times \left[\|\mathbf u_m(\cdot,0)\|_{\mathbf H^{r}_{\#}}^2
+12 C_{\mathbb A}\|\mathbf f\|_{L_2(0,T_*;\mathbf{H}_\#^{r-1})}^2\right],
\end{multline}
where $\|\mathbf u_m(\cdot,0)\|_{\mathbf H^{r}_{\#}}\le \|\mathbf u^0\|_{\mathbf H^{r}_{\#}}$.

If $k=1$, then
 the sequence $\|\mathbf{u}_m\|_{L_2(0,T_*;\mathbf H_\#^{r})}$ in \eqref{E5.79} is bounded due to \eqref{E5.70a} and \eqref{E5.75a}. 
Then \eqref{E5.79} implies that the sequence $\|\mathbf u_m\|^2_{L_\infty(0,T_*;\dot{\mathbf H}_{\#\sigma}^{r})}$ is bounded as well.
Integrating \eqref{E4.93BTe}, we also conclude that for $k=1$,
\begin{multline}\label{E5.80}
\|\mathbf u_m\|^2_{L_2(0,T_*;\dot{\mathbf H}_{\#\sigma}^{r+1})}
\le 48 C^2_{\mathbb A}\Big(\bar C_{0,r,\tilde\sigma}^2T_*
\| |\mathbb A|_{H_\#^{\tilde\sigma+1},F}\|^2_{L_\infty(0,T_*)}
+C^2_{*rn} \|\mathbf{u}_m\|^2_{L_2(0,T_*;\mathbf H_\#^{r})}\Big)
\|\mathbf u_m\|^2_{L_\infty(0,T_*;\dot{\mathbf H}_{\#\sigma}^{r})}
\\
+4C_{\mathbb A}\|\mathbf u_m(\cdot,0)\|_{\mathbf H^{r}_{\#}}^2
+48 C^2_{\mathbb A}\|\mathbf f\|_{L_2(0,T_*;\mathbf{H}_\#^{r-1})}^2,
\quad kn/2<r\le (k+1)n/2.
\end{multline}
Inequalities \eqref{E5.79} and \eqref{E5.80} mean that for $k=1$ the sequences
\begin{align}\label{E5.80a}
\{\|\mathbf u_m\|_{L_\infty(0,T_*;\dot{\mathbf H}_{\#\sigma}^{r})}\}_{m=1}^\infty
\text { and } \{\|\mathbf u_m\|_{L_2(0,T_*;\dot{\mathbf H}_{\#\sigma}^{r+1})}\}_{m=1}^\infty
\text{ are bounded for } kn/2<r\le (k+1)n/2.
\end{align}

If we assume that properties \eqref{E5.80a} hold for some integer $k\ge 1$, then by the similar argument,  properties \eqref{E5.80a} hold with $k$ replaced by $k+1$ and thus, by induction, they hold for any integer $k$.
Hence, collecting properties  \eqref{E5.70a}, \eqref{E5.75a}, and \eqref{E5.80a} we conclude that 
the sequences
\begin{align}\label{E5.80b}
\{\|\mathbf u_m\|_{L_\infty(0,T_*;\dot{\mathbf H}_{\#\sigma}^{r})}\}_{m=1}^\infty
\text { and } \{\|\mathbf u_m\|_{L_2(0,T_*;\dot{\mathbf H}_{\#\sigma}^{r+1})}\}_{m=1}^\infty
\text{ are bounded for } n/2-1<r.
\end{align}

Properties \eqref{E5.80b} imply that  there exists a subsequence of $\{\mathbf u_m\}$ converging weakly  in $L_2(0, T_* ; \dot{\mathbf H}_{\#\sigma}^{r+1})$ and weakly-star in $L_\infty(0, T_* ; \dot{\mathbf H}_{\#\sigma}^r)$ to a function 
${\mathbf u}^\dag\in L_2(0, T_* ; \dot{\mathbf H}_{\#\sigma}^{r+1})\cup L_\infty(0, T_* ; \dot{\mathbf H}_{\#\sigma}^r)$.
Then the subsequence converges to ${\mathbf u}^\dag$ also weakly in $L_2(0,T_*;\dot{\mathbf H}_{\#\sigma}^{1})$ and weakly-star in $L_\infty(0,T_*;\dot{\mathbf H}_{\#\sigma}^{0})$.
Since $\{\mathbf u_m\}$ is the subsequence of the sequence that converges weakly in $L_2(0,T;\dot{\mathbf H}_{\#\sigma}^{1})$ and weakly-star in $L_\infty(0,T;\dot{\mathbf H}_{\#\sigma}^{0})$ to the weak solution, $\mathbf u$, of problem \eqref{NS-problem-div0}--\eqref{NS-problem-div0-IC} on $\left[0, T_*\right]$, we conclude that 
$\mathbf u={\mathbf u}^\dag\in L_2(0, T_* ; \dot{\mathbf H}_{\#\sigma}^{r+1})\cup L_\infty(0, T_* ; \dot{\mathbf H}_{\#\sigma}^r)$, for any $r>n/2-1$ and we thus finished proving and we thus finished proving that
\begin{align}\label{E5.28V}
\mathbf u\in L_{\infty}(0,T_*;\dot{\mathbf H}_{\#\sigma}^{r})\cap L_2(0,T_*;\dot{\mathbf H}_{\#\sigma}^{r+1}).
\end{align}

\paragraph{Step(d).} 
Estimate \eqref{E4.91-0T0EE1r-uC} and inclusion \eqref{E5.28V} imply that 
$({\mathbf u}\cdot \nabla ){\mathbf u}\in{L_2(0,T_*;\mathbf{H}_\#^{r-1})}$.
By \eqref{L-oper} and \eqref{TensNorm} we have
\begin{align*}
\left\|\bs{\mathfrak L}\mathbf u\right\|^2_{{\mathbf H}_{\#}^{r-1}}
\le\|a_{ij}^{\alpha \beta }E_{i\alpha }({\mathbf u})\|_{({H}_{\#}^{r})^{n\times n}}
\le \|\mathbb A\|^2_{H_\#^{\tilde\sigma+1},F}\|\mathbf u\|^2_{{\mathbf H}_{\#}^{r+1}}\quad
\text{for a.e. } t\in(0,T),
\end{align*}
and thus
\begin{align*}
\|\bs{\mathfrak L}\mathbf u\|^2_{L_2(0,T_*;\dot{\mathbf H}_{\#}^{r-1})}
&\le \|\mathbb A\|^2_{L_\infty(0,T_*;H_\#^{\tilde\sigma+1}),F}\|\mathbf u\|^2_{L_2(0,T_*;{\mathbf H}_{\#}^{r+1})},
\end{align*}
i.e., $\bs{\mathfrak L}\mathbf u\in{L_2(0,T_*;\dot{\mathbf H}_{\#}^{r-1})}$.
We also have ${\mathbf f}\in L_2(0,T;\dot{\mathbf H}_\#^{r-1})$.
Thus $\mathbf F$ defined by \eqref{Eq-F} belongs to $L_2(0,T;\dot{\mathbf H}_{\#}^{r-1})$.
Then \eqref{NS-problem-just} implies that 
${\mathbf u}'\in{L_2(0,T_*;\mathbf{H}_{\#\sigma}^{r-1})}$ and since ${\mathbf u}\in{L_2(0,T_*;\mathbf{H}_{\#\sigma}^{r+1})}$, we obtain by Theorem \ref{LM-T3.1} that 
$
\mathbf u \in\mathcal C^0([0,T_*];\dot{\mathbf H}_{\#\sigma}^{r})
$,
which also means 
that $
\| \mathbf u(\cdot,t)-{\mathbf u}^0\| _{\dot{\mathbf H}_{\#\sigma}^{r}}\to 0
$
as ${t\to 0}$.

To prove the theorem claim about the associated pressure $p$, we remark that $p$ satisfies \eqref{Eq-p}.
By Lemma \ref{div-grad-is} for gradient, with $s=r$, equation \eqref{Eq-p} has a unique solution $p$ in 
$L_2(0,T_*;\dot{H}_{\#}^{r})$.
\end{proof}

As in Corollaries \ref{NS-problemCor-sigma-2} and \ref{NS-problemCor-sigma-1},  condition  \eqref{E4.143T*var} in Theorem \ref{NS-problemTh-sigma-Lv-er} can be replaced by simpler conditions for particular cases, which leads to the following two assertions.

\begin{corollary}[Serrin-type solution  for arbitrarily large data but small time or vice versa.]
\label{NS-problemCor-sigma-2>}
Let $n\ge 2$ and $T>0$. 
Let $a_{ij}^{\alpha \beta }\in L_\infty([0,T];H_\#^{\tilde\sigma+1})$, 
$\tilde\sigma>|r-1|+\frac{n}{2}$,
 and the relaxed ellipticity condition hold.
Let ${\mathbf f}\in L_2(0,T;\dot{\mathbf H}_\#^{r-1})\cap L_\infty(0,T;\dot{\mathbf H}_\#^{n/2-2})$  and 
$\mathbf u^0\in \dot{\mathbf H}_{\#\sigma}^{r}\cap \dot{\mathbf H}_{\#\sigma}^{n/2}$, $r>{n}/{2}-1$.
Let $T_*\in (0,T)$ satisfies inequality \eqref{E4.143T*} in Corollary \ref{NS-problemCor-sigma-2}.

Then the Serrin-type solution $\mathbf u$ of the anisotropic Navier-Stokes 
initial value problem  \eqref{NS-problem-div0}--\eqref{NS-problem-div0-IC} 
belongs to
$L_{\infty}(0,T_*;\dot{\mathbf H}_{\#\sigma}^{r})\cap L_2(0,T_*;\dot{\mathbf H}_{\#\sigma}^{r+1})$.
In addition, 
$
\mathbf u'\in L_2(0,T_*;\dot{\mathbf H}_{\#\sigma}^{r-1}),
$
$
\mathbf u \in\mathcal C^0([0,T_*];\dot{\mathbf H}_{\#\sigma}^{r})
$, 
$
\lim_{t\to 0}\| \mathbf u(\cdot,t)-{\mathbf u}^0\| _{\dot{\mathbf H}_{\#\sigma}^{r}}= 0
$
and $p\in L_2(0,T_*;\dot{H}_{\#}^{r})$.
\end{corollary}

\begin{corollary}[Serrin-type solution for arbitrary time but small data]
\label{NS-problemCor-sigma-1>}
Let $n\ge 2$,   $r\ge{n}/{2}-1$, and $T>0$. 
Let  $a_{ij}^{\alpha \beta }\in L_\infty([0,T];H_\#^{\tilde\sigma+1})$, 
$\tilde\sigma>|r-1|+\frac{n}{2}$,
 and the relaxed ellipticity condition hold.
Let the data ${\mathbf f}\in L_2(0,T;\dot{\mathbf H}_\#^{r-1})$  and $\mathbf u^0\in \dot{\mathbf H}_{\#\sigma}^{r}$ satisfy inequality \eqref{E4.143T} in Corollary \ref{NS-problemCor-sigma-1}.

Then the Serrin-type solution $\mathbf u$ of the anisotropic Navier-Stokes 
initial value problem  \eqref{NS-problem-div0}--\eqref{NS-problem-div0-IC} 
belongs to
$L_{\infty}(0,T;\dot{\mathbf H}_{\#\sigma}^{r})\cap L_2(0,T;\dot{\mathbf H}_{\#\sigma}^{r+1})$.
In addition, 
$
\mathbf u'\in L_2(0,T;\dot{\mathbf H}_{\#\sigma}^{r-1}),
$
$
\mathbf u \in\mathcal C^0([0,T];\dot{\mathbf H}_{\#\sigma}^{r})
$,
$
\lim_{t\to 0}\| \mathbf u(\cdot,t)-{\mathbf u}^0\| _{\dot{\mathbf H}_{\#\sigma}^{r}}= 0
$
and $p\in L_2(0,T;\dot{H}_{\#}^{r})$.
\end{corollary}

Theorem \ref{NS-problemTh-sigma-Lv-er} leads also to the following infinite regularity assertion.
\begin{corollary}
\label{Cor5.11sp}
Let $T>0$ and $n\ge 2$. 
Let  
$a_{ij}^{\alpha \beta }\in {\mathcal C}^{\infty}([0,T]; C^\infty_\#)$,
 and the relaxed ellipticity condition \eqref{mu} hold.
Let ${\mathbf f}\in L_2(0,T;\dot{\mathbf C}_\#^{\infty})$
and $\mathbf u^0\in \dot{\mathbf C}_{\#\sigma}^{\infty}$, while ${\mathbf f}$, $\mathbf u^0$ and $T_*\in(0,T]$ satisfy  inequality \eqref{E4.143T*var} from Theorem \ref{NS-problemTh-sigma-Lv-cr}.

Then the Serrin-type solution $\mathbf u$ of the anisotropic Navier-Stokes 
initial value problem  \eqref{NS-problem-div0}--\eqref{NS-problem-div0-IC} 
is such that
${\mathbf u}\in \mathcal C^0([0,T_*];\dot{\mathbf C}_{\#\sigma}^{\infty})$, 
$\mathbf u'\in L_2(0,T_*;\dot{\mathbf C}_{\#\sigma}^{\infty})$ and 
$p\in L_2(0,T_*;\dot{\mathcal C}_{\#}^{\infty})$.
\end{corollary}
\begin{proof}
Taking into account that $\dot{\mathbf C}_\#^{\infty}=\bigcap_{r\in\R} \dot{\mathbf{H}}_{\#\sigma}^r$, Theorem \ref{NS-problemTh-sigma-Lv-er} implies that
${\mathbf u}\in{L_\infty(0,T_*;\dot{\mathbf{H}}_{\#\sigma}^{r})}\cap L_2(0,T_*;\dot{\mathbf H}_\#^{r+1})$, 
$\mathbf u'\in L_2(0,T_*;\dot{\mathbf H}_{\#\sigma}^{r-1}),$
$p\in L_2(0,T;\dot{H}_{\#}^{r})$,
 $\forall\, r\in\R$.
Hence ${\mathbf u}\in \mathcal C^0([0,T_*];\dot{\mathbf C}_{\#\sigma}^{\infty})$, 
$\mathbf u'\in L_2(0,T_*;\dot{\mathbf C}_{\#\sigma}^{\infty})$
and $p\in L_2(0,T_*;\dot{\mathcal C}_{\#}^{\infty})$.
\end{proof}

\subsection{Spatial-temporal regularity of Serrin-type solutions for variable anisotropic viscosity coefficients}\label{S5}

\begin{theorem}
\label{NS-problemTh-sigma-Lv-er-t1}
Let $T>0$ and $n\ge 2$. Let $r\ge {n}/{2}-1$ if $n\ge 3$, while $r> {n}/{2}-1$  if $n=2$. 
Let $a_{ij}^{\alpha \beta }\in L_\infty([0,T];H_\#^{\tilde\sigma+1})$, 
$\tilde\sigma>|r-1|+\frac{n}{2}$,
 and the relaxed ellipticity condition \eqref{mu} hold.
Let ${\mathbf f}\in L_\infty(0,T;\dot{\mathbf H}_\#^{r-2})\cap L_2(0,T;\dot{\mathbf H}_\#^{r-1})$
and $\mathbf u^0\in \dot{\mathbf H}_{\#\sigma}^{r}$, while ${\mathbf f}$, $\mathbf u^0$ and $T_*\in(0,T]$ satisfy  inequality \eqref{E4.143T*var} from Theorem \ref{NS-problemTh-sigma-Lv-cr}.

Then the Serrin-type solution $\mathbf u$ of the anisotropic Navier-Stokes 
initial value problem  \eqref{NS-problem-div0}--\eqref{NS-problem-div0-IC} 
is such that
$\mathbf u'\in L_{\infty}(0,T_*;\dot{\mathbf H}_{\#\sigma}^{r-2})\cup L_2(0,T_*;\dot{\mathbf H}_{\#\sigma}^{r-1})$, while
$p\in L_\infty(0,T_*;\dot{H}_{\#}^{r-1})\cap L_2(0,T_*;\dot{H}_{\#}^{r})$.
\end{theorem}
\begin{proof}
By Theorems \ref{NS-problemTh-sigma-Lv-cr} and \ref{NS-problemTh-sigma-Lv-er},  we have the inclusions 
$\mathbf u\in L_{\infty}(0,T_*;\dot{\mathbf H}_{\#\sigma}^{r})$,
$\mathbf u'\in L_2(0,T_*;\dot{\mathbf H}_{\#\sigma}^{r-1})$ and $p\in L_2(0,T_*;\dot{H}_{\#}^{r})$. 
Then we only need to prove the inclusions $\mathbf u'\in L_{\infty}(0,T_*;\dot{\mathbf H}_{\#\sigma}^{r-2})$ and $p\in L_\infty(0,T_*;\dot{H}_{\#}^{r-1})$.

As in the proof of Theorem \ref{NS-problemTh-sigma-Lv-er-t1} we arrive at estimate \eqref{E4.91-0T0c*C}.

By \eqref{L-oper}, \eqref{TensNorm} and the multiplication Theorem \ref{RS-T1-S4.6.1}(a), we have
\begin{align*}
\left\|\bs{\mathfrak L}\mathbf u\right\|_{{\mathbf H}_{\#}^{r-2}}
\le\|a_{ij}^{\alpha \beta }E_{i\alpha }({\mathbf u})\|_{({H}_{\#}^{r-1})^{n\times n}}
\le C_*(r-1,\tilde\sigma +1,n)\|\mathbb A\|_{H_\#^{\tilde\sigma+1},F}\|\mathbf u\|_{{\mathbf H}_{\#}^{r}},
\end{align*}
where we took into account that $\tilde\sigma+1>n/2$ and $\tilde\sigma+1>r$.
Thus
\begin{align*}
\|\bs{\mathfrak L}\mathbf u\|_{L_\infty(0,T_*;\dot{\mathbf H}_{\#}^{r-2})}
&\le C_*(r-1,\tilde\sigma +1,n)\|\mathbb A\|_{L_\infty(0,T_*;H_\#^{\tilde\sigma+1}),F}\|\mathbf u\|_{L_\infty(0,T_*;{\mathbf H}_{\#}^{r})},
\end{align*}
i.e., $\bs{\mathfrak L}\mathbf u\in{L_\infty(0,T_*;\dot{\mathbf H}_{\#}^{r-2})}$.
We also have ${\mathbf f}\in L_\infty(0,T;\dot{\mathbf H}_\#^{r-2})$.

Then \eqref{NS-problem-just} implies that 
${\mathbf u}'\in{L_\infty(0,T_*;\mathbf{H}_{\#\sigma}^{r-2})}$, while \eqref{Eq-p} and Lemma \ref{div-grad-is} for gradient, with $s=r-1$, imply that  $p\in L_\infty(0,T_*;\dot{H}_{\#}^{r-1})$. 
\end{proof}

To simplify the following two assertions we assume there that the viscosity coefficients  are infinitely smooth it time and space coordinates. This smoothness condition can be relaxed if we instead assume that all the norms of these coefficients encountered in the proof are bounded. 
\begin{theorem}
\label{NS-problemTh-sigma-Lv-er-t}
Let $T>0$ and $n\ge 2$. 
Let $r> {n}/{2}-1$. 
Let 
$a_{ij}^{\alpha \beta }\in {\mathcal C}^{\infty}([0,T]; C^\infty_\#)$,
 and the relaxed ellipticity condition \eqref{mu} hold.
Let $k\in [1, r+1)$ be an integer. 
Let 
${\mathbf f}^{(l)}\in L_\infty(0,T;\dot{\mathbf H}_\#^{r-2-2l})\cap L_2(0,T;\dot{\mathbf H}_\#^{r-1-2l})$, $l=0,1,\ldots,k-1$,
and $\mathbf u^0\in \dot{\mathbf H}_{\#\sigma}^{r}$, while ${\mathbf f}$, $\mathbf u^0$ and $T_*\in(0,T]$ satisfy  inequality \eqref{E4.143T*var} from Theorem \ref{NS-problemTh-sigma-Lv-cr}.

Then the Serrin-type solution $\mathbf u$ of the anisotropic Navier-Stokes 
initial value problem  \eqref{NS-problem-div0}--\eqref{NS-problem-div0-IC} 
is such that
$\mathbf u^{(l)}\in L_{\infty}(0,T_*;\dot{\mathbf H}_{\#\sigma}^{r-2l})\cap L_2(0,T_*;\dot{\mathbf H}_{\#\sigma}^{r+1-2l})$, $l=0,\ldots,k$, while
$p^{(l)}\in L_\infty(0,T_*;\dot{H}_{\#}^{r-1-2l})\cap L_2(0,T_*;\dot{H}_{\#\sigma}^{r-2l})$, $l=0,\ldots,k-1$.
\end{theorem}
\begin{proof}
The proof coincide with the proof of  the corresponding constant-coefficient Theorem \eqref{NS-problemTh-sigma-Lv-er-tC}, except the  proof of inclusion \eqref{E5.65C} in Step 2, that for the variable coefficients is replaced by the following argument.
\begin{align}
\label{NS-problem-just-k}
\partial_t^{k-1}\bs{\mathfrak L}{\mathbf u}
=\sum_{l=0}^{k-1}C_{k-1}^l\nabla\cdot[{\mathbb A}^{(k-1-l)} \mathbb E ({\mathbf u}^{(l)})].
\end{align}
By \eqref{L-oper}, \eqref{TensNorm} and Theorem \ref{RS-T1-S4.6.1}(a), we have
\begin{multline*}
\|\nabla\cdot[\mathbb A^{(k-1-l)} \mathbb E ({\mathbf u}^{(l)})]\|_{\dot{\mathbf H}_{\#\sigma}^{r-2k}}
\le\|\mathbb A^{(k-1-l)} \mathbb E ({\mathbf u}^{(l)})\|_{({H}_{\#}^{r+1-2k})^{n\times n}}
\le\|\mathbb A^{(k-1-l)} \mathbb E ({\mathbf u}^{(l)})\|_{({H}_{\#}^{r-1-2l})^{n\times n}}
\\
\le C_{*l1}\|\mathbb A^{(k-1-l)}\|_{H_\#^{\tilde\sigma_{l1}},F}\|\mathbb E ({\mathbf u}^{(l)})\|_{({H}_{\#}^{r-1-2l})^{n\times n}}
\le C_{*l1}\|\mathbb A^{(k-1-l)}\|_{H_\#^{\tilde\sigma_{l1}},F}\|\mathbf u^{(l)}\|_{{\mathbf H}_{\#}^{r-2l}},
\end{multline*}
where $\tilde\sigma_{l1}>\max\{n/2,2l+1-r\}$,  $C_{*l1}=C_*(r-1-2l,\tilde\sigma_{l1},n)$.
Thus
\begin{align*}
\|\nabla\cdot[\mathbb A^{(k-1-l)} \mathbb E ({\mathbf u}^{(l)})]\|_{L_\infty(0,T_*;\dot{\mathbf H}_{\#\sigma}^{r-2k})}
&\le C_{*l1}\|\mathbb A^{(k-1-l)}\|_{L_\infty(0,T_*;H_\#^{\tilde\sigma_{l1}}),F}\|\mathbf u^{(l)}\|_{L_\infty(0,T_*;{\mathbf H}_{\#}^{r-2l})},
\end{align*}
i.e., due to \eqref{E5.75tC},
$
\nabla\cdot[\mathbb A^{(k-1-l)} \mathbb E ({\mathbf u}^{(l)})]\in{L_\infty(0,T_*;\dot{\mathbf H}_{\#\sigma}^{r-2k})},
\quad l=0,\ldots,k-1.
$

On the other hand, by \eqref{L-oper}, \eqref{TensNorm} and Theorem \ref{RS-T1-S4.6.1}(a), we have
\begin{multline*}
\|\nabla\cdot[\mathbb A^{(k-1-l)} \mathbb E ({\mathbf u}^{(l)})]\|_{\dot{\mathbf H}_{\#\sigma}^{r+1-2k}}
\le\|\mathbb A^{(k-1-l)} \mathbb E ({\mathbf u}^{(l)})\|_{({H}_{\#}^{r+2-2k})^{n\times n}}
\le\|\mathbb A^{(k-1-l)} \mathbb E ({\mathbf u}^{(l)})\|_{({H}_{\#}^{r-2l})^{n\times n}}
\\
\le C_{*l2}\|\mathbb A^{(k-1-l)}\|_{H_\#^{\tilde\sigma_{l2}},F}\|\mathbb E ({\mathbf u}^{(l)})\|_{({H}_{\#}^{r-2l})^{n\times n}}
\le C_{*l2}\|\mathbb A^{(k-1-l)}\|_{H_\#^{\tilde\sigma_{l2}},F}\|\mathbf u^{(l)}\|_{{\mathbf H}_{\#}^{r+1-2l}},
\end{multline*}
where $\tilde\sigma_{l2}>\max\{n/2,2l-r\}$,  $C_{*l2}=C_*(r-2l,\tilde\sigma_{l2},n)$.
Thus
\begin{align*}
\|\nabla\cdot[\mathbb A^{(k-1-l)} \mathbb E ({\mathbf u}^{(l)})]\|_{L_\infty(0,T_*;\dot{\mathbf H}_{\#\sigma}^{r+1-2k})}
&\le C_{*l2}\|\mathbb A^{(k-1-l)}\|_{L_\infty(0,T_*;H_\#^{\tilde\sigma_{l2}}),F}
\|\mathbf u^{(l)}\|_{L_2(0,T_*;{\mathbf H}_{\#}^{r+1-2l})},
\end{align*}
i.e., due to \eqref{E5.75tC},
$\nabla\cdot[\mathbb A^{(k-1-l)} \mathbb E ({\mathbf u}^{(l)})]\in{L_2(0,T_*;\dot{\mathbf H}_{\#\sigma}^{r+1-2k})},
\quad l=0,\ldots,k-1.$
Hence by  \eqref{NS-problem-just-k},
\begin{align}\label{E5.65}
\partial_t^{k-1}\bs{\mathfrak L}{\mathbf u}\in L_\infty(0,T_*;\mathbf H_\#^{r-2k})\cap L_2(0,T_*;\mathbf H_\#^{r+1-2k}).
\end{align}
\end{proof}

\begin{corollary}
\label{Cor5.11}
Let $T>0$ and $n\ge 2$. 
Let  
$a_{ij}^{\alpha \beta }\in {\mathcal C}^{\infty}([0,T]; C^\infty_\#)$,
 and the relaxed ellipticity condition \eqref{mu} hold.
Let ${\mathbf f}\in \mathcal C^\infty(0,T;\dot{\mathbf C}_\#^{\infty})$,  
and $\mathbf u^0\in \dot{\mathbf C}_{\#\sigma}^{\infty}$, while ${\mathbf f}$, $\mathbf u^0$ and $T_*\in(0,T]$ satisfy  inequality \eqref{E4.143T*var} from Theorem \ref{NS-problemTh-sigma-Lv-cr}.

Then the Serrin-type solution $\mathbf u$ of the anisotropic Navier-Stokes 
initial value problem  \eqref{NS-problem-div0}--\eqref{NS-problem-div0-IC} 
is such that
$\mathbf u\in \mathcal C^{\infty}(0,T_*;\dot{\mathbf C}_{\#\sigma}^{\infty})$,
$p\in \mathcal C^\infty(0,T_*;\dot{\mathcal C}_{\#}^\infty)$.
\end{corollary}
\begin{proof}
Taking into account that $\dot{\mathbf C}_\#^{\infty}=\bigcap_{r\in\R} \dot{\mathbf{H}}_{\#\sigma}^r$, Theorem \ref{NS-problemTh-sigma-Lv-er-t} implies that for any integer $k\ge 0$,
${\mathbf u}^{(k)}\in{L_\infty(0,T_*;\dot{\mathbf{H}}_{\#\sigma}^{r-2k})}\cap L_2(0,T_*;\dot{\mathbf H}_\#^{r+1-2k})$,
$p^{(k)}\in L_\infty(0,T_*;\dot{H}_{\#}^{r-1-2k})\cap L_2(0,T_*;\dot{H}_{\#\sigma}^{r-2k})$, for any $r\in\R$.
Hence $\mathbf u\in \mathcal C^{\infty}(0,T_*;\dot{\mathbf C}_{\#\sigma}^{\infty})$,
$p\in \mathcal C^\infty(0,T_*;\dot{\mathcal C}_{\#}^\infty)$.
\end{proof}

\subsection{Regularity of two-dimensional weak solution for variable anisotropic viscosity coefficients}\label{S2D}
Here we provide a counterpart of Section \ref{S2DC} generalised to variable viscosity coefficients.
\begin{theorem}[Spatial regularity of solution for arbitrarily large data.]
\label{NS-problemTh-sigma-Lv-er2}
Let $n= 2$, $r>0$, and $T>0$. 
Let  
$a_{ij}^{\alpha \beta }\in L_\infty([0,T];H_\#^{\tilde\sigma+1})$, 
$\tilde\sigma>1+\max\{|r-1|,1\}$,
 and the relaxed ellipticity condition \eqref{mu} hold.
Let ${\mathbf f}\in L_2(0,T;\dot{\mathbf H}_\#^{r-1})$  and $\mathbf u^0\in \dot{\mathbf H}_{\#\sigma}^{r}$.

Then the solution $\mathbf u$ of the anisotropic Navier-Stokes 
initial value problem  \eqref{NS-problem-div0}--\eqref{NS-problem-div0-IC}
obtained in Theorem \ref{NS-problemTh-sigma}  is of Serrin-type and 
belongs to
$L_{\infty}(0,T;\dot{\mathbf H}_{\#\sigma}^{r})\cap L_2(0,T;\dot{\mathbf H}_{\#\sigma}^{r+1})$.
In addition, 
$
\mathbf u'\in L_2(0,T;\dot{\mathbf H}_{\#\sigma}^{r-1}),
$
$
\mathbf u \in\mathcal C^0([0,T];\dot{\mathbf H}_{\#\sigma}^{r})
$
$
\lim_{t\to 0}\| \mathbf u(\cdot,t)-{\mathbf u}^0\| _{\dot{\mathbf H}_{\#\sigma}^{r}}= 0
$
and $p\in L_2(0,T_*;\dot{H}_{\#}^{r})$. 
\end{theorem}
\begin{proof}
The proof coincides word-for-word with the proof of Theorem \ref{NS-problemTh-sigma-Lv-er} if we take there $n=2$ while replacing $T_*$ by $T$ and the reference to \eqref{E4.14varb} for the boundedness of the sequence $\|\mathbf{u}_m\|_{L_2(0,T_*;\mathbf H_\#^{n/2})}$ for $n=2$ by the reference to the corresponding inequality
\begin{align*}
&\|\mathbf u_m\|^2_{L_2(0,T;\dot{\mathbf H}_\#^{1})}
\le 4C_{\mathbb A}\left( \| {\mathbf u}^0\| ^2_{\mathbf L_{2\#}}
+4 C_{\mathbb A}\|\mathbf f\|^2_{L_2(0,T;\dot{\mathbf H}_\#^{-1})}\right).
\end{align*}
 obtained as inequality (59) in our paper \cite{Mikhailov2024}. 
\end{proof}

The following assertion is proved similar to Corollary \ref{Cor5.11sp}. 
\begin{corollary}
Let $T>0$ and $n= 2$. 
Let  
$a_{ij}^{\alpha \beta }\in {\mathcal C}^{\infty}([0,T]; C^\infty_\#)$,
 and the relaxed ellipticity condition \eqref{mu} hold.
Let ${\mathbf f}\in L_2(0,T;\dot{\mathbf C}_\#^{\infty})$  
and $\mathbf u^0\in \dot{\mathbf C}_{\#\sigma}^{\infty}$.

Then the solution $\mathbf u$ of the anisotropic Navier-Stokes 
initial value problem  \eqref{NS-problem-div0}--\eqref{NS-problem-div0-IC}
obtained in Theorem \ref{NS-problemTh-sigma}  is of Serrin-type and 
is such that
${\mathbf u}\in \mathcal C^0([0,T];\dot{\mathbf C}_{\#\sigma}^{\infty})$, 
$\mathbf u'\in L_2(0,T;\dot{\mathbf C}_{\#\sigma}^{\infty})$ and 
$p\in L_2(0,T;\dot{\mathcal C}_{\#}^{\infty})$.
\end{corollary}

The next three assertions on spatial-temporal regularity for $n=2$ are the corresponding counterparts of Theorems \ref{NS-problemTh-sigma-Lv-er-t1}, \ref{NS-problemTh-sigma-Lv-er-t} and Corollary \ref{Cor5.11} and are proved in a similar way after replacing there $T_*$ by $T$.
\begin{theorem}
\label{NS-problemTh-sigma-Lv-er-t1-2}
Let $T>0$, $n\ge 2$ and  $r> 0$. 
Let $a_{ij}^{\alpha \beta }\in L_\infty([0,T];H_\#^{\tilde\sigma+1})$, 
$\tilde\sigma>|r-1|+1$,
 and the relaxed ellipticity condition \eqref{mu} hold.
Let ${\mathbf f}\in L_\infty(0,T;\dot{\mathbf H}_\#^{r-2})\cap L_2(0,T;\dot{\mathbf H}_\#^{r-1})$
and $\mathbf u^0\in \dot{\mathbf H}_{\#\sigma}^{r}$.

Then the solution $\mathbf u$ of the anisotropic Navier-Stokes 
initial value problem  \eqref{NS-problem-div0}--\eqref{NS-problem-div0-IC}
obtained in Theorem \ref{NS-problemTh-sigma}  is of Serrin-type and 
is such that
$\mathbf u'\in L_{\infty}(0,T;\dot{\mathbf H}_{\#\sigma}^{r-2})\cup L_2(0,T;\dot{\mathbf H}_{\#\sigma}^{r-1})$,
$p\in L_\infty(0,T;\dot{H}_{\#}^{r-1})\cap L_2(0,T;\dot{H}_{\#}^{r})$.
\end{theorem}

\begin{theorem}
\label{NS-problemTh-sigma-Lv-er-t-2}
Let $T>0$, $n= 2$ and  $r> 0$. 
Let 
$a_{ij}^{\alpha \beta }\in {\mathcal C}^{\infty}([0,T]; C^\infty_\#)$,
 and the relaxed ellipticity condition \eqref{mu} hold.
Let $k\in [1, r+1)$ be an integer. 
Let  
${\mathbf f}^{(l)}\in L_\infty(0,T;\dot{\mathbf H}_\#^{r-2-2l})\cap L_2(0,T;\dot{\mathbf H}_\#^{r-1-2l})$, $l=0,1,\ldots,k-1$,
and $\mathbf u^0\in \dot{\mathbf H}_{\#\sigma}^{r}$.

Then the solution $\mathbf u$ of the anisotropic Navier-Stokes 
initial value problem  \eqref{NS-problem-div0}--\eqref{NS-problem-div0-IC}
obtained in Theorem \ref{NS-problemTh-sigma}  is of Serrin-type and 
is such that
$\mathbf u^{(l)}\in L_{\infty}(0,T;\dot{\mathbf H}_{\#\sigma}^{r-2l})\cap L_2(0,T;\dot{\mathbf H}_{\#\sigma}^{r+1-2l})$, $l=0,\ldots,k$;
$p^{(l)}\in L_\infty(0,T;\dot{H}_{\#}^{r-1-2l})\cap L_2(0,T;\dot{H}_{\#\sigma}^{r-2l})$, $l=0,\ldots,k-1$.
\end{theorem}

\begin{corollary}
Let $T>0$ and $n= 2$. 
Let  
$a_{ij}^{\alpha \beta }\in {\mathcal C}^{\infty}([0,T]; C^\infty_\#)$,
 and the relaxed ellipticity condition \eqref{mu} hold.
Let ${\mathbf f}\in \mathcal C^\infty(0,T;\dot{\mathbf C}_\#^{\infty})$,  
and $\mathbf u^0\in \dot{\mathbf C}_{\#\sigma}^{\infty}$.

Then the solution $\mathbf u$ of the anisotropic Navier-Stokes 
initial value problem  \eqref{NS-problem-div0}--\eqref{NS-problem-div0-IC}
obtained in Theorem \ref{NS-problemTh-sigma}  is of Serrin-type and 
is such that
$\mathbf u\in \mathcal C^{\infty}(0,T;\dot{\mathbf C}_{\#\sigma}^{\infty})$,
$p\in \mathcal C^\infty(0,T;\dot{\mathcal C}_{\#}^\infty)$.
\end{corollary}

\section{Auxiliary results}\label{Appendix}

\subsection{Advection term properties}\label{S7.2} 
The  divergence theorem and periodicity
imply the following identity for any ${\mathbf v}_1,{\mathbf v}_2,{\mathbf v}_3\in \mathbf C^\infty_\#$.
\begin{align}
\label{eq:mik54}
\left\langle({\mathbf v}_1\cdot \nabla ){\mathbf v}_2,{\mathbf v}_3\right\rangle _{\T}
&=\int_{\T}\nabla\cdot\left({\mathbf v}_1({\mathbf v}_2\cdot {\mathbf v}_3)\right)d{\mathbf x}
-\left\langle(\nabla \cdot{\mathbf v}_1){\mathbf v}_3
+({\mathbf v}_1\cdot \nabla ){\mathbf v}_3,{\mathbf v}_2\right\rangle _{\T}\nonumber
\\
&=
-\left\langle({\mathbf v}_1\cdot \nabla ){\mathbf v}_3,{\mathbf v}_2\right\rangle _{\T}
-\left\langle(\nabla \cdot{\mathbf v}_1){\mathbf v}_3,{\mathbf v}_2\right\rangle _{\T}
 \end{align}
Hence for any ${\mathbf v}_1,{\mathbf v}_2\in \mathbf C^\infty_\#$,
\begin{equation}
\label{eq:mik55a}
\left\langle ({\mathbf v}_1\cdot \nabla ){\mathbf v}_2,{\mathbf v}_2\right\rangle _{\T}
=-\frac12\left\langle(\nabla \cdot{\mathbf v}_1){\mathbf v}_2,{\mathbf v}_2\right\rangle _{\T}
=-\frac12\left\langle \div\,{\mathbf v}_1,|{\mathbf v}_2|^2\right\rangle _{\T}.
\end{equation}
In view of \eqref{eq:mik54} we obtain the identity
\begin{align}\label{E-B.20}
\left\langle({\mathbf v}_1\cdot \nabla ){\mathbf v}_2,{\mathbf v}_3\right\rangle _{\T}
&\!\!=\!-\left\langle({\mathbf v}_1\cdot \nabla ){\mathbf v}_3,{\mathbf v}_2\right\rangle _{\T}\quad
 \forall\ {\mathbf v}_1\in \mathbf C^\infty_{\#\sigma},\
{\mathbf v}_2,\, {\mathbf v}_3\in \mathbf C^\infty_\#\,,
\end{align}
and hence the following well known formula for any  ${\mathbf v}_1\in \mathbf C^\infty_{\#\sigma}$,
${\mathbf v}_2\in \mathbf C^\infty_\#$,
\begin{equation}
\label{eq:mik55}
\left\langle ({\mathbf v}_1\cdot \nabla ){\mathbf v}_2,{\mathbf v}_2\right\rangle _{\T}=0.
\end{equation}
Equations \eqref{E-B.20} and \eqref{eq:mik55} evidently hold also  for ${\mathbf v}_1$, ${\mathbf v}_2$, and ${\mathbf v}_3$ from  the more general spaces, for which the dual products in \eqref{E-B.20} and \eqref{eq:mik55}  are bounded and in which  $\mathbf C^\infty_{\#\sigma}$ and $\mathbf C^\infty_\#$, respectively, are densely embedded.

\subsection{Some point-wise multiplication results}\label{S-RS}  
Let us  accommodated to the periodic function spaces in $\R^n$, $n\ge 1$, a particular case of a much more general Theorem 1 in Section 4.6.1 of \cite{Runst-Sickel1996} about point-wise products of functions/distributions.
\begin{theorem}\label{RS-T1-S4.6.1}
Assume $n\ge 1$, $ s_1\le s_2$ and $s_1+s_2>0$.
Then there exists a constant $C_*(s_1,s_2,n)>0$ such that for any $f_1\in H^{s_1}_\#$ and $f_2\in H^{s_1}_\#$, 

(a)  \quad 
$
f_1\cdot f_2\in H^{s_1}_\# \text{ and } \|f_1\cdot f_2\|_{H^{s_1}_\#}\le  C_*(s_1,s_2,n) \|f_1\|_{H^{s_1}_\#}\, \|f_2\|_{H^{s_2}_\#}
\quad\text{if }s_2>n/2;
$

(b) \quad
$
f_1\cdot f_2\in H^{s_1+s_2-n/2}_\# \text{ and } \|f_1\cdot f_2\|_{H^{s_1+s_2-n/2}_\#}\le  C_* (s_1,s_2,n)\|f_1\|_{H^{s_1}_\#}\, \|f_2\|_{H^{s_2}_\#}
 \quad \text{if }s_2<n/2.
$
\end{theorem}
\begin{proof}
Items (a) and (b) follow, respectively, from items (i) and (iii) of \cite[Theorem 1 in Section 4.6.1]{Runst-Sickel1996} when we take into account the norm equivalence in the standard and periodic Sobolev spaces.
\end{proof}

Let
$$
(u  \star v)(\bs\xi):=\sum_{\bs\eta\in\Z^n}u (\bs\eta)v(\bs\xi-\bs\eta),\quad \bs\xi\in\Z^n,
$$
be the convolution in $\Z^n$.
We will need the following Young's type inequality for discrete convolution of sequences in $\Z^n$.  For other choices of parameters see, e.g., \cite[p. 316]{Bullen2015} and references therein. 
\begin{lemma}\label{Young-discr1-q}
Let $n\in \mathbb N$, $1\le q\le \infty$
Then the convolution the sequences $u\in \ell_1(\Z^n)$ and $v\in \ell_q(\Z^n)$ belongs to $\ell_q(\Z^n)$ and
\begin{align}\label{Young-discr-in}
\|u\star v\|_{\ell_q(\Z^n)}\le \|u\|_{\ell_1(\Z^n)}\|v\|_{\ell_q(\Z^n)}.
\end{align}
\end{lemma}
\begin{proof}
By the triangle inequality, we obtain
\begin{align*}\label{E-C.5}
\|u\star v\|_{\ell_q(\Z^n)}
=\left\|\sum_{\bs\eta\in\Z^n}u (\bs\eta)v(\cdot-\bs\eta)\right\|_{\ell_q}
\le\sum_{\bs\eta\in\Z^n}\|u (\bs\eta)v(\cdot-\bs\eta)\|_{\ell_q}
&=\sum_{\bs\eta\in\Z^n}|u (\bs\eta)| \|v\|_{\ell_q}
=\|u\|_{\ell_1(\Z^n)}\|v\|_{\ell_q(\Z^n)}.
\end{align*}
\end{proof}
\begin{theorem}\label{Jcomm}
Assume $n\ge1$.
Let $s,\theta\in\R$,   $w\in H_\#^s$, $g\in H_\#^{\tilde\sigma+1}$, 
$\tilde\sigma>\max\{|s|,|s -\theta+1|\} + \frac{n}{2}
=|s-\frac{\theta-1}{2}|+|\frac{\theta-1}{2}|+\frac{n}{2}$. 
Then $\Lambda_\#^\theta(gw)-g\Lambda_\#^\theta w\in { H_\#^{s-\theta+1}}$ and
$$
 \|\Lambda_\#^\theta(gw)-g\Lambda_\#^\theta w\|_{ H_\#^{s -\theta+1}}\le
 \bar C_{s,\theta,\tilde\sigma} |g|_{H_\#^{\tilde\sigma+1}}\|w\|_{H_\#^s},
$$
where the constant $\bar C_{s,\theta,\tilde\sigma}$ does not depend on $g$ or $w$ but may depend on $s$, $\theta$ and $\tilde\sigma$.
\end{theorem}
\begin{proof}

By \eqref{E3.14} and \eqref{rho} we have,
\begin{align*}
K(\bs\xi):=\F_\T[\Lambda_\#^\theta(gw)-g\Lambda_\#^\theta w](\bs\xi)
&=(2\pi)^\theta(1+|\bs\xi|^2)^{\theta/2}(\widehat g  \star \widehat w)(\bs\xi) - (\widehat g \star\F_\T[\Lambda_\#^\theta w])(\bs\xi)\\
&=(2\pi)^\theta\sum_{\bs\eta \in\Z^n}[(1+|\bs\xi|^2)^{\theta/2}-(1+|\bs\xi-\bs\eta|^2)^{\theta/2}]\widehat g (\bs\eta)\widehat w(\bs\xi-\bs\eta)
\\
&=(2\pi)^2\sum_{\bs\eta \in\Z^n}[(\bs\eta\cdot\bs\xi+\bs\eta\cdot(\bs\xi-\bs\eta)]f_\theta(\bs\xi,\bs\xi-\bs\eta)\widehat g (\bs\eta)  \widehat w(\bs\xi-\bs\eta)\\
&=\frac{2\pi}{i}\sum_{\bs\eta \in\Z^n}\widehat{\nabla g}(\bs\eta)\cdot(\bs\xi + \bs\xi-\bs\eta)
f_\theta(\bs\xi,\bs\xi-\bs\eta)\widehat w(\bs\xi-\bs\eta).
\end{align*}
Here 
 $$
 f_\theta(\bs\xi,\bs\xi-\bs\eta):=(2\pi)^{\theta-2}\frac{(1+|\bs\xi|^2)^{\theta/2}-(1+|\bs\xi-\bs\eta|^2)^{\theta/2}}{|\bs\xi|^2-|\bs\xi-\bs\eta|^2} =\frac{\varrho^{\theta}(\bs\xi)-\varrho^{\theta}(\bs\xi-\bs\eta)}{\varrho^2(\bs\xi)-\varrho^2(\bs\xi-\bs\eta)},
 $$
and we took into account that
$|\bs\xi|^2-|\bs\xi-\bs\eta|^2=\bs\eta\cdot\bs\xi+\bs\eta\cdot(\bs\xi-\bs\eta)$.
%
Since the inequality
$|c_1^\beta-c_2^\beta|\le$
\mbox{$|\beta||c_1-c_2|(c_1^{\beta-1}+c_2^{\beta-1})$} holds for any $c_1,c_2>0$,
$\beta\in\R$,   we have,
$$
 |f_\theta(\bs\xi,\bs\xi-\bs\eta)|
\le\frac{|\varrho^{\theta-1}(\bs\xi)+\varrho^{\theta-1}(\bs\xi-\bs\eta)|}{|\varrho(\bs\xi)+\varrho(\bs\xi-\bs\eta)|}.
 $$
Hence
\begin{align*}
2\pi |(\bs\xi + \bs\xi-\bs\eta)f_\theta(\bs\xi,\bs\xi-\bs\eta)|
&\le2\pi (|\bs\xi| + |\bs\xi-\bs\eta|)|f_\theta(\bs\xi,\bs\xi-\bs\eta)|
\\
&\le
 2\pi|\theta|\left[\varrho^{\theta-1}(\bs\xi)+\varrho^{\theta-1}(\bs\xi-\bs\eta)\right]\frac{|\bs\xi| + |\bs\xi-\bs\eta|}
  {\varrho(\bs\xi)+\varrho(\bs\xi-\bs\eta)}
\le |\theta|\left[\varrho^{\theta-1}(\bs\xi)+\varrho^{\theta-1}(\bs\xi-\bs\eta)\right],
\end{align*}
for any $\theta\in \R$.
Then
\begin{eqnarray*}
|K(\bs\xi)| &\le& {|\theta|}\sum_{\bs\eta \in\Z^n}\left[\varrho^{\theta-1}(\bs\xi)+\varrho^{\theta-1}(\bs\xi-\bs\eta)\right] |\widehat{\nabla g}(\bs\eta)\widehat w(\bs\xi-\bs\eta)|\\
&=& {|\theta|}\sum_{\bs\eta \in\Z^n}\left[\varrho^{\theta-1}(\bs\xi)|\widehat{\nabla
g}(\bs\eta) \widehat w(\bs\xi-\bs\eta)|
+ |\widehat{\nabla g}(\bs\eta) \varrho^{\theta-1}(\bs\xi-\bs\eta)\widehat w(\bs\xi-\bs\eta)|\right]\\
&=& {|\theta|}\left[\varrho^{\theta-1}(\bs\xi)\{|\widehat{\nabla g}|\star|\widehat w|\}(\bs\xi)
+ \{|\widehat{\nabla g}|\star|\varrho^{\theta-1}\widehat w|\}(\bs\xi)\right].
\end{eqnarray*}
Taking into account Petree's inequality
$$
\varrho^s(\bs\xi)\le \frac{2^{|s|/2}}{(2\pi)^{|s|}}\varrho^{|s|}(\bs\eta)\varrho^{s}(\bs\xi-\bs\eta)\ \forall\ \bs\xi,\bs\eta\in\Z^n,\ \forall\ s\in\R,
$$
and the discrete Young's type inequality \eqref{Young-discr-in} for convolutions with $q=2$,
we obtain
\begin{align*}
\|\Lambda_\#^\theta(gw)-g\Lambda_\#^\theta w\|_{ H_\#^{s -\theta +1}}&=\|\varrho^{s -\theta +1} K\|_{\bs\ell_2}
\le {|\theta |}\left\|\varrho^{s }\{|\widehat{\nabla g}|\star |\widehat w|\}
+ \varrho^{s -\theta +1}\{|\widehat{\nabla g}|\star |\varrho^{\theta -1}\widehat w|\}\right\|_{\bs\ell_2}
\\
 &\le \frac{2^{|s|/2}|\theta |}{(2\pi)^{|s|}}\left\||\varrho^{|s|}\widehat{\nabla g}|\star |\varrho^{s }\widehat w|
+ |\varrho^{|s -\theta +1|}\widehat{\nabla g}|\star |\varrho^{s}\widehat w|\right\|_{\bs\ell_2}\\
&\le \frac{2^{|s|/2}|\theta |}{2\pi}\left(\|\varrho^{|s|}\widehat{\nabla g}\|_{\bs\ell_1}
+ \|\varrho^{|s -\theta +1|}\widehat{\nabla g}\|_{\bs\ell_1}\right)\|\varrho^{s}\widehat w\|_{\bs\ell_2}.
\end{align*}

By the Schwartz inequality for any $\tilde\sigma>|s| + n/2$ we have,
\begin{align}
\|\varrho^{|s|}\widehat{\nabla g}\|_{\bs\ell_1}=\sum_{\bs\xi \in\Z^n} \varrho^{|s|}(\bs\xi)|\widehat{\nabla g}(\bs\xi)| 
\le\left[\sum_{\bs\xi \in\Z^n} \varrho^{2\tilde\sigma}(\bs\xi)|\widehat{\nabla g}(\bs\xi)|^2 \right]^{1/2} 
\left[\sum_{\bs\xi \in\Z^n} \varrho^{2|s|-2\tilde\sigma}(\bs\xi) \right]^{1/2}
\end{align}
Similarly, for any $\tilde\sigma>|s -\theta +1| + n/2$ we have,
\begin{align}
\|\varrho^{|s -\theta +1|}\widehat{\nabla g}\|_{\bs\ell_1}
\le\left[\sum_{\bs\xi \in\Z^n} \varrho^{2\tilde\sigma}(\bs\xi)|\widehat{\nabla g}(\bs\xi)|^2 \right]^{1/2} 
\left[\sum_{\bs\xi \in\Z^n} \varrho^{2|s -\theta +1|-2\tilde\sigma}(\bs\xi) \right]^{1/2}
\end{align}
Then for 
$\displaystyle\tilde\sigma>\tilde\sigma_0:=\max\{|s| + \frac{n}{2},|s -\theta +1| + \frac{n}{2}\}
=\left|s-\frac{\theta -1}{2}\right|+\left|\frac{\theta -1}{2}\right|+\frac{n}{2}$, 
we obtain
\begin{align*}
\|\Lambda_\#^\theta (gw)-g\Lambda_\#^\theta w\|_{ H_\#^{s -\theta +1}}
&\le \bar C_{s,\theta,\tilde\sigma}\|{\nabla g}\|_{H_\#^{\tilde \sigma}}\|\widehat w\|_{H_\#^s}
\le  \bar C_{s,\theta,\tilde\sigma} |{g}|_{H_\#^{\tilde \sigma+1}}\|\widehat w\|_{H_\#^s},
\end{align*}
where
\begin{align*}
\bar C_{s,\theta,\tilde\sigma}:=  \frac{2^{|s|/2}}{2\pi}|\theta |\left[\sum_{\bs\xi \in\Z^n} \varrho^{2\tilde\sigma_0-n-2\tilde\sigma}(\bs\xi) \right]^{1/2}.
\end{align*}
and \eqref{eq:mik13} is taken into account.
\end{proof}

\subsection{Spectrum of the periodic Bessel potential operator}\label{BPOS}
In this section we assume that vector-functions/distributions $\mathbf u$ are generally complex-valued and the Sobolev spaces $\dot{\mathbf H}_{\#\sigma}^{s}$ are complex.
Let us recall the definition 
\begin{align}\label{E3.14sigma}
 \left(\Lambda_\#^{r}\,{\mathbf u}\right)(\mathbf x)
:=\sum_{\bs\xi\in\dot\Z^n}\varrho(\bs\xi)^{r}\widehat {\mathbf u}(\bs\xi)e^{2\pi i \mathbf x\cdot\bs\xi}\quad 
\forall\, \mathbf u\in  \dot{\mathbf H}_{\#\sigma}^{s},\ s,r \in\R.
\end{align}
of  the continuous periodic Bessel potential operator 
$\Lambda_\#^r: \dot{\mathbf H}_{\#\sigma}^s\to \dot{\mathbf H}_{\#\sigma}^{s-r}$, $r\in\R$,
see \eqref{E3.14}, \eqref{E3.15}, \eqref{3.23}.

The following assertion is given in \cite[Theorem 4 and Remark 2]{Mikhailov2024}.
\begin{theorem}\label{TE.1r}
Let $r\in\R$, $r\ne 0$.

(i) Then the operator $\Lambda^r_\#$ in  $\dot{\mathbf H}_{\#\sigma}^{0}$  possesses  a (non-strictly) monotone sequence of real eigenvalues $\lambda^{(r)}_j=\lambda^{r}_j$ and a real orthonormal sequence of associated eigenfunctions $\mathbf w_j$ such that
\begin{align}
\label{S.2r}
&\Lambda^r_\# \mathbf w_j=\lambda^{r}_j \mathbf w_j,\  j \geq 1,\ 
\lambda_j>0,\\  
&\lambda_j \rightarrow+\infty,\  j \rightarrow+\infty;
\\
&\mathbf{w}_j \in \dot{\mathbf C}_{\#\sigma}^{\infty}, \quad
(\mathbf w_j, \mathbf w_k)_{\dot{\mathbf H}_{\#\sigma}^{0}}
=\delta_{j k} \quad
\forall\, j, k >0.
\end{align}

(ii) Moreover, the sequence $\{\mathbf{w}_j\}$ is an orthonormal basis in  $\dot{\mathbf H}_{\#\sigma}^{0}$, that is 
\begin{align}\label{S.9ru2} 
\mathbf u=\sum_{i=1}^\infty \langle \mathbf u, \mathbf w_j\rangle_\T \mathbf w_j
\end{align}
where the series converges in $\dot{\mathbf H}_{\#\sigma}^{0}$ for any $\mathbf u\in \dot{\mathbf H}_{\#\sigma}^{0}$.

(iii) In addition, the sequence $\{\mathbf{w}_j\}$ is also an orthogonal basis in $\dot{\mathbf H}_{\#\sigma}^{r}$ with
$$
(\mathbf w_j, \mathbf w_k)_{\dot{\mathbf H}_{\#\sigma}^{r}}
=\lambda_j^{r}\lambda_k^{r}\delta_{jk}\quad 
\forall\, j, k >0.
$$
and  for any $\mathbf u\in \dot{\mathbf H}_{\#\sigma}^{r}$ series \eqref{S.9ru2}  converges also in $\dot{\mathbf H}_{\#\sigma}^{r}$, that is, the sequence of partial sums 
\begin{align}\label{m-Projector}
P_m\mathbf u:=\sum_{j=1}^m \langle\mathbf u,\mathbf w_j\rangle_\T \mathbf w_j
\end{align}
converges to $\mathbf u$ in $\dot{\mathbf H}_{\#\sigma}^{r}$ as $m\to \infty$.
The operator $P_m$ defined by \eqref{m-Projector} is for any $r\in\R$ the orthogonal projector operator from ${\mathbf H}^{r}_{\#}$ to ${\rm Span}\{\mathbf w_1,\ldots,\mathbf w_m\}$. 
\end{theorem}

\subsection{Isomorphism of divergence and gradient operators in periodic spaces}
The following assertion proved in \cite[Lemma 2]{Mikhailov2024} provides for arbitrary $s\in\R$ and dimension $n\ge 2$ the periodic version of Bogovskii/deRham--type results well known for non-periodic domains and particular values of $s$, see, e.g., \cite{Bogovskii1979},  \cite{AmCiMa2015} and references therein.
\begin{lemma}\label{div-grad-is}
Let $s\in\R$ and $n\ge 2$. 
The following operators are isomorphisms,
\begin{align}
&\div\, :  \dot{\mathbf H}_{\# g}^{s+1}\to  \dot{H}_{\#}^{s},
\label{2.37}
\\
&{\rm grad}\, :  \dot{H}_{\#}^{s}\to  \dot{\mathbf H}_{\# g}^{s-1}.
\label{2.38}
\end{align}
\end{lemma}

\subsection{Some functional analysis results}
Let us provide the Sobolev embedding theorem that can be considered, e.g., as a particular case of \cite[Section 2.2.4, Corollary 2]{Runst-Sickel1996} adapted to the periodic spaces.
\begin{theorem}\label{SET}
Let $n\in\N$ be the dimension, $q_0<q_1<\infty$ and $q_1\ge 1$. The periodic Bessel-potential space $H^s_{\#q}$ is continuously embedded in $L_{\#q_1}$ if and only if \ $\displaystyle \frac{s}{n}\ge\frac{1}{q_0}-\frac{1}{q_1}.$
\end{theorem} 

The following version of the Sobolev interpolation inequality without a multiplicative constant, generalised also to  any real (including negative) smoothness indices, on periodic Bessel-potential spaces was given in \cite[Theorem 5]{Mikhailov2024}.
\begin{theorem}\label{T-SII2}
Let $s$, $s_1$, $s_2$, $\theta_1$, $\theta_2$ be real numbers such that $0\le\theta_1,\theta_2\le 1$,  $\theta_1+\theta_2=1$ and 
$s=\theta_1 s_1+\theta_2 s_2$.
Then for any $g\in H_\#^{s_1}\cap H_\#^{s_2}$,
\begin{align}\label{SII2}
\|g\|_{H_\#^s}\le\|g\|_{H_\#^{s_1}}^{\theta_1}\|g\|_{H_\#^{s_2}}^{\theta_2}.
\end{align}
\end{theorem}

Theorem 3.1 and Remark 3.2 in Chapter 1 of \cite{Lions-Magenes1}  imply the following assertion.
\begin{theorem}\label{LM-T3.1}
Let $X$ and $Y$ be separable Hilbert spaces and $X\subset Y$ with continuous injection.
Let $u\in W^1(0,T;X, Y)$.
Than $u$ almost everywhere on $[0,T]$ equals to a function   
$\tilde u\in \mathcal C^0([0,T];Z)$, where $Z=[X,Y]_{1/2}$ is the intermediate space.
Moreover, the trace $u(0)\in Z$ is well defined as the corresponding value of $\tilde u\in \mathcal C^0([0,T];Z)$ at $t=0$.
\end{theorem}

The following assertion was proved in \cite[Lemma 4]{Mikhailov2024}.
\begin{lemma}\label{L4.9}
Let $s,s'\in\R$, $s'\le s$ and $u\in W^1(0,T;H^s_\#,H^{s'}_\#)$ be real-valued. 

(i) Then
\begin{align}\label{E3.69}
\partial_t\|u\|^2_{H^{(s+s')/2}_\#}= 2\langle \Lambda_\#^{s'}u',\Lambda_\#^{s}u\rangle_{\T} = 2\langle \Lambda_\#^{s'+s}u',u\rangle_{\T} 
\end{align}
for a.e. $t\in(0,T)$ and also in the distribution sense on $t\in (0,T)$.

(ii) Moreover, for any real-valued $v\in W^1(0,T;{  H}_{\#}^{-s'}, {  H}_{\#}^{-s})$ and $t\in(0,T]$,
\begin{align}
\label{E3.30}
\int_0^t\left[\langle {  u}' (\tau),{  v} (\tau)\rangle _{\T }
+\langle {  u} (\tau),{  v}' (\tau)\rangle _{\T }\right]d\tau
=\left\langle {  u}(t),{  v}(t)\right\rangle_{\T }
-\left\langle {  u}(0),{  v}(0)\right\rangle _{\T }.
\end{align}
\end{lemma} 

\subsection{Gronwall's inequalities.}
Gronwall's inequality is well known and can be found, e.g., in \cite[Appendix B.2.j]{Evans1998}, \cite[Lemma A.24]{RRS2016}. 
Here we provide its slightly more general version valid also for arbitrary-sign coefficients.
\begin{lemma}\label{GL}
Let $\eta: [0, T] \to \R$ be an absolutely continuous function that satisfies the differential inequality
\begin{align}
\eta'(t)\le\phi(t)\eta(t)+\psi(t) \label{E15}\quad \text{for a.e. } t\in[0,T],
\end{align}
where $\phi$ and $\psi$ are real integrable functions. 

(a) Then
\begin{align}
\eta(t)\le e^{\int_0^t\phi(r)dr}[\eta(0)+\int_0^te^{-\int_0^s\phi(\rho)d\rho}\psi(s)ds]\quad\forall\, t\in[0,T].
\label{E16}
\end{align}

(b) Moreover, for non-negative $\phi$ and $\psi$ \eqref{E16} implies
\begin{align}
\eta(t)\le e^{\int_0^t\phi(r)dr}[\eta(0)+\int_0^t\psi(s)ds]\quad\forall\, t\in[0,T].
\label{E17}
\end{align}

(c) In particular, if $\eta$ is non-negative, while $\psi\equiv 0$ on $[0,T]$ and $\eta(0)=0$, then $\eta\equiv 0$ on $[0,T]$.
\end{lemma}
\begin{proof}
Multiplying \eqref{E15} by 
$$\displaystyle a(t):=e^{-\int_0^t\phi(r)dr}>0,$$ 
we obtain 
$$
\frac{d}{dt}[a(t)\eta(t)]\le a(t)\psi(t).
$$ 
Integration gives
$$a(t)\eta(t)\le a(0)\eta(0)+\int_0^t a(s)\psi(s)ds.$$ 
Dividing by $a(t)$, we arrive at 
\begin{align*}
\eta(t)\le \frac{1}{a(t)}[\eta(0)+\int_0^ta(s)\psi(s)ds]\quad\forall\, t\in[0,T]
\end{align*}
giving \eqref{E16} and thus proving item (a).
Items (b) and (c) follow from \eqref{E16}.
\end{proof}

Let us slightly generalise and give an alternative proof of 
\cite[Lemma 10.3]{RRS2016}.
\begin{lemma}\label{RRS2016-L10.3alt}
Let $\eta: [0, T] \to [0,\infty)$ be an absolutely continuous function  that satisfies the differential inequality
\begin{align}
\eta'(t) +by(t) \le cy(t)\eta(t)+\psi(t),\quad \text{for a.e. } t\in[0,T];\quad \eta(0)=\eta_0, \label{E15a-alpha}
\end{align}
where $\psi, y\ge0$ are integrable real functions, while $b,c>0$  and $\eta_0\ge 0$ are real constants. 

If 
\begin{align}\label{D-def-alpha}
D:=\eta_0+\int_0^T\psi(\tau)d\tau<\frac{b}{ec}
\end{align}
 then 
\begin{align}\label{res-ine}
\sup_{0\le\tau\le T} \eta(\tau)<{D}{e}<\frac{b}{c}\quad\text{and}\quad \int_0^Ty(\tau)d\tau<\frac{1}{c}.
\end{align}
\end{lemma}
\begin{proof}
By Lemma \ref{GL}(a), inequality \eqref{E15a-alpha} and condition \eqref{D-def-alpha} lead to
\begin{align}\label{D6alt}
a(t)\eta(t)+b\int_0^t a(s)y(s)ds\le \eta_0+\int_0^t a(s)\psi(s)ds\le D,
\end{align} 
where $a(s):=e^{-cY(s)}>0$ and $Y(s):=\int_0^s y(\tau)d\tau$.
Inequality \eqref{D6alt} implies
\begin{align}\label{D6alt2}
be^{-cY(t)}Y(t)\le b\int_0^t a(s)y(s)ds\le D<\frac{b}{ec}\quad\forall\, t\in[0,T].
\end{align} 
Let us consider the function $f(Y):=e^{-cY}Y$ on the interval $0\le Y<\infty$.
One can elementary obtain that $\max_{0\le Y<\infty}f(Y)$ 
is reached at $Y=1/c$ and equals to $1/(ec)$.
But due to \eqref{D6alt2} this maximum for $e^{-cY(t)}Y(t)$ is not reached for $t\in[0,T]$ and hence $Y(T)<1/c$, giving the second inequality in \eqref{res-ine}.

Further, \eqref{D6alt} implies that
$$
\eta(t)\le \frac{D}{a(t)}=De^{cY(t)}<{D}{e}<\frac{b}{c}\quad\forall\, t\in[0,T],
$$
thus giving the first inequality in \eqref{res-ine}.
\end{proof}

Let us give a generalisation of \cite[Lemma 10.3]{RRS2016}
and of Lemma \ref{RRS2016-L10.3alt}.
\begin{lemma}
\label{RRS2016-L10.3phi}
Let $\eta: [0, T] \to [0,\infty)$ be an absolutely continuous function  that satisfies the differential inequality
\begin{align}
\eta'(t) +by(t) \le \left[cy(t)+\phi(t)\right]\eta(t)+\psi(t),\quad  \text{for a.e. }t\in[0,T];\quad \eta(0)=\eta_0, \label{E15a-phi}
\end{align}
where $\phi,\psi, y\ge0$ are integrable real functions, while $b,c>0$  and $\eta_0\ge 0$ are real constants. 

If 
\begin{align}\label{D-def-phi}
D:=\eta_0+\int_0^T e^{-\Phi(\tau)} \psi(\tau)d\tau<\frac{b}{c}e^{-1-\Phi(T)},
\end{align}
where $\Phi(s):=\int_0^s \phi(\tau)d\tau$, then 
\begin{align}\label{res-ine-phi}
\sup_{0\le\tau\le T} \eta(\tau)
<\frac{b}{c}\quad\text{and}\quad \int_0^Ty(\tau)d\tau<\frac{1}{c}.
\end{align}
\end{lemma}
\begin{proof}
By Lemma \ref{GL}(a), inequality \eqref{E15a-phi} and condition \eqref{D-def-phi} lead to
\begin{align}\label{D6phi}
a(t)\eta(t)+b\int_0^t a(s)y(s)ds\le \eta_0+\int_0^t a(s)\psi(s)ds\le D,
\end{align} 
where $a(s):=e^{-cY(s)-\Phi(s)}>0$ and $Y(s):=\int_0^s y(\tau)d\tau$.
Inequality \eqref{D6phi} implies
\begin{align}\label{D6phi2}
be^{-\Phi(T)}e^{-cY(t)}Y(t)\le b\int_0^t a(s)y(s)ds\le D<\frac{b}{c}e^{-1-\Phi(T)}\quad\forall\, t\in[0,T].
\end{align} 
Let us consider the function $f(Y):=e^{-cY}Y$ on the interval $0\le Y<\infty$.
One can elementary obtain that $\max_{0\le Y<\infty}f(Y)$ 
is reached at $Y=1/c$ and equals to $1/(ec)$.
But due to \eqref{D6phi2} this maximum of $e^{-cY(t)}Y(t)$ is not reached for $t\in[0,T]$ and hence $Y(T)<1/c$, giving the second inequality {\mg in \eqref{res-ine-phi}.}

Further, \eqref{D6phi} implies that
$$
\eta(t)\le \frac{D}{a(t)}=De^{cY(t)+\Phi(t)}<{D}e^{1+\Phi(T)}<\frac{b}{c}\quad\forall\, t\in[0,T],
$$
thus giving the first inequality \mg in \eqref{res-ine-phi}.
\end{proof}

Let us give a version of integral Gronwall's inequality implied, e.g., by Theorem 1.3 and Remark 1.5 in \cite{Bainov-Semeonov1992}.
\begin{lemma}\label{IGTL}
Let $u$, $b$ and $a$ be measurable functions in $J=[\alpha, \beta]$,  such that $bu, ba \in L_1(J)$. Suppose that $b(t)$ is non-negative a.e. on $J$.
Suppose
\begin{align*}
u(t) \leq a(t)+\int_\alpha^t b(s) u(s) d s, \quad \  \text{a.e. }t \in J.
\end{align*}
Then
\begin{align*}
u(t) \leq a(t)+\int_\alpha^t a(s) b(s)\exp\left(\int_s^t b(\tau) d \tau\right) d s, \quad \  \text{a.e. }t \in J.
\end{align*}
\end{lemma}

\end{document}